\documentclass[12pt,twoside]{amsart}
\usepackage[latin1]{inputenc}
\usepackage{amsmath, amsthm, amscd, amsfonts, amssymb, graphicx}
\usepackage[bookmarksnumbered, plainpages]{hyperref}

\newtheorem{thm}{Theorem}[section]
\newtheorem{cor}[thm]{Corollary}

\numberwithin{equation}{section}

\begin{document}

\title{{\bf $SL(2,{\bf Z})$ modular forms and Witten genus in odd dimensions}}

\author{ Jianyun Guan \ \ Yong Wang* \ \ Haiming Liu\\
 }

\date{}

\thanks{{\scriptsize
\hskip -0.4 true cm \textit{2010 Mathematics Subject Classification:}
58C20; 57R20; 53C80.
\newline \textit{* Corresponding author.}}}

\maketitle

\begin{abstract}
 By some $SL(2,{\bf Z})$ modular forms introduced in \cite{Li2} and \cite{CHZ}, we construct some modular forms over $SL_2({\bf{Z}})$ and some modular forms over $\Gamma^0(2)$ and $\Gamma_0(2)$ in odd dimensions. In parallel, we obtain some new cancellation formulas for odd dimensional spin manifolds and odd dimensional spin$^c$ manifolds respectively. As corollaries, we get some divisibility results of index of the Toeplitz operators on spin manifolds and spin$^c$ manifolds .
\end{abstract}

\vskip 0.2 true cm

\section{ Introduction}
In \cite{AW}, Alvarez-Gaum\'{e} and Witten discovered a formula that represents the beautiful relationship between the top
  components of the Hirzebruch $\widehat{L}$-form and $\widehat{A}$-form of a $12$-dimensional smooth Riemannian
  manifold. This formula is called the "miracle cancellation" formula for gravitational anomalies.
In \cite{Li1}, Liu established higher dimensional "miraculous cancellation"
  formulas for $(8k+4)$-dimensional Riemannian manifolds by
  developing modular invariance properties of characteristic forms. These formulas could be used to deduce some divisibility results.
  In \cite{HZ1}, \cite{HZ2}, Han and Zhang established some more general cancellation formulas that involve a complex line bundle over each $(8k + 4)$-dimensional smooth Riemannian manifold.  This formula was applied to ${\rm spin^c}$ manifolds, then an analytic Ochanine congruence formula was derived. In \cite{Wa}, Wang obtained some new anomaly cancellation formulas by studying the modular invariance of some characteristic forms. This formula was applied to spin manifolds and ${\rm spin^c}$ manifolds, then  some results on divisibilities on spin manifolds and ${\rm spin^c}$ manifolds were derived. Moreover, Han, Liu
and Zhang using the Eisenstein series, a more general cancellation formula was derived \cite{HLZ1}.
And in \cite{HLZ2}, the authors showed that both of the Green-Schwarz anomaly factorization formula for the gauge group $E_8\times E_8$ and the Horava-Witten anomaly factorization formula for the gauge group $E_8$ could be derived through modular forms of weight $14$. This answered a question of J. H. Schwarz. They also established generalizations of these decomposition formulas and obtained a new Horava-Witten type decomposition formula on $12$-dimensional manifolds. In \cite{HHLZ}, Han, Huang, Liu and Zhang introduced a modular form of weight $14$ over $SL(2,{\bf Z})$ and a modular form of weight $10$ over $SL(2,{\bf Z})$ and they got some interesting
anomaly cancellation formulas on $12$-dimensional manifolds. In \cite{Wa1} and \cite{Wa2}, Wang obtained some new anomaly cancellation formulas by studying some $SL(2,{\bf Z})$ modular forms. Some divisibility results of the index of the twisted Dirac operator are obtained.

 In \cite{Li2}, Liu introduced a modular form of a 4k-dimensional spin manifold with a weight of 2k. In \cite{CHZ}, Chen, Han and Zhang defined an integral modular form of weight $2k$ for a $4k$-dimensional $spin^c$ manifold and an integral modular form of weight $2k$ for a $4k+2$-dimensional $spin^c$ manifold. A natural question is whether we can get some interesting cancellation formulas and more results on divisibilities in odd dimensions. In \cite{LW1} and \cite{LW2}, Liu and Wang go through studying modular invariance properties of some characteristic forms to get some new anomaly cancellation formulas on $(4k-1)$ dimensional manifolds. And they derive some results on divisibilities on $(4k-1)$ dimensional spin manifolds and congruences on $(4k-1)$ dimensional ${\rm spin^c}$ manifolds. Inspired by this, we introduce some modular forms of weight $2k$ over $SL_{2}(\bf Z)$ and some modular forms of weight $2k$ over $\Gamma^0(2)$ and $\Gamma_0(2)$ in odd dimensions respectively through the $SL(2,{\bf Z})$ modular forms introduced in \cite{Li2} and \cite{CHZ}. In parallel, we derive some new anomaly cancellation formulas and some divisibility results over spin manifolds and ${\rm spin^c}$ manifolds in odd dimensions.

 The structure of this paper is briefly described below: In Section 2, we have introduce some definitions and basic concepts that we will use in the paper. In Section 3, we prove some generalized cancellation formulas over $SL_2(\bf Z)$ in odd dimensions. Finally, in section 4, we obtain some generalized cancellation formulas over $\Gamma^0(2)$ and $\Gamma_0(2)$ in odd dimensions.\\

\section{Characteristic Forms and Modular Forms}
\quad The purpose of this section is to review the necessary knowledge on
characteristic forms and modular forms that we are going to use.\\

 \noindent {\bf  2.1 characteristic forms }\\
 \indent Let $M$ be a Riemannian manifold.
 Let $\nabla^{ TM}$ be the associated Levi-Civita connection on $TM$
 and $R^{TM}=(\nabla^{TM})^2$ be the curvature of $\nabla^{ TM}$.
 Let $\widehat{A}(TM,\nabla^{ TM})$ and $\widehat{L}(TM,\nabla^{ TM})$
 be the Hirzebruch characteristic forms defined respectively by (cf. \cite{Zh})
\begin{equation}
   \widehat{A}(TM,\nabla^{ TM})={\rm
det}^{\frac{1}{2}}\left(\frac{\frac{\sqrt{-1}}{4\pi}R^{TM}}{{\rm
sinh}(\frac{\sqrt{-1}}{4\pi}R^{TM})}\right),
\end{equation}
 \begin{equation}
     \widehat{L}(TM,\nabla^{ TM})={\rm
 det}^{\frac{1}{2}}\left(\frac{\frac{\sqrt{-1}}{2\pi}R^{TM}}{{\rm
 tanh}(\frac{\sqrt{-1}}{4\pi}R^{TM})}\right).
 \end{equation}
   Let $E$, $F$ be two Hermitian vector bundles over $M$ carrying
   Hermitian connection $\nabla^E,\nabla^F$ respectively. Let
   $R^E=(\nabla^E)^2$ (resp. $R^F=(\nabla^F)^2$) be the curvature of
   $\nabla^E$ (resp. $\nabla^F$). If we set the formal difference
   $G=E-F$, then $G$ carries an induced Hermitian connection
   $\nabla^G$ in an obvious sense. We define the associated Chern
   character form as
   \begin{equation}
       {\rm ch}(G,\nabla^G)={\rm tr}\left[{\rm
   exp}(\frac{\sqrt{-1}}{2\pi}R^E)\right]-{\rm tr}\left[{\rm
   exp}(\frac{\sqrt{-1}}{2\pi}R^F)\right].
   \end{equation}
   For any complex number $t$, let
   $$\wedge_t(E)={\bf C}|_M+tE+t^2\wedge^2(E)+\cdots,~S_t(E)={\bf
   C}|_M+tE+t^2S^2(E)+\cdots$$
   denote respectively the total exterior and symmetric powers of
   $E$, which live in $K(M)[[t]].$ The following relations between
   these operations hold,
   \begin{equation}
       S_t(E)=\frac{1}{\wedge_{-t}(E)},~\wedge_t(E-F)=\frac{\wedge_t(E)}{\wedge_t(F)}.
   \end{equation}
   Moreover, if $\{\omega_i\},\{\omega_j'\}$ are formal Chern roots
   for Hermitian vector bundles $E,F$ respectively, then
   \begin{equation}
       {\rm ch}(\wedge_t(E))=\prod_i(1+e^{\omega_i}t)
   \end{equation}
   Then we have the following formulas for Chern character forms,
   \begin{equation}
       {\rm ch}(S_t(E))=\frac{1}{\prod_i(1-e^{\omega_i}t)},~
{\rm ch}(\wedge_t(E-F))=\frac{\prod_i(1+e^{\omega_i}t)}{\prod_j(1+e^{\omega_j'}t)}.
   \end{equation}
\indent If $W$ is a real Euclidean vector bundle over $M$ carrying a
Euclidean connection $\nabla^W$, then its complexification $W_{\bf
C}=W\otimes {\bf C}$ is a complex vector bundle over $M$ carrying a
canonical induced Hermitian metric from that of $W$, as well as a
Hermitian connection $\nabla^{W_{\bf C}}$ induced from $\nabla^W$.
If $E$ is a vector bundle (complex or real) over $M$, set
$\widetilde{E}=E-{\rm dim}E$ in $K(M)$ or $KO(M)$.\\

\noindent{\bf 2.2 Some properties about the Jacobi theta functions
and modular forms}\\
   \indent We first recall the four Jacobi theta functions are
   defined as follows( cf. \cite{Ch}):
   \begin{equation}
      \theta(v,\tau)=2q^{\frac{1}{8}}{\rm sin}(\pi
   v)\prod_{j=1}^{\infty}[(1-q^j)(1-e^{2\pi\sqrt{-1}v}q^j)(1-e^{-2\pi\sqrt{-1}v}q^j)],
   \end{equation}
\begin{equation}
    \theta_1(v,\tau)=2q^{\frac{1}{8}}{\rm cos}(\pi
   v)\prod_{j=1}^{\infty}[(1-q^j)(1+e^{2\pi\sqrt{-1}v}q^j)(1+e^{-2\pi\sqrt{-1}v}q^j)],
\end{equation}
\begin{equation}
    \theta_2(v,\tau)=\prod_{j=1}^{\infty}[(1-q^j)(1-e^{2\pi\sqrt{-1}v}q^{j-\frac{1}{2}})
(1-e^{-2\pi\sqrt{-1}v}q^{j-\frac{1}{2}})],
\end{equation}
\begin{equation}
   \theta_3(v,\tau)=\prod_{j=1}^{\infty}[(1-q^j)(1+e^{2\pi\sqrt{-1}v}q^{j-\frac{1}{2}})
(1+e^{-2\pi\sqrt{-1}v}q^{j-\frac{1}{2}})],
\end{equation}
 \noindent
where $q=e^{2\pi\sqrt{-1}\tau}$ with $\tau\in\textbf{H}$, the upper
half complex plane. Let
\begin{equation}
    \theta'(0,\tau)=\frac{\partial\theta(v,\tau)}{\partial v}|_{v=0}.
\end{equation} \noindent Then the following Jacobi identity
(cf. \cite{Ch}) holds,
\begin{equation}   \theta'(0,\tau)=\pi\theta_1(0,\tau)\theta_2(0,\tau)\theta_3(0,\tau).
\end{equation}
\noindent Denote $$SL_2({\bf Z})=\left\{\left(\begin{array}{cc}
\ a & b  \\
 c  & d
\end{array}\right)\mid a,b,c,d \in {\bf Z},~ad-bc=1\right\}$$ the
modular group. Let $S=\left(\begin{array}{cc}
\ 0 & -1  \\
 1  & 0
\end{array}\right),~T=\left(\begin{array}{cc}
\ 1 &  1 \\
 0  & 1
\end{array}\right)$ be the two generators of $SL_2(\bf{Z})$. They
act on $\textbf{H}$ by $S\tau=-\frac{1}{\tau},~T\tau=\tau+1$. One
has the following transformation laws of theta functions under the
actions of $S$ and $T$ (cf. \cite{Ch}):
\begin{equation}
    \theta(v,\tau+1)=e^{\frac{\pi\sqrt{-1}}{4}}\theta(v,\tau),~~\theta(v,-\frac{1}{\tau})
=\frac{1}{\sqrt{-1}}\left(\frac{\tau}{\sqrt{-1}}\right)^{\frac{1}{2}}e^{\pi\sqrt{-1}\tau
v^2}\theta(\tau v,\tau);
\end{equation}
 \begin{equation}
     \theta_1(v,\tau+1)=e^{\frac{\pi\sqrt{-1}}{4}}\theta_1(v,\tau),~~\theta_1(v,-\frac{1}{\tau})
=\left(\frac{\tau}{\sqrt{-1}}\right)^{\frac{1}{2}}e^{\pi\sqrt{-1}\tau
v^2}\theta_2(\tau v,\tau);
 \end{equation}
\begin{equation}   \theta_2(v,\tau+1)=\theta_3(v,\tau),~~\theta_2(v,-\frac{1}{\tau})
=\left(\frac{\tau}{\sqrt{-1}}\right)^{\frac{1}{2}}e^{\pi\sqrt{-1}\tau
v^2}\theta_1(\tau v,\tau);
\end{equation}
\begin{equation}    \theta_3(v,\tau+1)=\theta_2(v,\tau),~~\theta_3(v,-\frac{1}{\tau})
=\left(\frac{\tau}{\sqrt{-1}}\right)^{\frac{1}{2}}e^{\pi\sqrt{-1}\tau
v^2}\theta_3(\tau v,\tau).
\end{equation}
\begin{equation}
    \theta'(v,\tau+1)=e^{\frac{\pi\sqrt{-1}}{4}}\theta'(v,\tau),~~
 \theta'(0,-\frac{1}{\tau})=\frac{1}{\sqrt{-1}}\left(\frac{\tau}{\sqrt{-1}}\right)^{\frac{1}{2}}
\tau\theta'(0,\tau),
\end{equation}

\noindent
 \noindent {\bf Definition 2.1} A modular form over $\Gamma$, a
 subgroup of $SL_2({\bf Z})$, is a holomorphic function $f(\tau)$ on
 $\textbf{H}$ such that
 \begin{equation}
    f(g\tau):=f\left(\frac{a\tau+b}{c\tau+d}\right)=\chi(g)(c\tau+d)^kf(\tau),
 ~~\forall g=\left(\begin{array}{cc}
\ a & b  \\
 c & d
\end{array}\right)\in\Gamma,
 \end{equation}
\noindent where $\chi:\Gamma\rightarrow {\bf C}^{\star}$ is a
character of $\Gamma$. $k$ is called the weight of $f$.\\
Let $$\Gamma_0(2)=\left\{\left(\begin{array}{cc}
\ a & b  \\
 c  & d
\end{array}\right)\in SL_2({\bf Z})\mid c\equiv 0~({\rm
mod}~2)\right\},$$
$$\Gamma^0(2)=\left\{\left(\begin{array}{cc}
\ a & b  \\
 c  & d
\end{array}\right)\in SL_2({\bf Z})\mid b\equiv 0~({\rm
mod}~2)\right\},$$
be the two modular subgroups of $SL_2({\bf Z})$.
It is known that the generators of $\Gamma_0(2)$ are $T,~ST^2ST$,
the generators of $\Gamma^0(2)$ are $STS,~T^2STS$ (cf. \cite{Ch}).\\
\indent If $\Gamma$ is a modular subgroup, let ${\mathcal{M}}_{{\bf
R}}(\Gamma)$ denote the ring of modular forms over $\Gamma$ with
real Fourier coefficients. Writing $\theta_j=\theta_j(0,\tau),~1\leq
j\leq 3,$ we introduce six explicit modular forms (cf. \cite{Li1}),
$$\delta_1(\tau)=\frac{1}{8}(\theta_2^4+\theta_3^4),~~\varepsilon_1(\tau)=\frac{1}{16}\theta_2^4\theta_3^4,$$
$$\delta_2(\tau)=-\frac{1}{8}(\theta_1^4+\theta_3^4),~~\varepsilon_2(\tau)=\frac{1}{16}\theta_1^4\theta_3^4,$$
\noindent They have the following Fourier expansions in
$q^{\frac{1}{2}}$:
$$\delta_1(\tau)=\frac{1}{4}+6q+\cdots,~~\varepsilon_1(\tau)=\frac{1}{16}-q+\cdots,$$
$$\delta_2(\tau)=-\frac{1}{8}-3q^{\frac{1}{2}}+\cdots,~~\varepsilon_2(\tau)=q^{\frac{1}{2}}+\cdots,$$
\noindent where the $"\cdots"$ terms are the higher degree terms,
all of which have integral coefficients. They also satisfy the
transformation laws,
\begin{equation}
    \delta_2(-\frac{1}{\tau})=\tau^2\delta_1(\tau),~~~~~~\varepsilon_2(-\frac{1}{\tau})
=\tau^4\varepsilon_1(\tau),
\end{equation}
\noindent {\bf Lemma 2.2} (\cite{Li1}) {\it $\delta_1(\tau)$ (resp.
$\varepsilon_1(\tau)$) is a modular form of weight $2$ (resp. $4$)
over $\Gamma_0(2)$, $\delta_2(\tau)$ (resp. $\varepsilon_2(\tau)$)
is a modular form of weight $2$ (resp. $4$) over $\Gamma^0(2)$,
while  $\delta_3(\tau)$ (resp. $\varepsilon_3(\tau)$) is a modular
form of weight $2$ (resp. $4$) over $\Gamma_\theta(2)$ and moreover
${\mathcal{M}}_{{\bf R}}(\Gamma^0(2))={\bf
R}[\delta_2(\tau),\varepsilon_2(\tau)]$.}

\section{Some modular forms and Witten genus over $SL_2({\bf Z})$ in odd dimensions}
 Let $M$ be a $(4k-1)$-dimensional spin manifold and $\triangle(M)$ be the spinor bundle. Let $\widetilde{T_{\mathbf{C}}M}=T_{\mathbf{C}}M-\dim M$.
 Set
 \begin{equation}
   \Theta_1(T_{\mathbf{C}}M)=
   \bigotimes _{n=1}^{\infty}S_{q^n}(\widetilde{T_{\mathbf{C}}M})\otimes
\bigotimes _{m=1}^{\infty}\wedge_{q^m}(\widetilde{T_{\mathbf{C}}M})
,\end{equation}
\begin{equation}
\Theta_2(T_{\mathbf{C}}M)=\bigotimes _{n=1}^{\infty}S_{q^n}(\widetilde{T_{\mathbf{C}}M})\otimes
\bigotimes _{m=1}^{\infty}\wedge_{-q^{m-\frac{1}{2}}}(\widetilde{T_{\mathbf{C}}M}),
\end{equation}
\begin{equation}
\Theta_3(T_{\mathbf{C}}M)=\bigotimes _{n=1}^{\infty}S_{q^n}(\widetilde{T_{\mathbf{C}}M})\otimes
\bigotimes _{m=1}^{\infty}\wedge_{q^{m-\frac{1}{2}}}(\widetilde{T_{\mathbf{C}}M}).
\end{equation}

We recall the odd Chern character  of a smooth map g from M to the general linear group $GL(N,\mathbf{C})$
with $N$ a  positive integer (see [14]). Let $d$ denote a trivial connection on $\mathbf{C}^{N}|_{M}$. We will denote by $c_g(M,[g])$ the cohomology class associated to the closed $n$-form
\begin{equation}
  c_n(\mathbf{C}^{N}|_M,g,d)=\left(\frac{1}{2\pi\sqrt{-1}}\right)^{\frac{(n+1)}{2}}\mathrm{Tr}[(g^{-1}dg)^n].
\end{equation}
The odd Chern character form ${\rm ch}(\mathbf{C}^{N}|_M,g,d)$ associated to $g$ and $d$ by definition is
\begin{equation}
  {\rm ch}(\mathbf{C}^{N}|_M,g,d)=\sum^{\infty}_{n=1}\frac{n!}{(2n+1)!}c_{2n+1}((\mathbf{C}^{N}|_M,g,d)).
\end{equation}
Let the connection $\nabla_{u}$ on the trivial bundle $\mathbf{C}^{N}|_M$ defined by
\begin{equation}
  \nabla_u=(1-u)d+ug^{-1}\cdot d \cdot g,\ \ u\in[0,1].
\end{equation}
Then we have
\begin{equation}
  d{\rm ch}(\mathbf{C}^{N}|_M,g,d)={\rm ch}(\mathbf{C}^{N}|_M,d)-{\rm ch}(\mathbf{C}^{N}|_M,g^{-1}\cdot d\cdot g).
\end{equation}

Now let $g:M\to SO(N)$ and we assume that $N$ is even and large enough. Let $E$ denote the trivial real vector bundle of rank $N$ over $M$. We equip $E$ with the canonical trivial metric and trivial connection $d$. Set
$$\nabla_u=d+ug^{-1}dg,\ \ u\in[0,1].$$
Let $R_u$ be the curvature of $\nabla_u$, then
\begin{equation}
  R_u=(u^2-u)(g^{-1}dg)^2.
\end{equation}
We also consider the complexification of $E$ and $g$ extends to a unitary automorphism of $E_{\mathbf{C}}$. The connection $\nabla_u$ extends to a Hermitian connection on $E_{\mathbf{C}}$ with curvature still given by (3.6). Let $\Delta(E)$ be the spinor bundle of $E$, which is a trivial Hermitian
bundle of rank $2^{\frac{N}{2}}$. We assume that $g$ has a lift to the Spin group ${\rm Spin}(N):g^{\Delta}:M\to {\rm Spin}(N)$. So $g^{\Delta}$ can be viewed as an automorphism of $\Delta(E)$ preserving the Hermitian metric. We lift $d$ on $E$ to be a trivial Hermitian connection $d^{\Delta}$ on $\Delta(E)$, then
\begin{equation}
  \nabla_u^{\Delta}=(1-u)d^{\Delta}+u(g^{\Delta})^{-1}\cdot d^{\Delta} \cdot g^{\Delta},\ \ u\in[0,1]
\end{equation}
lifts $\nabla_u$ on $E$ to $\Delta(E)$. Let $Q_j(E),j=1,2,3$ be the virtual bundles defined as follows:
\begin{equation}
Q_1(E)=\triangle(E)\otimes
   \bigotimes _{n=1}^{\infty}\wedge_{q^n}(\widetilde{E_C});
\end{equation}
\begin{equation}
Q_2(E)=\bigotimes _{n=1}^{\infty}\wedge_{-q^{n-\frac{1}{2}}}(\widetilde{E_C});
\end{equation}
\begin{equation}
Q_3(E)=\bigotimes _{n=1}^{\infty}\wedge_{q^{n-\frac{1}{2}}}(\widetilde{E_C}).
\end{equation}
Let $g$ on $E$ have a lift $g^{Q(E)}$ on $Q(E)$ and $\nabla_u$ have a lift $\nabla^{Q(E)}_u$ on $Q(E)$. Following \cite{HY}, we defined ${\rm ch}(Q(E),g^{Q(E)},d,\tau)$ as following
\begin{equation}
{\rm ch}(Q(E),\nabla^{Q(E)}_0,\tau)-{\rm ch}(Q(E),\nabla^{Q(E)}_1,\tau)=-d{\rm ch}(Q(E),g^{Q(E)},d,\tau),
\end{equation}
where
$$Q(E)=Q_1(E)\otimes Q_2(E)\otimes Q_3(E),$$
and
\begin{equation}
{\rm ch}(Q(E),g^{Q(E)},d,\tau)=-\frac{2^{\frac{N}{2}}}{8\pi^2}\int^1_0{\rm Tr}\left[g^{-1}dg\left(A\right)\right]du,
\end{equation}
with
$$A=\frac{\theta'_1(R_u/(4\pi^2),\tau)}{\theta_1(R_u/(4\pi^2),\tau)}
+\frac{\theta'_2(R_u/(4\pi^2),\tau)}{\theta_2(R_u/(4\pi^2),\tau)}+\frac{\theta'_3(R_u/(4\pi^2),\tau)}
{\theta_3(R_u/(4\pi^2),\tau)}.$$
By Proposition 2.2 in \cite{HY}, we have if $c_3(E_C,g,d)=0$, then for any integer $r\geq 1.$ We have ${\rm ch}(Q(E),g^{Q(E)},d,\tau+1)^{(4r-1)}={\rm ch}(Q(E),g^{Q(E)},d,\tau)^{(4r-1)}$and ${\rm ch}(Q(E),g^{Q(E)},d,-\frac{1}{\tau})^{(4r-1)}=\tau^{2r}{\rm ch}(Q(E),g^{Q(E)},d,\tau)^{(4r-1)}$,so ${\rm ch}(Q(E),g^{Q(E)},d,\tau)^{(4r-1)}$ are modular forms of weight $2r$ over $SL_2(\bf{Z})$.
Let
\begin{equation}
\begin{split}
Q(\nabla^{TM},g,d,\tau)=&\{\widehat{A}(TM,\nabla^{TM}){\rm ch}([\triangle(M)\otimes \Theta_1(T_{C}M)+2^{2k}\Theta_2(T_{C}M)\\
&+2^{2k}\Theta_3(T_{C}M)]){\rm ch}(Q(E),g^{Q(E)},d,\tau)\}^{(4k-1)}.
\end{split}
\end{equation}
\begin{thm}
Let ${\rm dim}M=4k-1$. If $c_3(E,g,d)=0$, then for any integer $p,r\geq 1.$~$Q(\nabla^{TM},g,d,\tau)$ is a modular form over $SL_2({\bf Z})$ with the weight $2p+2r=2k$.
\end{thm}
\begin{proof}
Let
\begin{equation}
  \begin{split}
  Q(M,\tau)=\{\widehat{A}(TM,\nabla^{TM}){\rm ch}([\triangle(M)\otimes \Theta_1(T_{C}M)+2^{2k}\Theta_2(T_{C}M)+2^{2k}\Theta_3(T_{C}M)])\}^{(4p)},
  \end{split}
\end{equation}
then we have
\begin{equation}
  Q(\nabla^{TM},g,d,\tau)=Q(M,\tau)\cdot{\rm ch}(Q(E),g^{Q(E)},d,\tau)^{(4r-1)}.
\end{equation}
Let $\{\pm 2\pi\sqrt{-1}x_j\}~(1\leq j \leq 2k-1)$ be the formal Chern roots for $T_\mathbf{C}M$, then we have
 \begin{equation}{\widehat{A}(TM,\nabla^{TM})}{\rm ch}(\triangle(M)){\rm ch}(\Theta_1(T_\mathbf{C}M))=\prod_{j=1}^{2k-1}\frac{2 x_j\theta'(0,\tau)}{\theta(x_j,\tau)}\frac{\theta_1(x_j,\tau)}{\theta_1(0,\tau)},\end{equation}
\begin{equation}{\widehat{A}(TM,\nabla^{TM})}{\rm ch}(2^{2k}\Theta_2(T_\mathbf{C}M))=\prod_{j=1}^{2k-1}\frac{2 x_j\theta'(0,\tau)}{\theta(x_j,\tau)}\frac{\theta_2(x_j,\tau)}{\theta_2(0,\tau)},\end{equation}
 \begin{equation}{\widehat{A}(TM,\nabla^{TM})}{\rm ch}(2^{2k}\Theta_3(T_\mathbf{C}M))=\prod_{j=1}^{2k-1}\frac{2 x_j\theta'(0,\tau)}{\theta(x_j,\tau)}\frac{\theta_3(x_j,\tau)}{\theta_3(0,\tau)},\end{equation}
 So we have
\begin{equation}Q(M,\tau)=\left(\prod_{j=1}^{2k-1}\frac{2x_j\theta'(0,\tau)}{\theta(x_j,\tau)}
\left(\prod_{j=1}^{2k-1}\frac{\theta_1(x_j,\tau)}{\theta_1(0,\tau)}+\prod_{j=1}^{2k-1}\frac{\theta_2
(x_j,\tau)}{\theta_2(0,\tau)}+\prod_{j=1}^{2k-1}\frac{\theta_3(x_j,\tau)}{\theta_3(0,\tau)}
\right)\right)^{(4p)}.\end{equation}
By Proposition 2.2 in \cite{HY} and (3.16)-(3.21), we have $$Q(\nabla^{TM},g,d,\tau+1)=Q(M,\tau)\cdot{\rm ch}(Q(E),g^{Q(E)},d,\tau)^{(4r-1)}=Q(\nabla^{TM},g,d,\tau)$$ and $$Q(\nabla^{TM},g,d,-\frac{1}{\tau})=\tau^{2p}Q(M,\tau)\cdot\tau^{2r}{\rm ch}(Q(E),g^{Q(E)},d,\tau)^{(4r-1)}=\tau^{2k}Q(\nabla^{TM},g,d,\tau),$$ so
$Q(\nabla^{TM},g,d,\tau)$ is a modular form over $SL_2({\bf Z})$ with the weight $2p+2r=2k$.
\end{proof}
Set
\begin{equation}
\begin{split}
\left<Q(\nabla^{TM},g,d,\tau),[M]\right>=&-{\rm Ind}(T\otimes (\triangle(M)\otimes \Theta_1(T_{\mathbf{C}}M)+ 2^{2d}\Theta_2(T_{\mathbf{C}}M)\\
&+2^{2d}\Theta_3(T_{\mathbf{C}}M))\otimes(Q(E),g^{Q(E),d}))
\end{split}
\end{equation}
where ${\rm Ind}(T\otimes \cdots)$ denotes the index of the Toeplitz operator. Clearly,
$\Theta_1(T_{\mathbf{C}}M)\otimes Q(E)$,~$\Theta_2(T_{\mathbf{C}}M)\otimes Q(E)$, and  $\Theta_2(T_{\mathbf{C}}M)\otimes Q(E)$ admit formal Fourier expansion in $q^{\frac{1}{2}}$ as
$$\Theta_1(T_{\mathbf{C}}M)\otimes Q(E)=B^1_0(T_{\mathbf{C}}M,
E_{\mathbf{C}})+B^1_1(T_{\mathbf{C}}M,E_{\mathbf{C}})q+B^1_2(T_{\mathbf{C}}M,E_{\mathbf{C}})q^2+
\cdots,$$
$$\Theta_2(T_{\mathbf{C}}M)\otimes Q(E)=B^2_0(T_{\mathbf{C}}M,
E_{\mathbf{C}})+B^2_1(T_{\mathbf{C}}M,E_{\mathbf{C}})q^{\frac{1}{2}}+B^2_2(T_{\mathbf{C}}M,
E_{\mathbf{C}})q+\cdots,$$
$$\Theta_3(T_{\mathbf{C}}M)\otimes Q(E)=B^3_0(T_{\mathbf{C}}M,
E_{\mathbf{C}})+B^3_1(T_{\mathbf{C}}M,E_{\mathbf{C}})q^{\frac{1}{2}}+B^3_2(T_{\mathbf{C}}M,
E_{\mathbf{C}})q+\cdots,$$
where the $B^i_j$ are
elements in the semi-group formally generated by Hermitian vector
bundles over $M$. Moreover, they carry canonically induced Hermitian
connections. If $$B^i_j(T_{\mathbf{C}}M,E)=B^i_{j,1}(T_{\mathbf{C}}M)\otimes B^i_{j,2}(E),$$ we let
$$\widetilde{{\rm
ch}}(B^i_j(T_{\mathbf{C}}M,E))={\rm ch}( B^i_{j,1}(T_{\mathbf{C}}M)){\rm ch}(B^i_{j,2}(E),g^{B^i_{j,2}(E)},d).$$
If $\omega$ is a differential form over $M$, we denote
$\omega^{(4k-1)}$ its top degree component. Our main results include the following theorem.
\begin{thm}
  When ${\rm dim}M=7$ and $c_3(E,g,d)=0$, we have
  \begin{equation}
   \begin{split}
    \{&{\widehat{A}(TM,\nabla^{TM})}[{\rm ch}(\triangle(M)\otimes 2\widetilde{T_{\mathbf{C}}M}){\rm ch}(\triangle(E),g^{\triangle(E)},d)\\
    &+{\rm ch}(\triangle(M)){\rm ch}(\Delta(E)\otimes(2\wedge^2\widetilde{E_\mathbf{C}}-\widetilde{E_\mathbf{C}}
    \otimes\widetilde{E_\mathbf{C}}+\widetilde{E_\mathbf{C}}),g,d)]\\
    &+32\widehat{A}(TM,\nabla^{TM})[{\rm ch}(\widetilde{T_{\mathbf{C}}M}+\wedge^2\widetilde{T_{\mathbf{C}}M}){\rm ch}(\triangle(E),g^{\triangle(E)},d)\\
    &+{\rm ch}(\Delta(E)\otimes(2\wedge^2\widetilde{E_\mathbf{C}}-\widetilde{E_\mathbf{C}}
    \otimes\widetilde{E_\mathbf{C}}+\widetilde{E_\mathbf{C}}),g,d)]\}^{(7)}\\
=&240[\widehat{A}(TM,\nabla^{TM}){\rm ch}(\triangle(M)){\rm ch}(\triangle(E),g^{\triangle(E)},d)\\
&+32\widehat{A}(TM,\nabla^{TM}){\rm ch}(\triangle(E),g^{\triangle(E)},d)],
   \end{split}
  \end{equation}
  \begin{equation}
    \begin{split}
    \{\widehat{A}&(TM,\nabla^{TM})[{\rm ch}(\triangle(M)\otimes(2\widetilde{T_\mathbf{C}M}+\wedge^2\widetilde{T_\mathbf{C}M}\\
&+\widetilde{T_\mathbf{C}M}\otimes\widetilde{T_\mathbf{C}M}+S^2\widetilde{T_\mathbf{C}M})){\rm ch}(\triangle(E),g^{\triangle(E)},d)\\
&+{\rm ch}(\triangle(M)\otimes2\widetilde
{T_{\mathbf{C}}M}){\rm ch}(\triangle(E)\otimes(2\wedge^2\widetilde{E_\mathbf{C}}-\widetilde
{E_\mathbf{C}}\otimes\widetilde{E_\mathbf{C}}+\widetilde{E_\mathbf{C}}),g,d)\\
&+{\rm ch}(\triangle(M)){\rm ch}(\triangle(E)\otimes(\wedge^2\widetilde{E_\mathbf{C}}\otimes\wedge^2\widetilde{E_\mathbf{C}}
+2\wedge^4\widetilde
{E_\mathbf{C}}-2\widetilde{E_\mathbf{C}}\otimes\wedge^3
\widetilde{E_\mathbf{C}}\\
&+2\widetilde{E_\mathbf{C}}\otimes\wedge^2\widetilde{E_\mathbf{C}}
-\widetilde{E_\mathbf{C}}\otimes\widetilde{E_\mathbf{C}}\otimes\widetilde{E_\mathbf{C}}
+\widetilde{E_\mathbf{C}}+\wedge^2\widetilde{E_\mathbf{C}}),g,d)]\\
&+32\widehat{A}(TM,\nabla^{TM})[{\rm ch}(\wedge^4\widetilde{T_\mathbf{C}M}+\wedge^2\widetilde{T_\mathbf{C}M}\otimes
\widetilde{T_\mathbf{C}M}\\
&+\widetilde{T_\mathbf{C}M}\otimes\widetilde{T_\mathbf{C}M}+S^2
\widetilde{T_\mathbf{C}M}+\widetilde{T_\mathbf{C}M}){\rm ch}(\triangle(E),g^{\triangle(E)},d)\\
&+{\rm ch}(\widetilde{T_{
\mathbf{C}}M}+\wedge^2\widetilde{T_{\mathbf{C}}M}){\rm ch}(\triangle(E)\otimes(2\wedge^2\widetilde{E_\mathbf{C}}-\widetilde
{E_\mathbf{C}}\otimes\widetilde{E_\mathbf{C}}+\widetilde{E_\mathbf{C}}),g,d)\\
&+{\rm ch}(\triangle(E)\otimes(\wedge^2\widetilde{E_\mathbf{C}}\otimes\wedge^2
\widetilde{E_\mathbf{C}}
+2\wedge^4\widetilde{E_\mathbf{C}}-2\widetilde{E_\mathbf{C}}\otimes\wedge^3
\widetilde{E_\mathbf{C}}\\
&+2\widetilde{E_\mathbf{C}}\otimes\wedge^2\widetilde{E_\mathbf{C}}
-\widetilde{E_\mathbf{C}}\otimes\widetilde{E_\mathbf{C}}\otimes\widetilde{E_\mathbf{C}}
+\widetilde{E_\mathbf{C}}+\wedge^2\widetilde{E_\mathbf{C}}),g,d)]\}^{(7)}\\
=&2160[\widehat{A}(TM,\nabla^{TM}){\rm ch}(\triangle(M)){\rm ch}(\triangle(E),g^{\triangle(E)},d)\\
&+32\widehat{A}(TM,\nabla^{TM}){\rm ch}(\triangle(E),g^{\triangle(E)},d)].
    \end{split}
  \end{equation}

\end{thm}
\begin{proof}
  It is well known that modular forms over $SL_2({\bf Z})$ can be expressed as polynomials of the Einsentein series $E_4(\tau)$ and $E_6(\tau)$,
where
 \begin{equation}
E_4(\tau)=1+240q+2160q^2+6720q^3+\cdots,
\end{equation}
\begin{equation}
E_6(\tau)=1-504q-16632q^2-122976q^3+\cdots.
\end{equation}
Their weights are $4$ and $6$ respectively. When ${\rm dim}M=7$, then $Q(\nabla^{TM},g,d,\tau)$ is a modular form over $SL_2({\bf Z})$ with the weight $4$.
By (3.1)-(3.3) and (3.10)-(3.12), we have
\begin{equation}
\Theta_1(T_{\mathbf{C}}M)=1+2q\widetilde{T_\mathbf{C}M}+q^2(2\widetilde{T_\mathbf{C}M}+\wedge^2
\widetilde{T_\mathbf{C}M}+\widetilde{T_\mathbf{C}M}\otimes\widetilde{T_\mathbf{C}M}+S^2\widetilde
{T_\mathbf{C}M})+O(q^3),
\end{equation}
\begin{equation}
\begin{split}
\Theta_2(T_{\mathbf{C}}M)=&1-q^{\frac{1}{2}}\widetilde{T_\mathbf{C}M}+q(\widetilde{T_\mathbf{C}M}+
\wedge^2\widetilde{T_\mathbf{C}M})\\
&-q^{\frac{3}{2}}(\wedge^3\widetilde{T_\mathbf{C}M}
+\widetilde{T_\mathbf{C}M}+\widetilde{T_\mathbf{C}M}\otimes\widetilde{T_\mathbf{C}M})\\
&+q^2(\wedge^4\widetilde{T_\mathbf{C}M}+\wedge^2\widetilde{T_\mathbf{C}M}\otimes\widetilde
{T_\mathbf{C}M}+\widetilde{T_\mathbf{C}M}\otimes\widetilde{T_\mathbf{C}M}\\
&+S^2\widetilde{T_\mathbf{C}M}+\widetilde{T_\mathbf{C}M})+O(q^{\frac{5}{2}}),
\end{split}
\end{equation}
\begin{equation}
\begin{split}
\Theta_3(T_{\mathbf{C}}M)=&1+q^{\frac{1}{2}}\widetilde{T_\mathbf{C}M}+q(\widetilde{T_\mathbf{C}M}+
\wedge^2\widetilde{T_\mathbf{C}M})\\
&+q^{\frac{3}{2}}(\wedge^3\widetilde{T_\mathbf{C}M}
+\widetilde{T_\mathbf{C}M}+\widetilde{T_\mathbf{C}M}\otimes\widetilde{T_\mathbf{C}M})\\
&+q^2(\wedge^4\widetilde{T_\mathbf{C}M}+\wedge^2\widetilde{T_\mathbf{C}M}\otimes\widetilde
{T_\mathbf{C}M}+\widetilde{T_\mathbf{C}M}\otimes\widetilde{T_\mathbf{C}M}\\
&+S^2\widetilde{T_\mathbf{C}M}+\widetilde{T_\mathbf{C}M})+O(q^{\frac{5}{2}}),
\end{split}
\end{equation}
\begin{equation}
\begin{split}
Q(E)=&{\rm ch}(\triangle(E),g^{\triangle(E)},d)+q{\rm ch}(\triangle(E)\otimes(2\wedge^2\widetilde{E_\mathbf{C}}-\widetilde{E_\mathbf{C}}\otimes\widetilde{E_\mathbf{C}}+
\widetilde{E_\mathbf{C}}),g,d)\\
&+q^2{\rm ch}(\triangle(E)\otimes(\wedge^2\widetilde{E_\mathbf{C}}\otimes\wedge^2\widetilde{E_\mathbf{C}}+2\wedge^4\widetilde
{E_\mathbf{C}}-2\widetilde{E_\mathbf{C}}\otimes\wedge^3\widetilde{E_\mathbf{C}}+2\widetilde
{E_\mathbf{C}}\otimes\wedge^2\widetilde{E_\mathbf{C}}\\
&-\widetilde{E_\mathbf{C}}\otimes\widetilde
{E_\mathbf{C}}\otimes\widetilde{E_\mathbf{C}}+\widetilde{E_\mathbf{C}}+\wedge^2\widetilde
{E_\mathbf{C}}),g,d)+\cdots,
\end{split}
\end{equation}
so
\begin{equation}
\begin{split}
Q(\nabla^{TM},g,d,\tau)=&[{\widehat{A}(TM,\nabla^{TM})}{\rm ch}(\triangle(M)){\rm ch}(\triangle(E),g^{\triangle(E)},d)\\
&+2^{2k+1}\widehat{A}(TM,\nabla^{TM}){\rm ch}(\triangle(E),g^{\triangle(E)},d)]^{(4k-1)}\\
&+q\{\widehat{A}(TM,\nabla^{TM})[{\rm ch}(\triangle(M)\otimes 2\widetilde{T_\mathbf{C}M}){\rm ch}(\triangle(E),g^{\triangle(E)},d)\\
&+{\rm ch}(\triangle(M)){\rm ch}(\triangle(E)\otimes(2\wedge^2\widetilde{E_\mathbf{C}}-\widetilde
{E_\mathbf{C}}\otimes\widetilde{E_\mathbf{C}}+\widetilde{E_\mathbf{C}}),g,d)]\\
&+2^{2k+1}\widehat{A}(TM,\nabla^{TM})[{\rm ch}(\widetilde{T_\mathbf{C}M}+\wedge^2\widetilde{T_\mathbf{C}M}){\rm ch}(\triangle(E),g^{\triangle(E)},d)\\
&+{\rm ch}(\triangle(E)\otimes(2\wedge^2\widetilde{E_\mathbf{C}}
-\widetilde{E_\mathbf{C}}\otimes\widetilde{E_\mathbf{C}}+\widetilde{E_\mathbf{C}}),g,d)]\}^{(4k-1)}\\
&+q^2\{{\widehat{A}(TM,\nabla^{TM})}[{\rm ch}(\triangle(M)\otimes(2\widetilde{T_\mathbf{C}M}+\wedge^2\widetilde{T_\mathbf{C}M}\\
&+\widetilde{T_\mathbf{C}M}\otimes\widetilde{T_\mathbf{C}M}+S^2\widetilde{T_\mathbf{C}M})){\rm ch}(\triangle(E),g^{\triangle(E)},d)\\
&+{\rm ch}(\triangle(M)\otimes2\widetilde
{T_{\mathbf{C}}M}){\rm ch}(\triangle(E)\otimes(2\wedge^2\widetilde{E_\mathbf{C}}-\widetilde
{E_\mathbf{C}}\otimes\widetilde{E_\mathbf{C}}+\widetilde{E_\mathbf{C}}),g,d)\\
&+{\rm ch}(\triangle(M)){\rm ch}(\triangle(E)\otimes(\wedge^2\widetilde{E_\mathbf{C}}\otimes\wedge^2\widetilde{E_\mathbf{C}}
+2\wedge^4\widetilde
{E_\mathbf{C}}-2\widetilde{E_\mathbf{C}}\otimes\wedge^3
\widetilde{E_\mathbf{C}}\\
&+2\widetilde{E_\mathbf{C}}\otimes\wedge^2\widetilde{E_\mathbf{C}}
-\widetilde{E_\mathbf{C}}\otimes\widetilde{E_\mathbf{C}}\otimes\widetilde{E_\mathbf{C}}
+\widetilde{E_\mathbf{C}}+\wedge^2\widetilde{E_\mathbf{C}}),g,d)]\\
&+2^{2k+1}\widehat{A}(TM,\nabla^{TM})[{\rm ch}(\wedge^4\widetilde{T_\mathbf{C}M}+\wedge^2\widetilde{T_\mathbf{C}M}\otimes
\widetilde{T_\mathbf{C}M}\\
&+\widetilde{T_\mathbf{C}M}\otimes\widetilde{T_\mathbf{C}M}+S^2
\widetilde{T_\mathbf{C}M}+\widetilde{T_\mathbf{C}M}){\rm ch}(\triangle(E),g^{\triangle(E)},d)\\
&+{\rm ch}(\widetilde{T_{
\mathbf{C}}M}+\wedge^2\widetilde{T_{\mathbf{C}}M}){\rm ch}(\triangle(E)\otimes(2\wedge^2\widetilde{E_\mathbf{C}}-\widetilde
{E_\mathbf{C}}\otimes\widetilde{E_\mathbf{C}}+\widetilde{E_\mathbf{C}}),g,d)\\
&+{\rm ch}(\triangle(E)\otimes(\wedge^2\widetilde{E_\mathbf{C}}\otimes\wedge^2
\widetilde{E_\mathbf{C}}
+2\wedge^4\widetilde{E_\mathbf{C}}-2\widetilde{E_\mathbf{C}}\otimes\wedge^3
\widetilde{E_\mathbf{C}}\\
&+2\widetilde{E_\mathbf{C}}\otimes\wedge^2\widetilde{E_\mathbf{C}}
-\widetilde{E_\mathbf{C}}\otimes\widetilde{E_\mathbf{C}}\otimes\widetilde{E_\mathbf{C}}
+\widetilde{E_\mathbf{C}}+\wedge^2\widetilde{E_\mathbf{C}}),g,d)]\}^{(4k-1)}+\cdots.
\end{split}
\end{equation}
When ${\rm dim}M=7$, then $Q(\nabla^{TM},g,d,\tau)$ must be a multiple of
$$E_4(\tau)=1+240q+2160q^2+6720q^3+\cdots,$$
By (3.25) and (3.31), we compare the coefficients of $1$, $q$, $q^2$. We get Theorem 3.2.
\end{proof}

\begin{cor}
Let $M$ be an $7$-dimensional spin manifold without boundary. If $c_3(E,g,d)=0$, then
\begin{equation}
\begin{split}
{\rm Ind}(T &\otimes (\triangle(M)\otimes2\widetilde{T_\mathbf{C}M}\otimes(\triangle(E),g^{\triangle(E)},d)\\
&+\triangle(M){\rm ch}(\Delta(E)\otimes(2\wedge^2\widetilde{E_\mathbf{C}}-\widetilde{E_\mathbf{C}}
\otimes\widetilde{E_\mathbf{C}}+\widetilde{E_\mathbf{C}}),g,d)))\equiv 0 ~~{\rm mod} ~~16Z,
\end{split}
\end{equation}
\begin{equation}
  \begin{split}
   {\rm Ind}(T &\otimes (\triangle(M)\otimes((2\widetilde{T_\mathbf{C}M}+\wedge^2\widetilde{T_\mathbf{C}M}+
   \widetilde{T_\mathbf{C}M}\otimes\widetilde{T_\mathbf{C}M}\\
   &+S^2\widetilde{T_\mathbf{C}M})
   \otimes(\triangle(E),g^{\triangle(E)},d)+2\widetilde{T_\mathbf{C}M}\otimes((2\wedge^2\widetilde
   {E_\mathbf{C}}\\
   &-\widetilde{E_\mathbf{C}}\otimes\widetilde{E_\mathbf{C}}
   +\widetilde{E_\mathbf{C}}),g,d)+(\triangle(E)\otimes(\wedge^2\widetilde{E_\mathbf{C}}
   \otimes\wedge^2\widetilde{E_\mathbf{C}}\\
   &+2\wedge^4\widetilde{E_\mathbf{C}}-2\widetilde{E_\mathbf{C}}
   \otimes\wedge^3\widetilde{E_\mathbf{C}}+2\widetilde{E_\mathbf{C}}\otimes\wedge^2\widetilde
   {E_\mathbf{C}}-\widetilde{E_\mathbf{C}}\otimes\widetilde{E_\mathbf{C}}
   \otimes\widetilde{E_\mathbf{C}}\\
   &+\widetilde{E_\mathbf{C}}+\wedge^2\widetilde{E_\mathbf{C}}),g,d))))\equiv 0 ~~{\rm mod} ~~16Z
  \end{split}
\end{equation}
\end{cor}

\begin{thm}
  When ${\rm dim}M=11$ and $c_3(E,g,d)=0$, we have
  \begin{equation}
   \begin{split}
    \{&{\widehat{A}(TM,\nabla^{TM})}[{\rm ch}(\triangle(M)\otimes 2\widetilde{T_{\mathbf{C}}M}){\rm ch}(\triangle(E),g^{\triangle(E)},d)\\
    &+{\rm ch}(\triangle(M)){\rm ch}(\Delta(E)\otimes(2\wedge^2\widetilde{E_\mathbf{C}}-\widetilde{E_\mathbf{C}}
    \otimes\widetilde{E_\mathbf{C}}+\widetilde{E_\mathbf{C}}),g,d)]\\
    &+128\widehat{A}(TM,\nabla^{TM})[{\rm ch}(\widetilde{T_{\mathbf{C}}M}+\wedge^2\widetilde{T_{\mathbf{C}}M}){\rm ch}(\triangle(E),g^{\triangle(E)},d)\\
    &+{\rm ch}(\Delta(E)\otimes(2\wedge^2\widetilde{E_\mathbf{C}}-\widetilde{E_\mathbf{C}}
    \otimes\widetilde{E_\mathbf{C}}+\widetilde{E_\mathbf{C}}),g,d)]\}^{(11)}\\
=&-504[\widehat{A}(TM,\nabla^{TM}){\rm ch}(\triangle(M)){\rm ch}(\triangle(E),g^{\triangle(E)},d)\\
&+128\widehat{A}(TM,\nabla^{TM}){\rm ch}(\triangle(E),g^{\triangle(E)},d)],
   \end{split}
  \end{equation}
    \begin{equation}
    \begin{split}
    \{\widehat{A}&(TM,\nabla^{TM})[{\rm ch}(\triangle(M)\otimes(2\widetilde{T_\mathbf{C}M}+\wedge^2\widetilde{T_\mathbf{C}M}\\
&+\widetilde{T_\mathbf{C}M}\otimes\widetilde{T_\mathbf{C}M}+S^2\widetilde{T_\mathbf{C}M})){\rm ch}(\triangle(E),g^{\triangle(E)},d)\\
&+{\rm ch}(\triangle(M)\otimes2\widetilde
{T_{\mathbf{C}}M}){\rm ch}(\triangle(E)\otimes(2\wedge^2\widetilde{E_\mathbf{C}}-\widetilde
{E_\mathbf{C}}\otimes\widetilde{E_\mathbf{C}}+\widetilde{E_\mathbf{C}}),g,d)\\
&+{\rm ch}(\triangle(M)){\rm ch}(\triangle(E)\otimes(\wedge^2\widetilde{E_\mathbf{C}}\otimes\wedge^2\widetilde{E_\mathbf{C}}
+2\wedge^4\widetilde
{E_\mathbf{C}}-2\widetilde{E_\mathbf{C}}\otimes\wedge^3
\widetilde{E_\mathbf{C}}\\
&+2\widetilde{E_\mathbf{C}}\otimes\wedge^2\widetilde{E_\mathbf{C}}
-\widetilde{E_\mathbf{C}}\otimes\widetilde{E_\mathbf{C}}\otimes\widetilde{E_\mathbf{C}}
+\widetilde{E_\mathbf{C}}+\wedge^2\widetilde{E_\mathbf{C}}),g,d)]\\
&+128\widehat{A}(TM,\nabla^{TM})[{\rm ch}(\wedge^4\widetilde{T_\mathbf{C}M}+\wedge^2\widetilde{T_\mathbf{C}M}\otimes
\widetilde{T_\mathbf{C}M}\\
&+\widetilde{T_\mathbf{C}M}\otimes\widetilde{T_\mathbf{C}M}+S^2
\widetilde{T_\mathbf{C}M}+\widetilde{T_\mathbf{C}M}){\rm ch}(\triangle(E),g^{\triangle(E)},d)\\
&+{\rm ch}(\widetilde{T_{
\mathbf{C}}M}+\wedge^2\widetilde{T_{\mathbf{C}}M}){\rm ch}(\triangle(E)\otimes(2\wedge^2\widetilde{E_\mathbf{C}}-\widetilde
{E_\mathbf{C}}\otimes\widetilde{E_\mathbf{C}}+\widetilde{E_\mathbf{C}}),g,d)\\
&+{\rm ch}(\triangle(E)\otimes(\wedge^2\widetilde{E_\mathbf{C}}\otimes\wedge^2
\widetilde{E_\mathbf{C}}
+2\wedge^4\widetilde{E_\mathbf{C}}-2\widetilde{E_\mathbf{C}}\otimes\wedge^3
\widetilde{E_\mathbf{C}}\\
&+2\widetilde{E_\mathbf{C}}\otimes\wedge^2\widetilde{E_\mathbf{C}}
-\widetilde{E_\mathbf{C}}\otimes\widetilde{E_\mathbf{C}}\otimes\widetilde{E_\mathbf{C}}
+\widetilde{E_\mathbf{C}}+\wedge^2\widetilde{E_\mathbf{C}}),g,d)]\}^{(11)}\\
=&-16632[\widehat{A}(TM,\nabla^{TM}){\rm ch}(\triangle(M)){\rm ch}(\triangle(E),g^{\triangle(E)},d)\\
&+128\widehat{A}(TM,\nabla^{TM}){\rm ch}(\triangle(E),g^{\triangle(E)},d)].
    \end{split}
  \end{equation}
\end{thm}
\begin{proof}
  When ${\rm dim}M=11$, then $Q(\nabla^{TM},g,d,\tau)$ is a modular form over $SL_2({\bf Z})$ with the weight $6$, so
\begin{equation}
  Q(\nabla^{TM},g,d,\tau)=\lambda E_6(\tau),
\end{equation}
where $\lambda$ is degree 11 forms.
When ${\rm dim}M=11$, direct computations show that
\begin{equation}
\begin{split}
Q(\nabla^{TM},g,d,\tau)=&[{\widehat{A}(TM,\nabla^{TM})}{\rm ch}(\triangle(M)){\rm ch}(\triangle(E),g^{\triangle(E)},d)\\
&+128\widehat{A}(TM,\nabla^{TM}){\rm ch}(\triangle(E),g^{\triangle(E)},d)]^{(11)}\\
&+q\{\widehat{A}(TM,\nabla^{TM})[{\rm ch}(\triangle(M)\otimes 2\widetilde{T_\mathbf{C}M}){\rm ch}(\triangle(E),g^{\triangle(E)},d)\\
&+{\rm ch}(\triangle(M)){\rm ch}(\triangle(E)\otimes(2\wedge^2\widetilde{E_\mathbf{C}}-\widetilde
{E_\mathbf{C}}\otimes\widetilde{E_\mathbf{C}}+\widetilde{E_\mathbf{C}}),g,d)]\\
&+128\widehat{A}(TM,\nabla^{TM})[{\rm ch}(\widetilde{T_\mathbf{C}M}+\wedge^2\widetilde{T_\mathbf{C}M}){\rm ch}(\triangle(E),g^{\triangle(E)},d)\\
&+{\rm ch}(\triangle(E)\otimes(2\wedge^2\widetilde{E_\mathbf{C}}
-\widetilde{E_\mathbf{C}}\otimes\widetilde{E_\mathbf{C}}+\widetilde{E_\mathbf{C}}),g,d)]\}^{(11)}\\
&+q^2\{{\widehat{A}(TM,\nabla^{TM})}[{\rm ch}(\triangle(M)\otimes(2\widetilde{T_\mathbf{C}M}+\wedge^2\widetilde{T_\mathbf{C}M}\\
&+\widetilde{T_\mathbf{C}M}\otimes\widetilde{T_\mathbf{C}M}+S^2\widetilde{T_\mathbf{C}M})){\rm ch}(\triangle(E),g^{\triangle(E)},d)\\
&+{\rm ch}(\triangle(M)\otimes2\widetilde
{T_{\mathbf{C}}M}){\rm ch}(\triangle(E)\otimes(2\wedge^2\widetilde{E_\mathbf{C}}-\widetilde
{E_\mathbf{C}}\otimes\widetilde{E_\mathbf{C}}+\widetilde{E_\mathbf{C}}),g,d)\\
&+{\rm ch}(\triangle(M)){\rm ch}(\triangle(E)\otimes(\wedge^2\widetilde{E_\mathbf{C}}\otimes\wedge^2\widetilde{E_\mathbf{C}}
+2\wedge^4\widetilde
{E_\mathbf{C}}-2\widetilde{E_\mathbf{C}}\otimes\wedge^3
\widetilde{E_\mathbf{C}}\\
&+2\widetilde{E_\mathbf{C}}\otimes\wedge^2\widetilde{E_\mathbf{C}}
-\widetilde{E_\mathbf{C}}\otimes\widetilde{E_\mathbf{C}}\otimes\widetilde{E_\mathbf{C}}
+\widetilde{E_\mathbf{C}}+\wedge^2\widetilde{E_\mathbf{C}}),g,d)]\\
&+128\widehat{A}(TM,\nabla^{TM})[{\rm ch}(\wedge^4\widetilde{T_\mathbf{C}M}+\wedge^2\widetilde{T_\mathbf{C}M}\otimes
\widetilde{T_\mathbf{C}M}\\
&+\widetilde{T_\mathbf{C}M}\otimes\widetilde{T_\mathbf{C}M}+S^2
\widetilde{T_\mathbf{C}M}+\widetilde{T_\mathbf{C}M}){\rm ch}(\triangle(E),g^{\triangle(E)},d)\\
&+{\rm ch}(\widetilde{T_{
\mathbf{C}}M}+\wedge^2\widetilde{T_{\mathbf{C}}M}){\rm ch}(\triangle(E)\otimes(2\wedge^2\widetilde{E_\mathbf{C}}-\widetilde
{E_\mathbf{C}}\otimes\widetilde{E_\mathbf{C}}+\widetilde{E_\mathbf{C}}),g,d)\\
&+{\rm ch}(\triangle(E)\otimes(\wedge^2\widetilde{E_\mathbf{C}}\otimes\wedge^2
\widetilde{E_\mathbf{C}}
+2\wedge^4\widetilde{E_\mathbf{C}}-2\widetilde{E_\mathbf{C}}\otimes\wedge^3
\widetilde{E_\mathbf{C}}\\
&+2\widetilde{E_\mathbf{C}}\otimes\wedge^2\widetilde{E_\mathbf{C}}
-\widetilde{E_\mathbf{C}}\otimes\widetilde{E_\mathbf{C}}\otimes\widetilde{E_\mathbf{C}}
+\widetilde{E_\mathbf{C}}+\wedge^2\widetilde{E_\mathbf{C}}),g,d)]\}^{(11)}+\cdots.
\end{split}
\end{equation}
In (3.36), we compare the coefficients of (3.36), we get three equations about $\lambda$. By
(3.26), (3.36) and (3.37), we get Theorem 3.4.
\end{proof}
\begin{cor}
Let $M$ be an $11$-dimensional spin manifold without boundary. If $c_3(E,g,d)=0$, then
\begin{equation}
\begin{split}
{\rm Ind}(T &\otimes (\triangle(M)\otimes2\widetilde{T_\mathbf{C}M}\otimes(\triangle(E),g^{\triangle(E)},d)\\
&+\triangle(M){\rm ch}(\Delta(E)\otimes(2\wedge^2\widetilde{E_\mathbf{C}}-\widetilde{E_\mathbf{C}}
\otimes\widetilde{E_\mathbf{C}}+\widetilde{E_\mathbf{C}}),g,d)))\equiv 0 ~~{\rm mod} ~~8Z,
\end{split}
\end{equation}
\begin{equation}
  \begin{split}
   {\rm Ind}(T &\otimes (\triangle(M)\otimes((2\widetilde{T_\mathbf{C}M}+\wedge^2\widetilde{T_\mathbf{C}M}+
   \widetilde{T_\mathbf{C}M}\otimes\widetilde{T_\mathbf{C}M}\\
   &+S^2\widetilde{T_\mathbf{C}M})
   \otimes(\triangle(E),g^{\triangle(E)},d)+2\widetilde{T_\mathbf{C}M}\otimes((2\wedge^2\widetilde
   {E_\mathbf{C}}\\
   &-\widetilde{E_\mathbf{C}}\otimes\widetilde{E_\mathbf{C}}
   +\widetilde{E_\mathbf{C}}),g,d)+(\triangle(E)\otimes(\wedge^2\widetilde{E_\mathbf{C}}
   \otimes\wedge^2\widetilde{E_\mathbf{C}}\\
   &+2\wedge^4\widetilde{E_\mathbf{C}}-2\widetilde{E_\mathbf{C}}
   \otimes\wedge^3\widetilde{E_\mathbf{C}}+2\widetilde{E_\mathbf{C}}\otimes\wedge^2\widetilde
   {E_\mathbf{C}}-\widetilde{E_\mathbf{C}}\otimes\widetilde{E_\mathbf{C}}
   \otimes\widetilde{E_\mathbf{C}}\\
   &+\widetilde{E_\mathbf{C}}+\wedge^2\widetilde{E_\mathbf{C}}),g,d))))\equiv 0 ~~{\rm mod} ~~8Z.
  \end{split}
\end{equation}
\end{cor}

\begin{thm}
  When ${\rm dim}M=15$ and $c_3(E,g,d)=0$, we have
  \begin{equation}
   \begin{split}
    \{&{\widehat{A}(TM,\nabla^{TM})}[{\rm ch}(\triangle(M)\otimes 2\widetilde{T_{\mathbf{C}}M}){\rm ch}(\triangle(E),g^{\triangle(E)},d)\\
    &+{\rm ch}(\triangle(M)){\rm ch}(\Delta(E)\otimes(2\wedge^2\widetilde{E_\mathbf{C}}-\widetilde{E_\mathbf{C}}
    \otimes\widetilde{E_\mathbf{C}}+\widetilde{E_\mathbf{C}}),g,d)]\\
    &+512\widehat{A}(TM,\nabla^{TM})[{\rm ch}(\widetilde{T_{\mathbf{C}}M}+\wedge^2\widetilde{T_{\mathbf{C}}M}){\rm ch}(\triangle(E),g^{\triangle(E)},d)\\
    &+{\rm ch}(\Delta(E)\otimes(2\wedge^2\widetilde{E_\mathbf{C}}-\widetilde{E_\mathbf{C}}
    \otimes\widetilde{E_\mathbf{C}}+\widetilde{E_\mathbf{C}}),g,d)]\}^{(15)}\\
=&480[\widehat{A}(TM,\nabla^{TM}){\rm ch}(\triangle(M)){\rm ch}(\triangle(E),g^{\triangle(E)},d)\\
&+512\widehat{A}(TM,\nabla^{TM}){\rm ch}(\triangle(E),g^{\triangle(E)},d)],
   \end{split}
  \end{equation}
    \begin{equation}
    \begin{split}
    \{\widehat{A}&(TM,\nabla^{TM})[{\rm ch}(\triangle(M)\otimes(2\widetilde{T_\mathbf{C}M}+\wedge^2\widetilde{T_\mathbf{C}M}\\
&+\widetilde{T_\mathbf{C}M}\otimes\widetilde{T_\mathbf{C}M}+S^2\widetilde{T_\mathbf{C}M})){\rm ch}(\triangle(E),g^{\triangle(E)},d)\\
&+{\rm ch}(\triangle(M)\otimes2\widetilde
{T_{\mathbf{C}}M}){\rm ch}(\triangle(E)\otimes(2\wedge^2\widetilde{E_\mathbf{C}}-\widetilde
{E_\mathbf{C}}\otimes\widetilde{E_\mathbf{C}}+\widetilde{E_\mathbf{C}}),g,d)\\
&+{\rm ch}(\triangle(M)){\rm ch}(\triangle(E)\otimes(\wedge^2\widetilde{E_\mathbf{C}}\otimes\wedge^2\widetilde{E_\mathbf{C}}
+2\wedge^4\widetilde
{E_\mathbf{C}}-2\widetilde{E_\mathbf{C}}\otimes\wedge^3
\widetilde{E_\mathbf{C}}\\
&+2\widetilde{E_\mathbf{C}}\otimes\wedge^2\widetilde{E_\mathbf{C}}
-\widetilde{E_\mathbf{C}}\otimes\widetilde{E_\mathbf{C}}\otimes\widetilde{E_\mathbf{C}}
+\widetilde{E_\mathbf{C}}+\wedge^2\widetilde{E_\mathbf{C}}),g,d)]\\
&+512\widehat{A}(TM,\nabla^{TM})[{\rm ch}(\wedge^4\widetilde{T_\mathbf{C}M}+\wedge^2\widetilde{T_\mathbf{C}M}\otimes
\widetilde{T_\mathbf{C}M}\\
&+\widetilde{T_\mathbf{C}M}\otimes\widetilde{T_\mathbf{C}M}+S^2
\widetilde{T_\mathbf{C}M}+\widetilde{T_\mathbf{C}M}){\rm ch}(\triangle(E),g^{\triangle(E)},d)\\
&+{\rm ch}(\widetilde{T_{
\mathbf{C}}M}+\wedge^2\widetilde{T_{\mathbf{C}}M}){\rm ch}(\triangle(E)\otimes(2\wedge^2\widetilde{E_\mathbf{C}}-\widetilde
{E_\mathbf{C}}\otimes\widetilde{E_\mathbf{C}}+\widetilde{E_\mathbf{C}}),g,d)\\
&+{\rm ch}(\triangle(E)\otimes(\wedge^2\widetilde{E_\mathbf{C}}\otimes\wedge^2
\widetilde{E_\mathbf{C}}
+2\wedge^4\widetilde{E_\mathbf{C}}-2\widetilde{E_\mathbf{C}}\otimes\wedge^3
\widetilde{E_\mathbf{C}}\\
&+2\widetilde{E_\mathbf{C}}\otimes\wedge^2\widetilde{E_\mathbf{C}}
-\widetilde{E_\mathbf{C}}\otimes\widetilde{E_\mathbf{C}}\otimes\widetilde{E_\mathbf{C}}
+\widetilde{E_\mathbf{C}}+\wedge^2\widetilde{E_\mathbf{C}}),g,d)]\}^{(15)}\\
=&61920[\widehat{A}(TM,\nabla^{TM}){\rm ch}(\triangle(M)){\rm ch}(\triangle(E),g^{\triangle(E)},d)\\
&+512\widehat{A}(TM,\nabla^{TM}){\rm ch}(\triangle(E),g^{\triangle(E)},d)].
    \end{split}
  \end{equation}
\end{thm}
\begin{proof}
  When ${\rm dim}M=15$, then $Q(\nabla^{TM},g,d,\tau)$ is a modular form over $SL_2({\bf Z})$ with the weight $8$, so
\begin{equation}
  Q(\nabla^{TM},g,d,\tau)=\lambda E_4(\tau)^2,
\end{equation}
where $\lambda$ is degree 15 forms. By (3.25), we have
\begin{equation}
E_4(\tau)^2=1+480q+61920q^2+\cdots.
\end{equation}
When ${\rm dim}M=15$, direct computations show that
\begin{equation}
\begin{split}
Q(\nabla^{TM},g,d,\tau)=&[{\widehat{A}(TM,\nabla^{TM})}{\rm ch}(\triangle(M)){\rm ch}(\triangle(E),g^{\triangle(E)},d)\\
&+512\widehat{A}(TM,\nabla^{TM}){\rm ch}(\triangle(E),g^{\triangle(E)},d)]^{(15)}\\
&+q\{\widehat{A}(TM,\nabla^{TM})[{\rm ch}(\triangle(M)\otimes 2\widetilde{T_\mathbf{C}M}){\rm ch}(\triangle(E),g^{\triangle(E)},d)\\
&+{\rm ch}(\triangle(M)){\rm ch}(\triangle(E)\otimes(2\wedge^2\widetilde{E_\mathbf{C}}-\widetilde
{E_\mathbf{C}}\otimes\widetilde{E_\mathbf{C}}+\widetilde{E_\mathbf{C}}),g,d)]\\
&+512\widehat{A}(TM,\nabla^{TM})[{\rm ch}(\widetilde{T_\mathbf{C}M}+\wedge^2\widetilde{T_\mathbf{C}M}){\rm ch}(\triangle(E),g^{\triangle(E)},d)\\
&+{\rm ch}(\triangle(E)\otimes(2\wedge^2\widetilde{E_\mathbf{C}}
-\widetilde{E_\mathbf{C}}\otimes\widetilde{E_\mathbf{C}}+\widetilde{E_\mathbf{C}}),g,d)]\}^{(15)}\\
&+q^2\{{\widehat{A}(TM,\nabla^{TM})}[{\rm ch}(\triangle(M)\otimes(2\widetilde{T_\mathbf{C}M}+\wedge^2\widetilde{T_\mathbf{C}M}\\
&+\widetilde{T_\mathbf{C}M}\otimes\widetilde{T_\mathbf{C}M}+S^2\widetilde{T_\mathbf{C}M})){\rm ch}(\triangle(E),g^{\triangle(E)},d)\\
&+{\rm ch}(\triangle(M)\otimes2\widetilde
{T_{\mathbf{C}}M}){\rm ch}(\triangle(E)\otimes(2\wedge^2\widetilde{E_\mathbf{C}}-\widetilde
{E_\mathbf{C}}\otimes\widetilde{E_\mathbf{C}}+\widetilde{E_\mathbf{C}}),g,d)\\
&+{\rm ch}(\triangle(M)){\rm ch}(\triangle(E)\otimes(\wedge^2\widetilde{E_\mathbf{C}}\otimes\wedge^2\widetilde{E_\mathbf{C}}
+2\wedge^4\widetilde
{E_\mathbf{C}}-2\widetilde{E_\mathbf{C}}\otimes\wedge^3
\widetilde{E_\mathbf{C}}\\
&+2\widetilde{E_\mathbf{C}}\otimes\wedge^2\widetilde{E_\mathbf{C}}
-\widetilde{E_\mathbf{C}}\otimes\widetilde{E_\mathbf{C}}\otimes\widetilde{E_\mathbf{C}}
+\widetilde{E_\mathbf{C}}+\wedge^2\widetilde{E_\mathbf{C}}),g,d)]\\
&+512\widehat{A}(TM,\nabla^{TM})[{\rm ch}(\wedge^4\widetilde{T_\mathbf{C}M}+\wedge^2\widetilde{T_\mathbf{C}M}\otimes
\widetilde{T_\mathbf{C}M}\\
&+\widetilde{T_\mathbf{C}M}\otimes\widetilde{T_\mathbf{C}M}+S^2
\widetilde{T_\mathbf{C}M}+\widetilde{T_\mathbf{C}M}){\rm ch}(\triangle(E),g^{\triangle(E)},d)\\
&+{\rm ch}(\widetilde{T_{
\mathbf{C}}M}+\wedge^2\widetilde{T_{\mathbf{C}}M}){\rm ch}(\triangle(E)\otimes(2\wedge^2\widetilde{E_\mathbf{C}}-\widetilde
{E_\mathbf{C}}\otimes\widetilde{E_\mathbf{C}}+\widetilde{E_\mathbf{C}}),g,d)\\
&+{\rm ch}(\triangle(E)\otimes(\wedge^2\widetilde{E_\mathbf{C}}\otimes\wedge^2
\widetilde{E_\mathbf{C}}
+2\wedge^4\widetilde{E_\mathbf{C}}-2\widetilde{E_\mathbf{C}}\otimes\wedge^3
\widetilde{E_\mathbf{C}}\\
&+2\widetilde{E_\mathbf{C}}\otimes\wedge^2\widetilde{E_\mathbf{C}}
-\widetilde{E_\mathbf{C}}\otimes\widetilde{E_\mathbf{C}}\otimes\widetilde{E_\mathbf{C}}
+\widetilde{E_\mathbf{C}}+\wedge^2\widetilde{E_\mathbf{C}}),g,d)]\}^{(15)}+\cdots.
\end{split}
\end{equation}
By (3.42), (3.43) and (3.44), we compare the coefficients of $1$, $q$, $q^2$. We get Theorem 3.6.
\end{proof}
\begin{cor}
Let $M$ be an $15$-dimensional spin manifold without boundary. If $c_3(E,g,d)=0$, then
\begin{equation}
\begin{split}
{\rm Ind}(T &\otimes (\triangle(M)\otimes2\widetilde{T_\mathbf{C}M}\otimes(\triangle(E),g^{\triangle(E)},d)\\
&+\triangle(M){\rm ch}(\Delta(E)\otimes(2\wedge^2\widetilde{E_\mathbf{C}}-\widetilde{E_\mathbf{C}}
\otimes\widetilde{E_\mathbf{C}}+\widetilde{E_\mathbf{C}}),g,d)))\equiv 0 ~~{\rm mod} ~~32Z,
\end{split}
\end{equation}
\begin{equation}
  \begin{split}
   {\rm Ind}(T &\otimes (\triangle(M)\otimes((2\widetilde{T_\mathbf{C}M}+\wedge^2\widetilde{T_\mathbf{C}M}+
   \widetilde{T_\mathbf{C}M}\otimes\widetilde{T_\mathbf{C}M}\\
   &+S^2\widetilde{T_\mathbf{C}M})
   \otimes(\triangle(E),g^{\triangle(E)},d)+2\widetilde{T_\mathbf{C}M}\otimes((2\wedge^2\widetilde
   {E_\mathbf{C}}\\
   &-\widetilde{E_\mathbf{C}}\otimes\widetilde{E_\mathbf{C}}
   +\widetilde{E_\mathbf{C}}),g,d)+(\triangle(E)\otimes(\wedge^2\widetilde{E_\mathbf{C}}
   \otimes\wedge^2\widetilde{E_\mathbf{C}}\\
   &+2\wedge^4\widetilde{E_\mathbf{C}}-2\widetilde{E_\mathbf{C}}
   \otimes\wedge^3\widetilde{E_\mathbf{C}}+2\widetilde{E_\mathbf{C}}\otimes\wedge^2\widetilde
   {E_\mathbf{C}}-\widetilde{E_\mathbf{C}}\otimes\widetilde{E_\mathbf{C}}
   \otimes\widetilde{E_\mathbf{C}}\\
   &+\widetilde{E_\mathbf{C}}+\wedge^2\widetilde{E_\mathbf{C}}),g,d))))\equiv 0 ~~{\rm mod} ~~32Z.
  \end{split}
\end{equation}
\end{cor}
\begin{thm}
  When ${\rm dim}M=19$ and $c_3(E,g,d)=0$, we have
  \begin{equation}
   \begin{split}
    \{&{\widehat{A}(TM,\nabla^{TM})}[{\rm ch}(\triangle(M)\otimes 2\widetilde{T_{\mathbf{C}}M}){\rm ch}(\triangle(E),g^{\triangle(E)},d)\\
    &+{\rm ch}(\triangle(M)){\rm ch}(\Delta(E)\otimes(2\wedge^2\widetilde{E_\mathbf{C}}-\widetilde{E_\mathbf{C}}
    \otimes\widetilde{E_\mathbf{C}}+\widetilde{E_\mathbf{C}}),g,d)]\\
    &+2048\widehat{A}(TM,\nabla^{TM})[{\rm ch}(\widetilde{T_{\mathbf{C}}M}+\wedge^2\widetilde{T_{\mathbf{C}}M}){\rm ch}(\triangle(E),g^{\triangle(E)},d)\\
    &+{\rm ch}(\Delta(E)\otimes(2\wedge^2\widetilde{E_\mathbf{C}}-\widetilde{E_\mathbf{C}}
    \otimes\widetilde{E_\mathbf{C}}+\widetilde{E_\mathbf{C}}),g,d)]\}^{(19)}\\
=&-264[\widehat{A}(TM,\nabla^{TM}){\rm ch}(\triangle(M)){\rm ch}(\triangle(E),g^{\triangle(E)},d)\\
&+2048\widehat{A}(TM,\nabla^{TM}){\rm ch}(\triangle(E),g^{\triangle(E)},d)],
   \end{split}
  \end{equation}
    \begin{equation}
    \begin{split}
    \{\widehat{A}&(TM,\nabla^{TM})[{\rm ch}(\triangle(M)\otimes(2\widetilde{T_\mathbf{C}M}+\wedge^2\widetilde{T_\mathbf{C}M}\\
&+\widetilde{T_\mathbf{C}M}\otimes\widetilde{T_\mathbf{C}M}+S^2\widetilde{T_\mathbf{C}M})){\rm ch}(\triangle(E),g^{\triangle(E)},d)\\
&+{\rm ch}(\triangle(M)\otimes2\widetilde
{T_{\mathbf{C}}M}){\rm ch}(\triangle(E)\otimes(2\wedge^2\widetilde{E_\mathbf{C}}-\widetilde
{E_\mathbf{C}}\otimes\widetilde{E_\mathbf{C}}+\widetilde{E_\mathbf{C}}),g,d)\\
&+{\rm ch}(\triangle(M)){\rm ch}(\triangle(E)\otimes(\wedge^2\widetilde{E_\mathbf{C}}\otimes\wedge^2\widetilde{E_\mathbf{C}}
+2\wedge^4\widetilde
{E_\mathbf{C}}-2\widetilde{E_\mathbf{C}}\otimes\wedge^3
\widetilde{E_\mathbf{C}}\\
&+2\widetilde{E_\mathbf{C}}\otimes\wedge^2\widetilde{E_\mathbf{C}}
-\widetilde{E_\mathbf{C}}\otimes\widetilde{E_\mathbf{C}}\otimes\widetilde{E_\mathbf{C}}
+\widetilde{E_\mathbf{C}}+\wedge^2\widetilde{E_\mathbf{C}}),g,d)]\\
&+2048\widehat{A}(TM,\nabla^{TM})[{\rm ch}(\wedge^4\widetilde{T_\mathbf{C}M}+\wedge^2\widetilde{T_\mathbf{C}M}\otimes
\widetilde{T_\mathbf{C}M}\\
&+\widetilde{T_\mathbf{C}M}\otimes\widetilde{T_\mathbf{C}M}+S^2
\widetilde{T_\mathbf{C}M}+\widetilde{T_\mathbf{C}M}){\rm ch}(\triangle(E),g^{\triangle(E)},d)\\
&+{\rm ch}(\widetilde{T_{
\mathbf{C}}M}+\wedge^2\widetilde{T_{\mathbf{C}}M}){\rm ch}(\triangle(E)\otimes(2\wedge^2\widetilde{E_\mathbf{C}}-\widetilde
{E_\mathbf{C}}\otimes\widetilde{E_\mathbf{C}}+\widetilde{E_\mathbf{C}}),g,d)\\
&+{\rm ch}(\triangle(E)\otimes(\wedge^2\widetilde{E_\mathbf{C}}\otimes\wedge^2
\widetilde{E_\mathbf{C}}
+2\wedge^4\widetilde{E_\mathbf{C}}-2\widetilde{E_\mathbf{C}}\otimes\wedge^3
\widetilde{E_\mathbf{C}}\\
&+2\widetilde{E_\mathbf{C}}\otimes\wedge^2\widetilde{E_\mathbf{C}}
-\widetilde{E_\mathbf{C}}\otimes\widetilde{E_\mathbf{C}}\otimes\widetilde{E_\mathbf{C}}
+\widetilde{E_\mathbf{C}}+\wedge^2\widetilde{E_\mathbf{C}}),g,d)]\}^{(19)}\\
=&-117288[\widehat{A}(TM,\nabla^{TM}){\rm ch}(\triangle(M)){\rm ch}(\triangle(E),g^{\triangle(E)},d)\\
&+2048\widehat{A}(TM,\nabla^{TM}){\rm ch}(\triangle(E),g^{\triangle(E)},d)].
    \end{split}
  \end{equation}
\end{thm}
\begin{proof}
  When ${\rm dim}M=19$, then $Q(\nabla^{TM},g,d,\tau)$ is a modular form over $SL_2({\bf Z})$ with the weight $10$, so
\begin{equation}
  Q(\nabla^{TM},g,d,\tau)=\lambda E_4(\tau)E_6(\tau),
\end{equation}
where $\lambda$ is degree 19 forms. By (3.25) and (3.26), we have
\begin{equation}
E_4(\tau)E_6(\tau)=1-264q-117288q^2+\cdots.
\end{equation}
When ${\rm dim}M=19$, direct computations show that
\begin{equation}
\begin{split}
Q(\nabla^{TM},g,d,\tau)=&[{\widehat{A}(TM,\nabla^{TM})}{\rm ch}(\triangle(M)){\rm ch}(\triangle(E),g^{\triangle(E)},d)\\
&+2048\widehat{A}(TM,\nabla^{TM}){\rm ch}(\triangle(E),g^{\triangle(E)},d)]^{(19)}\\
&+q\{\widehat{A}(TM,\nabla^{TM})[{\rm ch}(\triangle(M)\otimes 2\widetilde{T_\mathbf{C}M}){\rm ch}(\triangle(E),g^{\triangle(E)},d)\\
&+{\rm ch}(\triangle(M)){\rm ch}(\triangle(E)\otimes(2\wedge^2\widetilde{E_\mathbf{C}}-\widetilde
{E_\mathbf{C}}\otimes\widetilde{E_\mathbf{C}}+\widetilde{E_\mathbf{C}}),g,d)]\\
&+2048\widehat{A}(TM,\nabla^{TM})[{\rm ch}(\widetilde{T_\mathbf{C}M}+\wedge^2\widetilde{T_\mathbf{C}M}){\rm ch}(\triangle(E),g^{\triangle(E)},d)\\
&+{\rm ch}(\triangle(E)\otimes(2\wedge^2\widetilde{E_\mathbf{C}}
-\widetilde{E_\mathbf{C}}\otimes\widetilde{E_\mathbf{C}}+\widetilde{E_\mathbf{C}}),g,d)]\}^{(19)}\\
&+q^2\{{\widehat{A}(TM,\nabla^{TM})}[{\rm ch}(\triangle(M)\otimes(2\widetilde{T_\mathbf{C}M}+\wedge^2\widetilde{T_\mathbf{C}M}\\
&+\widetilde{T_\mathbf{C}M}\otimes\widetilde{T_\mathbf{C}M}+S^2\widetilde{T_\mathbf{C}M})){\rm ch}(\triangle(E),g^{\triangle(E)},d)\\
&+{\rm ch}(\triangle(M)\otimes2\widetilde
{T_{\mathbf{C}}M}){\rm ch}(\triangle(E)\otimes(2\wedge^2\widetilde{E_\mathbf{C}}-\widetilde
{E_\mathbf{C}}\otimes\widetilde{E_\mathbf{C}}+\widetilde{E_\mathbf{C}}),g,d)\\
&+{\rm ch}(\triangle(M)){\rm ch}(\triangle(E)\otimes(\wedge^2\widetilde{E_\mathbf{C}}\otimes\wedge^2\widetilde{E_\mathbf{C}}
+2\wedge^4\widetilde
{E_\mathbf{C}}-2\widetilde{E_\mathbf{C}}\otimes\wedge^3
\widetilde{E_\mathbf{C}}\\
&+2\widetilde{E_\mathbf{C}}\otimes\wedge^2\widetilde{E_\mathbf{C}}
-\widetilde{E_\mathbf{C}}\otimes\widetilde{E_\mathbf{C}}\otimes\widetilde{E_\mathbf{C}}
+\widetilde{E_\mathbf{C}}+\wedge^2\widetilde{E_\mathbf{C}}),g,d)]\\
&+2048\widehat{A}(TM,\nabla^{TM})[{\rm ch}(\wedge^4\widetilde{T_\mathbf{C}M}+\wedge^2\widetilde{T_\mathbf{C}M}\otimes
\widetilde{T_\mathbf{C}M}\\
&+\widetilde{T_\mathbf{C}M}\otimes\widetilde{T_\mathbf{C}M}+S^2
\widetilde{T_\mathbf{C}M}+\widetilde{T_\mathbf{C}M}){\rm ch}(\triangle(E),g^{\triangle(E)},d)\\
&+{\rm ch}(\widetilde{T_{
\mathbf{C}}M}+\wedge^2\widetilde{T_{\mathbf{C}}M}){\rm ch}(\triangle(E)\otimes(2\wedge^2\widetilde{E_\mathbf{C}}-\widetilde
{E_\mathbf{C}}\otimes\widetilde{E_\mathbf{C}}+\widetilde{E_\mathbf{C}}),g,d)\\
&+{\rm ch}(\triangle(E)\otimes(\wedge^2\widetilde{E_\mathbf{C}}\otimes\wedge^2
\widetilde{E_\mathbf{C}}
+2\wedge^4\widetilde{E_\mathbf{C}}-2\widetilde{E_\mathbf{C}}\otimes\wedge^3
\widetilde{E_\mathbf{C}}\\
&+2\widetilde{E_\mathbf{C}}\otimes\wedge^2\widetilde{E_\mathbf{C}}
-\widetilde{E_\mathbf{C}}\otimes\widetilde{E_\mathbf{C}}\otimes\widetilde{E_\mathbf{C}}
+\widetilde{E_\mathbf{C}}+\wedge^2\widetilde{E_\mathbf{C}}),g,d)]\}^{(19)}+\cdots.
\end{split}
\end{equation}
By (3.49), (3.50) and (3.51), we compare the coefficients of $1$, $q$, $q^2$. We get Theorem 3.8.
\end{proof}
\begin{cor}
Let $M$ be an $19$-dimensional spin manifold without boundary. If $c_3(E,g,d)=0$, then
\begin{equation}
\begin{split}
{\rm Ind}(T &\otimes (\triangle(M)\otimes2\widetilde{T_\mathbf{C}M}\otimes(\triangle(E),g^{\triangle(E)},d)\\
&+\triangle(M){\rm ch}(\Delta(E)\otimes(2\wedge^2\widetilde{E_\mathbf{C}}-\widetilde{E_\mathbf{C}}
\otimes\widetilde{E_\mathbf{C}}+\widetilde{E_\mathbf{C}}),g,d)))\equiv 0 ~~{\rm mod} ~~8Z,
\end{split}
\end{equation}
\begin{equation}
  \begin{split}
   {\rm Ind}((T &\otimes (\triangle(M)\otimes((2\widetilde{T_\mathbf{C}M}+\wedge^2\widetilde{T_\mathbf{C}M}+
   \widetilde{T_\mathbf{C}M}\otimes\widetilde{T_\mathbf{C}M}\\
   &+S^2\widetilde{T_\mathbf{C}M})
   \otimes(\triangle(E),g^{\triangle(E)},d)+2\widetilde{T_\mathbf{C}M}\otimes((2\wedge^2\widetilde
   {E_\mathbf{C}}\\
   &-\widetilde{E_\mathbf{C}}\otimes\widetilde{E_\mathbf{C}}
   +\widetilde{E_\mathbf{C}}),g,d)+(\triangle(E)\otimes(\wedge^2\widetilde{E_\mathbf{C}}
   \otimes\wedge^2\widetilde{E_\mathbf{C}}\\
   &+2\wedge^4\widetilde{E_\mathbf{C}}-2\widetilde{E_\mathbf{C}}
   \otimes\wedge^3\widetilde{E_\mathbf{C}}+2\widetilde{E_\mathbf{C}}\otimes\wedge^2\widetilde
   {E_\mathbf{C}}-\widetilde{E_\mathbf{C}}\otimes\widetilde{E_\mathbf{C}}
   \otimes\widetilde{E_\mathbf{C}}\\
   &+\widetilde{E_\mathbf{C}}+\wedge^2\widetilde{E_\mathbf{C}}),g,d)))_+)\equiv 0 ~~{\rm mod} ~~8Z.
  \end{split}
\end{equation}
\end{cor}
\begin{thm}
  When ${\rm dim}M=23$ and $c_3(E,g,d)=0$, we have
    \begin{equation}
    \begin{split}
    \{\widehat{A}&(TM,\nabla^{TM})[{\rm ch}(\triangle(M)\otimes(2\widetilde{T_\mathbf{C}M}+\wedge^2\widetilde{T_\mathbf{C}M}\\
&+\widetilde{T_\mathbf{C}M}\otimes\widetilde{T_\mathbf{C}M}+S^2\widetilde{T_\mathbf{C}M})){\rm ch}(\triangle(E),g^{\triangle(E)},d)\\
&+{\rm ch}(\triangle(M)\otimes2\widetilde
{T_{\mathbf{C}}M}){\rm ch}(\triangle(E)\otimes(2\wedge^2\widetilde{E_\mathbf{C}}-\widetilde
{E_\mathbf{C}}\otimes\widetilde{E_\mathbf{C}}+\widetilde{E_\mathbf{C}}),g,d)\\
&+{\rm ch}(\triangle(M)){\rm ch}(\triangle(E)\otimes(\wedge^2\widetilde{E_\mathbf{C}}\otimes\wedge^2\widetilde{E_\mathbf{C}}
+2\wedge^4\widetilde
{E_\mathbf{C}}-2\widetilde{E_\mathbf{C}}\otimes\wedge^3
\widetilde{E_\mathbf{C}}\\
&+2\widetilde{E_\mathbf{C}}\otimes\wedge^2\widetilde{E_\mathbf{C}}
-\widetilde{E_\mathbf{C}}\otimes\widetilde{E_\mathbf{C}}\otimes\widetilde{E_\mathbf{C}}
+\widetilde{E_\mathbf{C}}+\wedge^2\widetilde{E_\mathbf{C}}),g,d)]\\
&+8192\widehat{A}(TM,\nabla^{TM})[{\rm ch}(\wedge^4\widetilde{T_\mathbf{C}M}+\wedge^2\widetilde{T_\mathbf{C}M}\otimes
\widetilde{T_\mathbf{C}M}\\
&+\widetilde{T_\mathbf{C}M}\otimes\widetilde{T_\mathbf{C}M}+S^2
\widetilde{T_\mathbf{C}M}+\widetilde{T_\mathbf{C}M}){\rm ch}(\triangle(E),g^{\triangle(E)},d)\\
&+{\rm ch}(\widetilde{T_{
\mathbf{C}}M}+\wedge^2\widetilde{T_{\mathbf{C}}M}){\rm ch}(\triangle(E)\otimes(2\wedge^2\widetilde{E_\mathbf{C}}-\widetilde
{E_\mathbf{C}}\otimes\widetilde{E_\mathbf{C}}+\widetilde{E_\mathbf{C}}),g,d)\\
&+{\rm ch}(\triangle(E)\otimes(\wedge^2\widetilde{E_\mathbf{C}}\otimes\wedge^2
\widetilde{E_\mathbf{C}}
+2\wedge^4\widetilde{E_\mathbf{C}}-2\widetilde{E_\mathbf{C}}\otimes\wedge^3
\widetilde{E_\mathbf{C}}\\
&+2\widetilde{E_\mathbf{C}}\otimes\wedge^2\widetilde{E_\mathbf{C}}
-\widetilde{E_\mathbf{C}}\otimes\widetilde{E_\mathbf{C}}\otimes\widetilde{E_\mathbf{C}}
+\widetilde{E_\mathbf{C}}+\wedge^2\widetilde{E_\mathbf{C}}),g,d)]\}^{(23)}\\
=&\{196560[\widehat{A}(TM,\nabla^{TM}){\rm ch}(\triangle(M)){\rm ch}(\triangle(E),g^{\triangle(E)},d)\\
&+8192\widehat{A}(TM,\nabla^{TM}){\rm ch}(\triangle(E),g^{\triangle(E)},d)]\\
&-24\{\widehat{A}(TM,\nabla^{TM})[{\rm ch}(\triangle(M)\otimes 2\widetilde{T_{\mathbf{C}}M}){\rm ch}(\triangle(E),g^{\triangle(E)},d)\\
    &+{\rm ch}(\triangle(M)){\rm ch}(\Delta(E)\otimes(2\wedge^2\widetilde{E_\mathbf{C}}-\widetilde{E_\mathbf{C}}
    \otimes\widetilde{E_\mathbf{C}}+\widetilde{E_\mathbf{C}}),g,d)]\\
    &+8192\widehat{A}(TM,\nabla^{TM})[{\rm ch}(\widetilde{T_{\mathbf{C}}M}+\wedge^2\widetilde{T_{\mathbf{C}}M}){\rm ch}(\triangle(E),g^{\triangle(E)},d)\\
    &+{\rm ch}(\Delta(E)\otimes(2\wedge^2\widetilde{E_\mathbf{C}}-\widetilde{E_\mathbf{C}}
    \otimes\widetilde{E_\mathbf{C}}+\widetilde{E_\mathbf{C}}),g,d)]\}\}^{(23)}
    \end{split}
  \end{equation}
\end{thm}
\begin{proof}
  When ${\rm dim}M=23$, then $Q(\nabla^{TM},g,d,\tau)$ is a modular form over $SL_2({\bf Z})$ with the weight $12$, so
\begin{equation}
  Q(\nabla^{TM},g,d,\tau)=\lambda_1E_4(\tau)^3+\lambda_2E_6(\tau)^2,
\end{equation}
where $\lambda_1$, $\lambda_2$ is degree 23 forms. By (3.25) and (3.26), we have
\begin{equation}
E_4(\tau)^3=1+720q+179280q^2+\cdots,
\end{equation}
\begin{equation}
E_6(\tau)^2=1-1008q+220752q^2+\cdots.
\end{equation}
When ${\rm dim}M=23$, direct computations show that
\begin{equation}
\begin{split}
Q(\nabla^{TM},g,d,\tau)=&[{\widehat{A}(TM,\nabla^{TM})}{\rm ch}(\triangle(M)){\rm ch}(\triangle(E),g^{\triangle(E)},d)\\
&+8192\widehat{A}(TM,\nabla^{TM}){\rm ch}(\triangle(E),g^{\triangle(E)},d)]^{(23)}\\
&+q\{\widehat{A}(TM,\nabla^{TM})[{\rm ch}(\triangle(M)\otimes 2\widetilde{T_\mathbf{C}M}){\rm ch}(\triangle(E),g^{\triangle(E)},d)\\
&+{\rm ch}(\triangle(M)){\rm ch}(\triangle(E)\otimes(2\wedge^2\widetilde{E_\mathbf{C}}-\widetilde
{E_\mathbf{C}}\otimes\widetilde{E_\mathbf{C}}+\widetilde{E_\mathbf{C}}),g,d)]\\
&+8192\widehat{A}(TM,\nabla^{TM})[{\rm ch}(\widetilde{T_\mathbf{C}M}+\wedge^2\widetilde{T_\mathbf{C}M}){\rm ch}(\triangle(E),g^{\triangle(E)},d)\\
&+{\rm ch}(\triangle(E)\otimes(2\wedge^2\widetilde{E_\mathbf{C}}
-\widetilde{E_\mathbf{C}}\otimes\widetilde{E_\mathbf{C}}+\widetilde{E_\mathbf{C}}),g,d)]\}^{(23)}\\
&+q^2\{{\widehat{A}(TM,\nabla^{TM})}[{\rm ch}(\triangle(M)\otimes(2\widetilde{T_\mathbf{C}M}+\wedge^2\widetilde{T_\mathbf{C}M}\\
&+\widetilde{T_\mathbf{C}M}\otimes\widetilde{T_\mathbf{C}M}+S^2\widetilde{T_\mathbf{C}M})){\rm ch}(\triangle(E),g^{\triangle(E)},d)\\
&+{\rm ch}(\triangle(M)\otimes2\widetilde
{T_{\mathbf{C}}M}){\rm ch}(\triangle(E)\otimes(2\wedge^2\widetilde{E_\mathbf{C}}-\widetilde
{E_\mathbf{C}}\otimes\widetilde{E_\mathbf{C}}+\widetilde{E_\mathbf{C}}),g,d)\\
&+{\rm ch}(\triangle(M)){\rm ch}(\triangle(E)\otimes(\wedge^2\widetilde{E_\mathbf{C}}\otimes\wedge^2\widetilde{E_\mathbf{C}}
+2\wedge^4\widetilde
{E_\mathbf{C}}-2\widetilde{E_\mathbf{C}}\otimes\wedge^3
\widetilde{E_\mathbf{C}}\\
&+2\widetilde{E_\mathbf{C}}\otimes\wedge^2\widetilde{E_\mathbf{C}}
-\widetilde{E_\mathbf{C}}\otimes\widetilde{E_\mathbf{C}}\otimes\widetilde{E_\mathbf{C}}
+\widetilde{E_\mathbf{C}}+\wedge^2\widetilde{E_\mathbf{C}}),g,d)]\\
&+8192\widehat{A}(TM,\nabla^{TM})[{\rm ch}(\wedge^4\widetilde{T_\mathbf{C}M}+\wedge^2\widetilde{T_\mathbf{C}M}\otimes
\widetilde{T_\mathbf{C}M}\\
&+\widetilde{T_\mathbf{C}M}\otimes\widetilde{T_\mathbf{C}M}+S^2
\widetilde{T_\mathbf{C}M}+\widetilde{T_\mathbf{C}M}){\rm ch}(\triangle(E),g^{\triangle(E)},d)\\
&+{\rm ch}(\widetilde{T_{
\mathbf{C}}M}+\wedge^2\widetilde{T_{\mathbf{C}}M}){\rm ch}(\triangle(E)\otimes(2\wedge^2\widetilde{E_\mathbf{C}}-\widetilde
{E_\mathbf{C}}\otimes\widetilde{E_\mathbf{C}}+\widetilde{E_\mathbf{C}}),g,d)\\
&+{\rm ch}(\triangle(E)\otimes(\wedge^2\widetilde{E_\mathbf{C}}\otimes\wedge^2
\widetilde{E_\mathbf{C}}
+2\wedge^4\widetilde{E_\mathbf{C}}-2\widetilde{E_\mathbf{C}}\otimes\wedge^3
\widetilde{E_\mathbf{C}}\\
&+2\widetilde{E_\mathbf{C}}\otimes\wedge^2\widetilde{E_\mathbf{C}}
-\widetilde{E_\mathbf{C}}\otimes\widetilde{E_\mathbf{C}}\otimes\widetilde{E_\mathbf{C}}
+\widetilde{E_\mathbf{C}}+\wedge^2\widetilde{E_\mathbf{C}}),g,d)]\}^{(23)}+\cdots.
\end{split}
\end{equation}
In (3.58), we compare the coefficients of $1$, $q$, $q^2$, we get three equations about $\lambda_1$, $\lambda_2$. By (3.59), (3.60) and (3.61), we get Theorem 3.10.
\end{proof}
\begin{cor}
Let $M$ be an $23$-dimensional spin manifold without boundary. If $c_3(E,g,d)=0$, then
\begin{equation}
  \begin{split}
   {\rm Ind}(T &\otimes (\triangle(M)\otimes((2\widetilde{T_\mathbf{C}M}+\wedge^2\widetilde{T_\mathbf{C}M}+
   \widetilde{T_\mathbf{C}M}\otimes\widetilde{T_\mathbf{C}M}\\
   &+S^2\widetilde{T_\mathbf{C}M})
   \otimes(\triangle(E),g^{\triangle(E)},d)+2\widetilde{T_\mathbf{C}M}\otimes((2\wedge^2\widetilde
   {E_\mathbf{C}}\\
   &-\widetilde{E_\mathbf{C}}\otimes\widetilde{E_\mathbf{C}}
   +\widetilde{E_\mathbf{C}}),g,d)+(\triangle(E)\otimes(\wedge^2\widetilde{E_\mathbf{C}}
   \otimes\wedge^2\widetilde{E_\mathbf{C}}\\
   &+2\wedge^4\widetilde{E_\mathbf{C}}-2\widetilde{E_\mathbf{C}}
   \otimes\wedge^3\widetilde{E_\mathbf{C}}+2\widetilde{E_\mathbf{C}}\otimes\wedge^2\widetilde
   {E_\mathbf{C}}-\widetilde{E_\mathbf{C}}\otimes\widetilde{E_\mathbf{C}}
   \otimes\widetilde{E_\mathbf{C}}\\
   &+\widetilde{E_\mathbf{C}}+\wedge^2\widetilde{E_\mathbf{C}}),g,d))))\equiv 0 ~~{\rm mod} ~~16Z.
  \end{split}
\end{equation}
\end{cor}

Let $M$ be closed oriented ${\rm spin^{c}}$-manifold and $L$ be the complex line bundle associated to the given ${\rm spin^{c}}$ structure on $M.$ Denote by $c=c_1(L)$ the first Chern class of $L.$ Also, we use $L_{\bf{R}}$ for the notation of $L,$ when it is viewed as an oriented real plane bundle.
Let $\Theta(T_{\mathbf{C}}M,L_{\bf{R}}\otimes\bf{C})$ be the virtual complex vector bundle over $M$ defined by
\begin{equation}
    \begin{split}
        \Theta(T_{\mathbf{C}}M,L_{\bf{R}}\otimes\mathbf{C})=&\bigotimes _{n=1}^{\infty}S_{q^n}(\widetilde{T_{\mathbf{C}}M})\otimes
\bigotimes _{m=1}^{\infty}\wedge_{q^m}(\widetilde{L_{\bf{R}}\otimes\mathbf{C}})\\
&\otimes
\bigotimes _{r=1}^{\infty}\wedge_{-q^{r-\frac{1}{2}}}(\widetilde{L_{\bf{R}}\otimes\mathbf{C}})\otimes
\bigotimes _{s=1}^{\infty}\wedge_{q^{s-\frac{1}{2}}}(\widetilde{L_{\bf{R}}\otimes\mathbf{C}}).\nonumber
    \end{split}
\end{equation}
Let ${\rm dim}M=4k-1$ and $y=-\frac{\sqrt{-1}}{2\pi}c$. Set
\begin{equation}
\begin{split}
\widetilde{Q}(&\nabla^{TM},\nabla^{L},g,d,\tau)\\
=&\{\widehat{A}(TM,\nabla^{TM}){\rm exp}(\frac{c}{2}){\rm ch}(\Theta(T_{\mathbf{C}}M,L_{\bf{R}}\otimes\mathbf{C})){\rm ch}(Q(E),g^{Q(E)},d,\tau)\}^{(4k-1)}.
\end{split}
\end{equation}
When $c=0,$~$\Theta(T_{\mathbf{C}}M)=\bigotimes _{n=1}^{\infty}S_{q^n}(\widetilde{T_{\mathbf{C}}M})$ be the Witten bundle over $M$. And
$$\widetilde{Q}(\nabla^{TM},g,d,\tau)=\{\widehat{A}(TM,\nabla^{TM}){\rm ch}(\Theta(T_{\mathbf{C}}M)){\rm ch}(Q(E),g^{Q(E)},d,\tau)\}^{(4k-1)}$$
is called the Witten form in odd dimensions.\\
Then
\begin{equation}
\widetilde{Q}(M,L,\tau)=\left(\prod_{j=1}^{2k-1}\frac{x_j\theta'(0,\tau)}{\theta(x_j,\tau)}
\left(\frac{\theta_1(y,\tau)}{\theta_1(0,\tau)}\frac{\theta_2(y,\tau)}
{\theta_2(0,\tau)}\frac{\theta_3(y,\tau)}{\theta_3(0,\tau)}
\right)\right)^{(4p)},
\end{equation}
and
\begin{equation}
  \widetilde{Q}(\nabla^{TM},\nabla^{L},g,d,\tau)=(\widetilde{Q}(M,L,\tau)\cdot{\rm ch}(Q(E),g^{Q(E)},d,\tau))^{(4k-1)}.
\end{equation}

Let $p_1$ denote the first Pontryagin class. By (2.13)-(2.17), we have $\widetilde{Q}(\nabla^{TM},\nabla^{L},g,d,\tau+1)=\widetilde{Q}(\nabla^{TM},\nabla^{L},g,d,\tau)$ and $\widetilde{Q}(\nabla^{TM},\nabla^{L},g,d,-\frac{1}{\tau})=\tau^{2k}\widetilde{Q}(\nabla^{TM},
\nabla^{L},g,d,\tau)$ if $3p_1(L)-p_1(M)=0$ and $c_3(E,g,d)=0$. So
\begin{thm}
Let ${\rm dim}M=4k-1$. If $3p_1(L)-p_1(M)=0$ and $c_3(E,g,d)=0$, then $\widetilde{Q}(\nabla^{TM},\nabla^{L},g,d,\tau)$ is a modular form over $SL_2({\bf Z})$ with the weight $2k$.
\end{thm}
\noindent{\bf Remark.}
  Let $M$ be a $(4k-1)$-dimensional spin manifold. If $p_1(M)=0$ and $c_3(E,g,d)=0$, then
  $\widetilde{Q}(\nabla^{TM},g,d,\tau)$ is a modular form over $SL_2({\bf Z})$ with the weight $2k$.

Direct computation show
\begin{equation}
\begin{split}
    \widetilde{Q}(&\nabla^{TM},\nabla^{L},g,d,\tau)\\
    =&[{\widehat{A}(TM,\nabla^{TM})}{\rm exp}(\frac{c}{2}){\rm ch}(\triangle(E),g^{\triangle(E)},d)]^{4k-1}\\
    &+q\{{\widehat{A}(TM,\nabla^{TM})}{\rm exp}(\frac{c}{2})[{\rm ch}(\widetilde{T_\mathbf{C}M}){\rm ch}(\triangle(E),g^{\triangle(E)},d)\\
    &+{\rm ch}(2\wedge^2\widetilde{L_{\bf{R}}\otimes\mathbf{C}}-(\widetilde{L_{\bf{R}}\otimes\mathbf{C}})
    \otimes(\widetilde{L_{\bf{R}}\otimes\mathbf{C}})
    +\widetilde{L_{\bf{R}}\otimes\mathbf{C}}){\rm ch}(\triangle(E),g^{\triangle(E)},d)\\
    &+{\rm ch}(\triangle(E)\otimes(2\wedge^2\widetilde{E_\mathbf{C}}-\widetilde
{E_\mathbf{C}}\otimes\widetilde{E_\mathbf{C}}+\widetilde{E_\mathbf{C}}),g,d)]\}^{4k-1}\\
&+q^2\{{\widehat{A}(TM,\nabla^{TM})}{\rm exp}(\frac{c}{2})[{\rm ch}(S^2\widetilde{T_\mathbf{C}M}+\widetilde{T_\mathbf{C}M}){\rm ch}(\triangle(E),g^{\triangle(E)},d)\\
&+{\rm ch}(\wedge^2\widetilde{L_{\bf{R}}\otimes\mathbf{C}}\otimes\wedge^2\widetilde{L_{\bf{R}}
\otimes\mathbf{C}}+2\wedge^4\widetilde{L_{\bf{R}}\otimes\mathbf{C}}\\
&-2\widetilde{L_{\bf{R}}
\otimes\mathbf{C}}\otimes\wedge^3\widetilde{L_{\bf{R}}
\otimes\mathbf{C}}
+2\widetilde{L_{\bf{R}}
\otimes\mathbf{C}}\otimes\wedge^2\widetilde{L_{\bf{R}}
\otimes\mathbf{C}}\\
&-\widetilde{L_{\bf{R}}
\otimes\mathbf{C}}\otimes\widetilde{L_{\bf{R}}
\otimes\mathbf{C}}\otimes\widetilde{L_{\bf{R}}
\otimes\mathbf{C}}+\widetilde{L_{\bf{R}}
\otimes\mathbf{C}}+\wedge^2\widetilde{L_{\bf{R}}\otimes\mathbf{C}}){\rm ch}(\triangle(E),g^{\triangle(E)},d)\\
&+{\rm ch}(\widetilde{T_\mathbf{C}M}\otimes(2\wedge^2\widetilde{L_{\bf{R}}
\otimes\mathbf{C}}-(\widetilde{L_{\bf{R}}\otimes\mathbf{C}})
    \otimes(\widetilde{L_{\bf{R}}\otimes\mathbf{C}})\\
    &+\widetilde{L_{\bf{R}}\otimes\mathbf{C}}){\rm ch}(\triangle(E),g^{\triangle(E)},d))
+{\rm ch}(2\wedge^2\widetilde{L_{\bf{R}}\otimes\mathbf{C}}-(\widetilde{L_{\bf{R}}\otimes\mathbf{C}})
    \otimes(\widetilde{L_{\bf{R}}\otimes\mathbf{C}})\\
    &+\widetilde{L_{\bf{R}}\otimes\mathbf{C}}){\rm ch}(\triangle(E)\otimes(2\wedge^2\widetilde{E_\mathbf{C}}-\widetilde
{E_\mathbf{C}}\otimes\widetilde{E_\mathbf{C}}+\widetilde{E_\mathbf{C}}),g,d)\\
&+{\rm ch}(\triangle(E)\otimes(2\wedge^2\widetilde{E_\mathbf{C}}-\widetilde
{E_\mathbf{C}}\otimes\widetilde{E_\mathbf{C}}+\widetilde{E_\mathbf{C}}),g,d)\\
&+{\rm ch}(\triangle(E)\otimes(\wedge^2\widetilde{E_\mathbf{C}}\otimes\wedge^2\widetilde{E_\mathbf{C}}
+2\wedge^4\widetilde{E_\mathbf{C}}-2\widetilde{E_\mathbf{C}}\otimes\wedge^3\widetilde{E_\mathbf{C}}\\
&+2\widetilde{E_\mathbf{C}}\otimes\wedge^2\widetilde{E_\mathbf{C}}-\widetilde{E_\mathbf{C}}\otimes
\widetilde{E_\mathbf{C}}\otimes\widetilde{E_\mathbf{C}}+\widetilde{E_\mathbf{C}}
+\wedge^2\widetilde{E_\mathbf{C}}),g,d)]\}^{4k-1}+\cdots.
\end{split}
\end{equation}

When $dim X=7$, then $\widetilde{Q}(\nabla^{TM},\nabla^{L},g,d,\tau)$ is a modular form over
$SL_2({\bf Z})$ with the weight $4$ and $\widetilde{Q}(\nabla^{TM},\nabla^{L},g,d,\tau)=\lambda
E_4(\tau)$.When $dim X=11$, then $\widetilde{Q}(\nabla^{TM},\nabla^{L},g,d,\tau)$ is a modular form
over $SL_2({\bf Z})$ with the weight $6$ and $\widetilde{Q}(\nabla^{TM},\nabla^{L},g,d,\tau)=\lambda E_6(\tau)$.When $dim X=15$, then $\widetilde{Q}(\nabla^{TM},\nabla^{L},g,d,\tau)$ is a modular form
over $SL_2({\bf Z})$ with the weight $8$ and $\widetilde{Q}(\nabla^{TM},\nabla^{L},g,d,\tau)=\lambda E_4(\tau)^2$. When $dim X=19$, then $\widetilde{Q}(\nabla^{TM},\nabla^{L},g,d,\tau)$ is a modular
form over $SL_2({\bf Z})$ with the weight $10$ and
$\widetilde{Q}(\nabla^{TM},\nabla^{L},g,d,\tau)=\lambda E_4(\tau)E_6(\tau)$. When $dim X=23$, then
$\widetilde{Q}(\nabla^{TM},\nabla^{L},g,d,\tau)$ is a modular form over $SL_2({\bf Z})$ with the
weight $12$ and
$\widetilde{Q}(\nabla^{TM},\nabla^{L},g,d,\tau)=\lambda_1E_4(\tau)^3+\lambda_2E_6(\tau)^2$.
So. we get the following theorem.

\begin{thm}
Let ${\rm dim}M=7$. If $3p_1(L)-p_1(M)=0$ and $c_3(E,g,d)=0$, we have
\begin{equation}
  \begin{split}
    \{\widehat{A}(&TM,\nabla^{TM}){\rm exp}(\frac{c}{2})[{\rm ch}(\widetilde{T_\mathbf{C}M}+2\wedge^2\widetilde{L_{\bf{R}}\otimes\mathbf{C}}\\
&-(\widetilde{L_{\bf{R}}\otimes\mathbf{C}})
    \otimes(\widetilde{L_{\bf{R}}\otimes\mathbf{C}})
    +\widetilde{L_{\bf{R}}\otimes\mathbf{C}})
{\rm ch}(\triangle(E),g^{\triangle(E)},d)\\
    &+{\rm ch}(\triangle(E)
    \otimes(2\wedge^2\widetilde{E_\mathbf{C}}-\widetilde
{E_\mathbf{C}}\otimes\widetilde{E_\mathbf{C}}+\widetilde{E_\mathbf{C}}),g,d)]\}^{(7)}\\
=&240\{{\widehat{A}(TM,\nabla^{TM})}{\rm exp}(\frac{c}{2}){\rm ch}(\triangle(E),g^{\triangle(E)},d)\}^{(7)},
  \end{split}
\end{equation}
\begin{equation}
  \begin{split}
   \{\widehat{A}(&TM,\nabla^{TM}){\rm exp}(\frac{c}{2})[{\rm ch}(S^2\widetilde{T_\mathbf{C}M}+\widetilde{T_\mathbf{C}M}+\wedge^2\widetilde{L_{\bf{R}}
   \otimes\mathbf{C}}\otimes\wedge^2\widetilde{L_{\bf{R}}
\otimes\mathbf{C}}\\
&+2\wedge^4\widetilde{L_{\bf{R}}\otimes\mathbf{C}}
-2\widetilde{L_{\bf{R}}
\otimes\mathbf{C}}\otimes\wedge^3\widetilde{L_{\bf{R}}
\otimes\mathbf{C}}
+2\widetilde{L_{\bf{R}}
\otimes\mathbf{C}}\otimes\wedge^2\widetilde{L_{\bf{R}}
\otimes\mathbf{C}}\\
&-\widetilde{L_{\bf{R}}
\otimes\mathbf{C}}\otimes\widetilde{L_{\bf{R}}
\otimes\mathbf{C}}\otimes\widetilde{L_{\bf{R}}
\otimes\mathbf{C}}+\widetilde{L_{\bf{R}}
\otimes\mathbf{C}}+\wedge^2\widetilde{L_{\bf{R}}\otimes\mathbf{C}}\\
&+\widetilde{T_\mathbf{C}M}\otimes(2\wedge^2\widetilde{L_{\bf{R}}
\otimes\mathbf{C}}-(\widetilde{L_{\bf{R}}\otimes\mathbf{C}})
    \otimes(\widetilde{L_{\bf{R}}\otimes\mathbf{C}})
    +\widetilde{L_{\bf{R}}\otimes\mathbf{C}}))\\
&\cdot{\rm ch}(\triangle(E),g^{\triangle(E)},d)
+{\rm ch}(2\wedge^2\widetilde{L_{\bf{R}}\otimes\mathbf{C}}-(\widetilde{L_{\bf{R}}\otimes\mathbf{C}})
    \otimes(\widetilde{L_{\bf{R}}\otimes\mathbf{C}})\\
    &+\widetilde{L_{\bf{R}}\otimes\mathbf{C}}){\rm ch}(\triangle(E)\otimes(2\wedge^2\widetilde{E_\mathbf{C}}-\widetilde
{E_\mathbf{C}}\otimes\widetilde{E_\mathbf{C}}+\widetilde{E_\mathbf{C}}),g,d)\\
&+{\rm ch}(\triangle(E)\otimes(2\wedge^2\widetilde{E_\mathbf{C}}-\widetilde
{E_\mathbf{C}}\otimes\widetilde{E_\mathbf{C}}+\widetilde{E_\mathbf{C}}),g,d)\\
&+{\rm ch}(\triangle(E)\otimes(\wedge^2\widetilde{E_\mathbf{C}}\otimes\wedge^2\widetilde{E_\mathbf{C}}
+2\wedge^4\widetilde{E_\mathbf{C}}-2\widetilde{E_\mathbf{C}}\otimes\wedge^3\widetilde{E_\mathbf{C}}\\
&+2\widetilde{E_\mathbf{C}}\otimes\wedge^2\widetilde{E_\mathbf{C}}-\widetilde{E_\mathbf{C}}\otimes
\widetilde{E_\mathbf{C}}\otimes\widetilde{E_\mathbf{C}}+\widetilde{E_\mathbf{C}}
+\wedge^2\widetilde{E_\mathbf{C}}),g,d)]\}^{(7)}\\
=&2160\{{\widehat{A}(TM,\nabla^{TM})}{\rm exp}(\frac{c}{2}){\rm ch}(\triangle(E),g^{\triangle(E)},d)\}^{(7)}.
  \end{split}
\end{equation}
\end{thm}
\begin{cor}
Let $M$ be an $7$-dimensional $ spin^c$ manifold without boundary. If $c_3(E,g,d)=0$, then
\begin{equation}
\begin{split}
{\rm Ind}(T^c &\otimes ((\widetilde{T_\mathbf{C}M}+2\wedge^2\widetilde{L_{\bf{R}}\otimes\mathbf{C}}
-(\widetilde{L_{\bf{R}}\otimes\mathbf{C}})
    \otimes(\widetilde{L_{\bf{R}}\otimes\mathbf{C}})\\
    &+\widetilde{L_{\bf{R}}\otimes\mathbf{C}})
\otimes(\triangle(E),g^{\triangle(E)},d)+(\triangle(E)\otimes(2\wedge^2\widetilde{E_\mathbf{C}}\\
&-\widetilde
{E_\mathbf{C}}\otimes\widetilde{E_\mathbf{C}}+\widetilde{E_\mathbf{C}}),g,d)))\equiv 0 ~~{\rm mod} ~~240Z,
\end{split}
\end{equation}
\begin{equation}
  \begin{split}
   {\rm Ind}(T &\otimes ((S^2\widetilde{T_\mathbf{C}M}+\widetilde{T_\mathbf{C}M}+\wedge^2\widetilde{L_{\bf{R}}
   \otimes\mathbf{C}}\otimes\wedge^2\widetilde{L_{\bf{R}}
\otimes\mathbf{C}}\\
&+2\wedge^4\widetilde{L_{\bf{R}}\otimes\mathbf{C}}
-2\widetilde{L_{\bf{R}}
\otimes\mathbf{C}}\otimes\wedge^3\widetilde{L_{\bf{R}}
\otimes\mathbf{C}}
+2\widetilde{L_{\bf{R}}
\otimes\mathbf{C}}\otimes\wedge^2\widetilde{L_{\bf{R}}
\otimes\mathbf{C}}\\
&-\widetilde{L_{\bf{R}}
\otimes\mathbf{C}}\otimes\widetilde{L_{\bf{R}}
\otimes\mathbf{C}}\otimes\widetilde{L_{\bf{R}}
\otimes\mathbf{C}}+\widetilde{L_{\bf{R}}
\otimes\mathbf{C}}+\wedge^2\widetilde{L_{\bf{R}}\otimes\mathbf{C}}\\
&+\widetilde{T_\mathbf{C}M}\otimes(2\wedge^2\widetilde{L_{\bf{R}}
\otimes\mathbf{C}}-(\widetilde{L_{\bf{R}}\otimes\mathbf{C}})
    \otimes(\widetilde{L_{\bf{R}}\otimes\mathbf{C}})
    +\widetilde{L_{\bf{R}}\otimes\mathbf{C}}))\\
    &\otimes(\triangle(E),g^{\triangle(E)},d)
    +(2\wedge^2\widetilde{L_{\bf{R}}\otimes\mathbf{C}}-(\widetilde{L_{\bf{R}}\otimes\mathbf{C}})
    \otimes(\widetilde{L_{\bf{R}}\otimes\mathbf{C}})\\
    &+\widetilde{L_{\bf{R}}\otimes\mathbf{C}})\otimes(\triangle(E)\otimes(2\wedge^2\widetilde
    {E_\mathbf{C}}-\widetilde
{E_\mathbf{C}}\otimes\widetilde{E_\mathbf{C}}+\widetilde{E_\mathbf{C}}),g,d)\\
&+(\triangle(E)\otimes(2\wedge^2\widetilde{E_\mathbf{C}}-\widetilde
{E_\mathbf{C}}\otimes\widetilde{E_\mathbf{C}}+\widetilde{E_\mathbf{C}}),g,d)\\
&+(\triangle(E)\otimes(\wedge^2\widetilde{E_\mathbf{C}}\otimes\wedge^2\widetilde{E_\mathbf{C}}
+2\wedge^4\widetilde{E_\mathbf{C}}-2\widetilde{E_\mathbf{C}}\otimes\wedge^3\widetilde{E_\mathbf{C}}\\
&+2\widetilde{E_\mathbf{C}}\otimes\wedge^2\widetilde{E_\mathbf{C}}-\widetilde{E_\mathbf{C}}\otimes
\widetilde{E_\mathbf{C}}\otimes\widetilde{E_\mathbf{C}}+\widetilde{E_\mathbf{C}}
+\wedge^2\widetilde{E_\mathbf{C}}),g,d)))\equiv 0 ~~{\rm mod} ~~2160Z.
  \end{split}
\end{equation}
\end{cor}
\begin{thm}
Let ${\rm dim}M=11$. If $3p_1(L)-p_1(M)=0$ and $c_3(E,g,d)=0$, we have
\begin{equation}
  \begin{split}
    \{\widehat{A}(&TM,\nabla^{TM}){\rm exp}(\frac{c}{2})[{\rm ch}(\widetilde{T_\mathbf{C}M}+2\wedge^2\widetilde{L_{\bf{R}}\otimes\mathbf{C}}\\
&-(\widetilde{L_{\bf{R}}\otimes\mathbf{C}})
    \otimes(\widetilde{L_{\bf{R}}\otimes\mathbf{C}})
    +\widetilde{L_{\bf{R}}\otimes\mathbf{C}})
{\rm ch}(\triangle(E),g^{\triangle(E)},d)\\
    &+{\rm ch}(\triangle(E)
    \otimes(2\wedge^2\widetilde{E_\mathbf{C}}-\widetilde
{E_\mathbf{C}}\otimes\widetilde{E_\mathbf{C}}+\widetilde{E_\mathbf{C}}),g,d)]\}^{(11)}\\
=&-504\{{\widehat{A}(TM,\nabla^{TM})}{\rm exp}(\frac{c}{2}){\rm ch}(\triangle(E),g^{\triangle(E)},d)\}^{(11)},
  \end{split}
\end{equation}
\begin{equation}
  \begin{split}
   \{\widehat{A}(&TM,\nabla^{TM}){\rm exp}(\frac{c}{2})[{\rm ch}(S^2\widetilde{T_\mathbf{C}M}+\widetilde{T_\mathbf{C}M}+\wedge^2\widetilde{L_{\bf{R}}
   \otimes\mathbf{C}}\otimes\wedge^2\widetilde{L_{\bf{R}}
\otimes\mathbf{C}}\\
&+2\wedge^4\widetilde{L_{\bf{R}}\otimes\mathbf{C}}
-2\widetilde{L_{\bf{R}}
\otimes\mathbf{C}}\otimes\wedge^3\widetilde{L_{\bf{R}}
\otimes\mathbf{C}}
+2\widetilde{L_{\bf{R}}
\otimes\mathbf{C}}\otimes\wedge^2\widetilde{L_{\bf{R}}
\otimes\mathbf{C}}\\
&-\widetilde{L_{\bf{R}}
\otimes\mathbf{C}}\otimes\widetilde{L_{\bf{R}}
\otimes\mathbf{C}}\otimes\widetilde{L_{\bf{R}}
\otimes\mathbf{C}}+\widetilde{L_{\bf{R}}
\otimes\mathbf{C}}+\wedge^2\widetilde{L_{\bf{R}}\otimes\mathbf{C}}\\
&+\widetilde{T_\mathbf{C}M}\otimes(2\wedge^2\widetilde{L_{\bf{R}}
\otimes\mathbf{C}}-(\widetilde{L_{\bf{R}}\otimes\mathbf{C}})
    \otimes(\widetilde{L_{\bf{R}}\otimes\mathbf{C}})
    +\widetilde{L_{\bf{R}}\otimes\mathbf{C}}))\\
&\cdot{\rm ch}(\triangle(E),g^{\triangle(E)},d)
+{\rm ch}(2\wedge^2\widetilde{L_{\bf{R}}\otimes\mathbf{C}}-(\widetilde{L_{\bf{R}}\otimes\mathbf{C}})
    \otimes(\widetilde{L_{\bf{R}}\otimes\mathbf{C}})\\
    &+\widetilde{L_{\bf{R}}\otimes\mathbf{C}}){\rm ch}(\triangle(E)\otimes(2\wedge^2\widetilde{E_\mathbf{C}}-\widetilde
{E_\mathbf{C}}\otimes\widetilde{E_\mathbf{C}}+\widetilde{E_\mathbf{C}}),g,d)\\
&+{\rm ch}(\triangle(E)\otimes(2\wedge^2\widetilde{E_\mathbf{C}}-\widetilde
{E_\mathbf{C}}\otimes\widetilde{E_\mathbf{C}}+\widetilde{E_\mathbf{C}}),g,d)\\
&+{\rm ch}(\triangle(E)\otimes(\wedge^2\widetilde{E_\mathbf{C}}\otimes\wedge^2\widetilde{E_\mathbf{C}}
+2\wedge^4\widetilde{E_\mathbf{C}}-2\widetilde{E_\mathbf{C}}\otimes\wedge^3\widetilde{E_\mathbf{C}}\\
&+2\widetilde{E_\mathbf{C}}\otimes\wedge^2\widetilde{E_\mathbf{C}}-\widetilde{E_\mathbf{C}}\otimes
\widetilde{E_\mathbf{C}}\otimes\widetilde{E_\mathbf{C}}+\widetilde{E_\mathbf{C}}
+\wedge^2\widetilde{E_\mathbf{C}}),g,d)]\}^{(11)}\\
=&-16632\{{\widehat{A}(TM,\nabla^{TM})}{\rm exp}(\frac{c}{2}){\rm ch}(\triangle(E),g^{\triangle(E)},d)\}^{(11)}.
  \end{split}
\end{equation}
\end{thm}
\begin{cor}
Let $M$ be an $11$-dimensional $spin^c$ manifold without boundary. If $c_3(E,g,d)=0$, then
\begin{equation}
\begin{split}
{\rm Ind}(T^c &\otimes ((\widetilde{T_\mathbf{C}M}+2\wedge^2\widetilde{L_{\bf{R}}\otimes\mathbf{C}}
-(\widetilde{L_{\bf{R}}\otimes\mathbf{C}})
    \otimes(\widetilde{L_{\bf{R}}\otimes\mathbf{C}})\\
    &+\widetilde{L_{\bf{R}}\otimes\mathbf{C}})
\otimes(\triangle(E),g^{\triangle(E)},d)+(\triangle(E)\otimes(2\wedge^2\widetilde{E_\mathbf{C}}\\
&-\widetilde
{E_\mathbf{C}}\otimes\widetilde{E_\mathbf{C}}+\widetilde{E_\mathbf{C}}),g,d)))\equiv 0 ~~{\rm mod} ~~504Z,
\end{split}
\end{equation}
\begin{equation}
  \begin{split}
   {\rm Ind}(T &\otimes ((S^2\widetilde{T_\mathbf{C}M}+\widetilde{T_\mathbf{C}M}+\wedge^2\widetilde{L_{\bf{R}}
   \otimes\mathbf{C}}\otimes\wedge^2\widetilde{L_{\bf{R}}
\otimes\mathbf{C}}\\
&+2\wedge^4\widetilde{L_{\bf{R}}\otimes\mathbf{C}}
-2\widetilde{L_{\bf{R}}
\otimes\mathbf{C}}\otimes\wedge^3\widetilde{L_{\bf{R}}
\otimes\mathbf{C}}
+2\widetilde{L_{\bf{R}}
\otimes\mathbf{C}}\otimes\wedge^2\widetilde{L_{\bf{R}}
\otimes\mathbf{C}}\\
&-\widetilde{L_{\bf{R}}
\otimes\mathbf{C}}\otimes\widetilde{L_{\bf{R}}
\otimes\mathbf{C}}\otimes\widetilde{L_{\bf{R}}
\otimes\mathbf{C}}+\widetilde{L_{\bf{R}}
\otimes\mathbf{C}}+\wedge^2\widetilde{L_{\bf{R}}\otimes\mathbf{C}}\\
&+\widetilde{T_\mathbf{C}M}\otimes(2\wedge^2\widetilde{L_{\bf{R}}
\otimes\mathbf{C}}-(\widetilde{L_{\bf{R}}\otimes\mathbf{C}})
    \otimes(\widetilde{L_{\bf{R}}\otimes\mathbf{C}})
    +\widetilde{L_{\bf{R}}\otimes\mathbf{C}}))\\
    &\otimes(\triangle(E),g^{\triangle(E)},d)
    +(2\wedge^2\widetilde{L_{\bf{R}}\otimes\mathbf{C}}-(\widetilde{L_{\bf{R}}\otimes\mathbf{C}})
    \otimes(\widetilde{L_{\bf{R}}\otimes\mathbf{C}})\\
    &+\widetilde{L_{\bf{R}}\otimes\mathbf{C}})\otimes(\triangle(E)\otimes(2\wedge^2\widetilde
    {E_\mathbf{C}}-\widetilde
{E_\mathbf{C}}\otimes\widetilde{E_\mathbf{C}}+\widetilde{E_\mathbf{C}}),g,d)\\
&+(\triangle(E)\otimes(2\wedge^2\widetilde{E_\mathbf{C}}-\widetilde
{E_\mathbf{C}}\otimes\widetilde{E_\mathbf{C}}+\widetilde{E_\mathbf{C}}),g,d)\\
&+(\triangle(E)\otimes(\wedge^2\widetilde{E_\mathbf{C}}\otimes\wedge^2\widetilde{E_\mathbf{C}}
+2\wedge^4\widetilde{E_\mathbf{C}}-2\widetilde{E_\mathbf{C}}\otimes\wedge^3\widetilde{E_\mathbf{C}}\\
&+2\widetilde{E_\mathbf{C}}\otimes\wedge^2\widetilde{E_\mathbf{C}}-\widetilde{E_\mathbf{C}}\otimes
\widetilde{E_\mathbf{C}}\otimes\widetilde{E_\mathbf{C}}+\widetilde{E_\mathbf{C}}
+\wedge^2\widetilde{E_\mathbf{C}}),g,d)))\equiv 0 ~~{\rm mod} ~~16632Z.
  \end{split}
\end{equation}
\end{cor}
\begin{thm}
Let ${\rm dim}M=15$. If $3p_1(L)-p_1(M)=0$ and $c_3(E,g,d)=0$, we have
\begin{equation}
  \begin{split}
    \{\widehat{A}(&TM,\nabla^{TM}){\rm exp}(\frac{c}{2})[{\rm ch}(\widetilde{T_\mathbf{C}M}+2\wedge^2\widetilde{L_{\bf{R}}\otimes\mathbf{C}}\\
&-(\widetilde{L_{\bf{R}}\otimes\mathbf{C}})
    \otimes(\widetilde{L_{\bf{R}}\otimes\mathbf{C}})
    +\widetilde{L_{\bf{R}}\otimes\mathbf{C}})
{\rm ch}(\triangle(E),g^{\triangle(E)},d)\\
    &+{\rm ch}(\triangle(E)
    \otimes(2\wedge^2\widetilde{E_\mathbf{C}}-\widetilde
{E_\mathbf{C}}\otimes\widetilde{E_\mathbf{C}}+\widetilde{E_\mathbf{C}}),g,d)]\}^{(15)}\\
=&480\{{\widehat{A}(TM,\nabla^{TM})}{\rm exp}(\frac{c}{2}){\rm ch}(\triangle(E),g^{\triangle(E)},d)\}^{(15)},
  \end{split}
\end{equation}
\begin{equation}
  \begin{split}
   \{\widehat{A}(&TM,\nabla^{TM}){\rm exp}(\frac{c}{2})[{\rm ch}(S^2\widetilde{T_\mathbf{C}M}+\widetilde{T_\mathbf{C}M}+\wedge^2\widetilde{L_{\bf{R}}
   \otimes\mathbf{C}}\otimes\wedge^2\widetilde{L_{\bf{R}}
\otimes\mathbf{C}}\\
&+2\wedge^4\widetilde{L_{\bf{R}}\otimes\mathbf{C}}
-2\widetilde{L_{\bf{R}}
\otimes\mathbf{C}}\otimes\wedge^3\widetilde{L_{\bf{R}}
\otimes\mathbf{C}}
+2\widetilde{L_{\bf{R}}
\otimes\mathbf{C}}\otimes\wedge^2\widetilde{L_{\bf{R}}
\otimes\mathbf{C}}\\
&-\widetilde{L_{\bf{R}}
\otimes\mathbf{C}}\otimes\widetilde{L_{\bf{R}}
\otimes\mathbf{C}}\otimes\widetilde{L_{\bf{R}}
\otimes\mathbf{C}}+\widetilde{L_{\bf{R}}
\otimes\mathbf{C}}+\wedge^2\widetilde{L_{\bf{R}}\otimes\mathbf{C}}\\
&+\widetilde{T_\mathbf{C}M}\otimes(2\wedge^2\widetilde{L_{\bf{R}}
\otimes\mathbf{C}}-(\widetilde{L_{\bf{R}}\otimes\mathbf{C}})
    \otimes(\widetilde{L_{\bf{R}}\otimes\mathbf{C}})
    +\widetilde{L_{\bf{R}}\otimes\mathbf{C}}))\\
&\cdot{\rm ch}(\triangle(E),g^{\triangle(E)},d)
+{\rm ch}(2\wedge^2\widetilde{L_{\bf{R}}\otimes\mathbf{C}}-(\widetilde{L_{\bf{R}}\otimes\mathbf{C}})
    \otimes(\widetilde{L_{\bf{R}}\otimes\mathbf{C}})\\
    &+\widetilde{L_{\bf{R}}\otimes\mathbf{C}}){\rm ch}(\triangle(E)\otimes(2\wedge^2\widetilde{E_\mathbf{C}}-\widetilde
{E_\mathbf{C}}\otimes\widetilde{E_\mathbf{C}}+\widetilde{E_\mathbf{C}}),g,d)\\
&+{\rm ch}(\triangle(E)\otimes(2\wedge^2\widetilde{E_\mathbf{C}}-\widetilde
{E_\mathbf{C}}\otimes\widetilde{E_\mathbf{C}}+\widetilde{E_\mathbf{C}}),g,d)\\
&+{\rm ch}(\triangle(E)\otimes(\wedge^2\widetilde{E_\mathbf{C}}\otimes\wedge^2\widetilde{E_\mathbf{C}}
+2\wedge^4\widetilde{E_\mathbf{C}}-2\widetilde{E_\mathbf{C}}\otimes\wedge^3\widetilde{E_\mathbf{C}}\\
&+2\widetilde{E_\mathbf{C}}\otimes\wedge^2\widetilde{E_\mathbf{C}}-\widetilde{E_\mathbf{C}}\otimes
\widetilde{E_\mathbf{C}}\otimes\widetilde{E_\mathbf{C}}+\widetilde{E_\mathbf{C}}
+\wedge^2\widetilde{E_\mathbf{C}}),g,d)]\}^{(15)}\\
=&61920\{{\widehat{A}(TM,\nabla^{TM})}{\rm exp}(\frac{c}{2}){\rm ch}(\triangle(E),g^{\triangle(E)},d)\}^{(15)}.
  \end{split}
\end{equation}
\end{thm}
\begin{cor}
Let $M$ be an $15$-dimensional $spin^c$ manifold without boundary. If $c_3(E,g,d)=0$, then
\begin{equation}
\begin{split}
{\rm Ind}(T^c &\otimes ((\widetilde{T_\mathbf{C}M}+2\wedge^2\widetilde{L_{\bf{R}}\otimes\mathbf{C}}
-(\widetilde{L_{\bf{R}}\otimes\mathbf{C}})
    \otimes(\widetilde{L_{\bf{R}}\otimes\mathbf{C}})\\
    &+\widetilde{L_{\bf{R}}\otimes\mathbf{C}})
\otimes(\triangle(E),g^{\triangle(E)},d)+(\triangle(E)\otimes(2\wedge^2\widetilde{E_\mathbf{C}}\\
&-\widetilde
{E_\mathbf{C}}\otimes\widetilde{E_\mathbf{C}}+\widetilde{E_\mathbf{C}}),g,d)))\equiv 0 ~~{\rm mod} ~~480Z,
\end{split}
\end{equation}
\begin{equation}
  \begin{split}
   {\rm Ind}(T &\otimes ((S^2\widetilde{T_\mathbf{C}M}+\widetilde{T_\mathbf{C}M}+\wedge^2\widetilde{L_{\bf{R}}
   \otimes\mathbf{C}}\otimes\wedge^2\widetilde{L_{\bf{R}}
\otimes\mathbf{C}}\\
&+2\wedge^4\widetilde{L_{\bf{R}}\otimes\mathbf{C}}
-2\widetilde{L_{\bf{R}}
\otimes\mathbf{C}}\otimes\wedge^3\widetilde{L_{\bf{R}}
\otimes\mathbf{C}}
+2\widetilde{L_{\bf{R}}
\otimes\mathbf{C}}\otimes\wedge^2\widetilde{L_{\bf{R}}
\otimes\mathbf{C}}\\
&-\widetilde{L_{\bf{R}}
\otimes\mathbf{C}}\otimes\widetilde{L_{\bf{R}}
\otimes\mathbf{C}}\otimes\widetilde{L_{\bf{R}}
\otimes\mathbf{C}}+\widetilde{L_{\bf{R}}
\otimes\mathbf{C}}+\wedge^2\widetilde{L_{\bf{R}}\otimes\mathbf{C}}\\
&+\widetilde{T_\mathbf{C}M}\otimes(2\wedge^2\widetilde{L_{\bf{R}}
\otimes\mathbf{C}}-(\widetilde{L_{\bf{R}}\otimes\mathbf{C}})
    \otimes(\widetilde{L_{\bf{R}}\otimes\mathbf{C}})
    +\widetilde{L_{\bf{R}}\otimes\mathbf{C}}))\\
    &\otimes(\triangle(E),g^{\triangle(E)},d)
    +(2\wedge^2\widetilde{L_{\bf{R}}\otimes\mathbf{C}}-(\widetilde{L_{\bf{R}}\otimes\mathbf{C}})
    \otimes(\widetilde{L_{\bf{R}}\otimes\mathbf{C}})\\
    &+\widetilde{L_{\bf{R}}\otimes\mathbf{C}})\otimes(\triangle(E)\otimes(2\wedge^2\widetilde
    {E_\mathbf{C}}-\widetilde
{E_\mathbf{C}}\otimes\widetilde{E_\mathbf{C}}+\widetilde{E_\mathbf{C}}),g,d)\\
&+(\triangle(E)\otimes(2\wedge^2\widetilde{E_\mathbf{C}}-\widetilde
{E_\mathbf{C}}\otimes\widetilde{E_\mathbf{C}}+\widetilde{E_\mathbf{C}}),g,d)\\
&+(\triangle(E)\otimes(\wedge^2\widetilde{E_\mathbf{C}}\otimes\wedge^2\widetilde{E_\mathbf{C}}
+2\wedge^4\widetilde{E_\mathbf{C}}-2\widetilde{E_\mathbf{C}}\otimes\wedge^3\widetilde{E_\mathbf{C}}\\
&+2\widetilde{E_\mathbf{C}}\otimes\wedge^2\widetilde{E_\mathbf{C}}-\widetilde{E_\mathbf{C}}\otimes
\widetilde{E_\mathbf{C}}\otimes\widetilde{E_\mathbf{C}}+\widetilde{E_\mathbf{C}}
+\wedge^2\widetilde{E_\mathbf{C}}),g,d)))\equiv 0 ~~{\rm mod} ~~61920Z.
  \end{split}
\end{equation}
\end{cor}
\begin{thm}
Let ${\rm dim}M=19$. If $3p_1(L)-p_1(M)=0$ and $c_3(E,g,d)=0$, we have
\begin{equation}
  \begin{split}
    \{\widehat{A}(&TM,\nabla^{TM}){\rm exp}(\frac{c}{2})[{\rm ch}(\widetilde{T_\mathbf{C}M}+2\wedge^2\widetilde{L_{\bf{R}}\otimes\mathbf{C}}\\
&-(\widetilde{L_{\bf{R}}\otimes\mathbf{C}})
    \otimes(\widetilde{L_{\bf{R}}\otimes\mathbf{C}})
    +\widetilde{L_{\bf{R}}\otimes\mathbf{C}})
{\rm ch}(\triangle(E),g^{\triangle(E)},d)\\
    &+{\rm ch}(\triangle(E)
    \otimes(2\wedge^2\widetilde{E_\mathbf{C}}-\widetilde
{E_\mathbf{C}}\otimes\widetilde{E_\mathbf{C}}+\widetilde{E_\mathbf{C}}),g,d)]\}^{(19)}\\
=&-264\{{\widehat{A}(TM,\nabla^{TM})}{\rm exp}(\frac{c}{2}){\rm ch}(\triangle(E),g^{\triangle(E)},d)\}^{(19)},
  \end{split}
\end{equation}
\begin{equation}
  \begin{split}
   \{\widehat{A}(&TM,\nabla^{TM}){\rm exp}(\frac{c}{2})[{\rm ch}(S^2\widetilde{T_\mathbf{C}M}+\widetilde{T_\mathbf{C}M}+\wedge^2\widetilde{L_{\bf{R}}
   \otimes\mathbf{C}}\otimes\wedge^2\widetilde{L_{\bf{R}}
\otimes\mathbf{C}}\\
&+2\wedge^4\widetilde{L_{\bf{R}}\otimes\mathbf{C}}
-2\widetilde{L_{\bf{R}}
\otimes\mathbf{C}}\otimes\wedge^3\widetilde{L_{\bf{R}}
\otimes\mathbf{C}}
+2\widetilde{L_{\bf{R}}
\otimes\mathbf{C}}\otimes\wedge^2\widetilde{L_{\bf{R}}
\otimes\mathbf{C}}\\
&-\widetilde{L_{\bf{R}}
\otimes\mathbf{C}}\otimes\widetilde{L_{\bf{R}}
\otimes\mathbf{C}}\otimes\widetilde{L_{\bf{R}}
\otimes\mathbf{C}}+\widetilde{L_{\bf{R}}
\otimes\mathbf{C}}+\wedge^2\widetilde{L_{\bf{R}}\otimes\mathbf{C}}\\
&+\widetilde{T_\mathbf{C}M}\otimes(2\wedge^2\widetilde{L_{\bf{R}}
\otimes\mathbf{C}}-(\widetilde{L_{\bf{R}}\otimes\mathbf{C}})
    \otimes(\widetilde{L_{\bf{R}}\otimes\mathbf{C}})
    +\widetilde{L_{\bf{R}}\otimes\mathbf{C}}))\\
&\cdot{\rm ch}(\triangle(E),g^{\triangle(E)},d)
+{\rm ch}(2\wedge^2\widetilde{L_{\bf{R}}\otimes\mathbf{C}}-(\widetilde{L_{\bf{R}}\otimes\mathbf{C}})
    \otimes(\widetilde{L_{\bf{R}}\otimes\mathbf{C}})\\
    &+\widetilde{L_{\bf{R}}\otimes\mathbf{C}}){\rm ch}(\triangle(E)\otimes(2\wedge^2\widetilde{E_\mathbf{C}}-\widetilde
{E_\mathbf{C}}\otimes\widetilde{E_\mathbf{C}}+\widetilde{E_\mathbf{C}}),g,d)\\
&+{\rm ch}(\triangle(E)\otimes(2\wedge^2\widetilde{E_\mathbf{C}}-\widetilde
{E_\mathbf{C}}\otimes\widetilde{E_\mathbf{C}}+\widetilde{E_\mathbf{C}}),g,d)\\
&+{\rm ch}(\triangle(E)\otimes(\wedge^2\widetilde{E_\mathbf{C}}\otimes\wedge^2\widetilde{E_\mathbf{C}}
+2\wedge^4\widetilde{E_\mathbf{C}}-2\widetilde{E_\mathbf{C}}\otimes\wedge^3\widetilde{E_\mathbf{C}}\\
&+2\widetilde{E_\mathbf{C}}\otimes\wedge^2\widetilde{E_\mathbf{C}}-\widetilde{E_\mathbf{C}}\otimes
\widetilde{E_\mathbf{C}}\otimes\widetilde{E_\mathbf{C}}+\widetilde{E_\mathbf{C}}
+\wedge^2\widetilde{E_\mathbf{C}}),g,d)]\}^{(19)}\\
=&-117288\{{\widehat{A}(TM,\nabla^{TM})}{\rm exp}(\frac{c}{2}){\rm ch}(\triangle(E),g^{\triangle(E)},d)\}^{(19)}.
  \end{split}
\end{equation}
\end{thm}
\begin{cor}
Let $M$ be an $19$-dimensional $spin^c$ manifold without boundary. If $c_3(E,g,d)=0$, then
\begin{equation}
\begin{split}
{\rm Ind}(T^c &\otimes ((\widetilde{T_\mathbf{C}M}+2\wedge^2\widetilde{L_{\bf{R}}\otimes\mathbf{C}}
-(\widetilde{L_{\bf{R}}\otimes\mathbf{C}})
    \otimes(\widetilde{L_{\bf{R}}\otimes\mathbf{C}})\\
    &+\widetilde{L_{\bf{R}}\otimes\mathbf{C}})
\otimes(\triangle(E),g^{\triangle(E)},d)+(\triangle(E)\otimes(2\wedge^2\widetilde{E_\mathbf{C}}\\
&-\widetilde
{E_\mathbf{C}}\otimes\widetilde{E_\mathbf{C}}+\widetilde{E_\mathbf{C}}),g,d)))\equiv 0 ~~{\rm mod} ~~264Z,
\end{split}
\end{equation}
\begin{equation}
  \begin{split}
   {\rm Ind}(T &\otimes ((S^2\widetilde{T_\mathbf{C}M}+\widetilde{T_\mathbf{C}M}+\wedge^2\widetilde{L_{\bf{R}}
   \otimes\mathbf{C}}\otimes\wedge^2\widetilde{L_{\bf{R}}
\otimes\mathbf{C}}\\
&+2\wedge^4\widetilde{L_{\bf{R}}\otimes\mathbf{C}}
-2\widetilde{L_{\bf{R}}
\otimes\mathbf{C}}\otimes\wedge^3\widetilde{L_{\bf{R}}
\otimes\mathbf{C}}
+2\widetilde{L_{\bf{R}}
\otimes\mathbf{C}}\otimes\wedge^2\widetilde{L_{\bf{R}}
\otimes\mathbf{C}}\\
&-\widetilde{L_{\bf{R}}
\otimes\mathbf{C}}\otimes\widetilde{L_{\bf{R}}
\otimes\mathbf{C}}\otimes\widetilde{L_{\bf{R}}
\otimes\mathbf{C}}+\widetilde{L_{\bf{R}}
\otimes\mathbf{C}}+\wedge^2\widetilde{L_{\bf{R}}\otimes\mathbf{C}}\\
&+\widetilde{T_\mathbf{C}M}\otimes(2\wedge^2\widetilde{L_{\bf{R}}
\otimes\mathbf{C}}-(\widetilde{L_{\bf{R}}\otimes\mathbf{C}})
    \otimes(\widetilde{L_{\bf{R}}\otimes\mathbf{C}})
    +\widetilde{L_{\bf{R}}\otimes\mathbf{C}}))\\
    &\otimes(\triangle(E),g^{\triangle(E)},d)
    +(2\wedge^2\widetilde{L_{\bf{R}}\otimes\mathbf{C}}-(\widetilde{L_{\bf{R}}\otimes\mathbf{C}})
    \otimes(\widetilde{L_{\bf{R}}\otimes\mathbf{C}})\\
    &+\widetilde{L_{\bf{R}}\otimes\mathbf{C}})\otimes(\triangle(E)\otimes(2\wedge^2\widetilde
    {E_\mathbf{C}}-\widetilde
{E_\mathbf{C}}\otimes\widetilde{E_\mathbf{C}}+\widetilde{E_\mathbf{C}}),g,d)\\
&+(\triangle(E)\otimes(2\wedge^2\widetilde{E_\mathbf{C}}-\widetilde
{E_\mathbf{C}}\otimes\widetilde{E_\mathbf{C}}+\widetilde{E_\mathbf{C}}),g,d)\\
&+(\triangle(E)\otimes(\wedge^2\widetilde{E_\mathbf{C}}\otimes\wedge^2\widetilde{E_\mathbf{C}}
+2\wedge^4\widetilde{E_\mathbf{C}}-2\widetilde{E_\mathbf{C}}\otimes\wedge^3\widetilde{E_\mathbf{C}}\\
&+2\widetilde{E_\mathbf{C}}\otimes\wedge^2\widetilde{E_\mathbf{C}}-\widetilde{E_\mathbf{C}}\otimes
\widetilde{E_\mathbf{C}}\otimes\widetilde{E_\mathbf{C}}+\widetilde{E_\mathbf{C}}
+\wedge^2\widetilde{E_\mathbf{C}}),g,d)))\equiv 0 ~~{\rm mod} ~~117288Z.
  \end{split}
\end{equation}
\end{cor}
\begin{thm}
  Let ${\rm dim}M=23$. If $3p_1(L)-p_1(M)=0$ and $c_3(E,g,d)=0$, we have
    \begin{equation}
    \begin{split}
    \{\widehat{A}(&TM,\nabla^{TM}){\rm exp}(\frac{c}{2})[{\rm ch}(S^2\widetilde{T_\mathbf{C}M}+\widetilde{T_\mathbf{C}M}+\wedge^2\widetilde{L_{\bf{R}}
   \otimes\mathbf{C}}\otimes\wedge^2\widetilde{L_{\bf{R}}
\otimes\mathbf{C}}\\
&+2\wedge^4\widetilde{L_{\bf{R}}\otimes\mathbf{C}}
-2\widetilde{L_{\bf{R}}
\otimes\mathbf{C}}\otimes\wedge^3\widetilde{L_{\bf{R}}
\otimes\mathbf{C}}
+2\widetilde{L_{\bf{R}}
\otimes\mathbf{C}}\otimes\wedge^2\widetilde{L_{\bf{R}}
\otimes\mathbf{C}}\\
&-\widetilde{L_{\bf{R}}
\otimes\mathbf{C}}\otimes\widetilde{L_{\bf{R}}
\otimes\mathbf{C}}\otimes\widetilde{L_{\bf{R}}
\otimes\mathbf{C}}+\widetilde{L_{\bf{R}}
\otimes\mathbf{C}}+\wedge^2\widetilde{L_{\bf{R}}\otimes\mathbf{C}}\\
&+\widetilde{T_\mathbf{C}M}\otimes(2\wedge^2\widetilde{L_{\bf{R}}
\otimes\mathbf{C}}-(\widetilde{L_{\bf{R}}\otimes\mathbf{C}})
    \otimes(\widetilde{L_{\bf{R}}\otimes\mathbf{C}})
    +\widetilde{L_{\bf{R}}\otimes\mathbf{C}}))\\
&\cdot{\rm ch}(\triangle(E),g^{\triangle(E)},d)
+{\rm ch}(2\wedge^2\widetilde{L_{\bf{R}}\otimes\mathbf{C}}-(\widetilde{L_{\bf{R}}\otimes\mathbf{C}})
    \otimes(\widetilde{L_{\bf{R}}\otimes\mathbf{C}})\\
    &+\widetilde{L_{\bf{R}}\otimes\mathbf{C}}){\rm ch}(\triangle(E)\otimes(2\wedge^2\widetilde{E_\mathbf{C}}-\widetilde
{E_\mathbf{C}}\otimes\widetilde{E_\mathbf{C}}+\widetilde{E_\mathbf{C}}),g,d)\\
&+{\rm ch}(\triangle(E)\otimes(2\wedge^2\widetilde{E_\mathbf{C}}-\widetilde
{E_\mathbf{C}}\otimes\widetilde{E_\mathbf{C}}+\widetilde{E_\mathbf{C}}),g,d)\\
&+{\rm ch}(\triangle(E)\otimes(\wedge^2\widetilde{E_\mathbf{C}}\otimes\wedge^2\widetilde{E_\mathbf{C}}
+2\wedge^4\widetilde{E_\mathbf{C}}-2\widetilde{E_\mathbf{C}}\otimes\wedge^3\widetilde{E_\mathbf{C}}\\
&+2\widetilde{E_\mathbf{C}}\otimes\wedge^2\widetilde{E_\mathbf{C}}-\widetilde{E_\mathbf{C}}\otimes
\widetilde{E_\mathbf{C}}\otimes\widetilde{E_\mathbf{C}}+\widetilde{E_\mathbf{C}}
+\wedge^2\widetilde{E_\mathbf{C}}),g,d)]\}^{(23)}\\
=&\{196560[{\widehat{A}(TM,\nabla^{TM})}{\rm exp}(\frac{c}{2}){\rm ch}(\triangle(E),g^{\triangle(E)},d)]\\
&-24\{\widehat{A}(TM,\nabla^{TM}){\rm exp}(\frac{c}{2})[{\rm ch}(\widetilde{T_\mathbf{C}M}+2\wedge^2\widetilde{L_{\bf{R}}\otimes\mathbf{C}}\\
&-(\widetilde{L_{\bf{R}}\otimes\mathbf{C}})
    \otimes(\widetilde{L_{\bf{R}}\otimes\mathbf{C}})
    +\widetilde{L_{\bf{R}}\otimes\mathbf{C}})
{\rm ch}(\triangle(E),g^{\triangle(E)},d)\\
    &+{\rm ch}(\triangle(E)
    \otimes(2\wedge^2\widetilde{E_\mathbf{C}}-\widetilde
{E_\mathbf{C}}\otimes\widetilde{E_\mathbf{C}}+\widetilde{E_\mathbf{C}}),g,d)]\}\}^{(23)}
    \end{split}
  \end{equation}
\end{thm}
\begin{cor}
Let $M$ be an $23$-dimensional $spin^c$ manifold without boundary. If $c_3(E,g,d)=0$, then
\begin{equation}
  \begin{split}
   {\rm Ind}(T &\otimes ((S^2\widetilde{T_\mathbf{C}M}+\widetilde{T_\mathbf{C}M}+\wedge^2\widetilde{L_{\bf{R}}
   \otimes\mathbf{C}}\otimes\wedge^2\widetilde{L_{\bf{R}}
\otimes\mathbf{C}}\\
&+2\wedge^4\widetilde{L_{\bf{R}}\otimes\mathbf{C}}
-2\widetilde{L_{\bf{R}}
\otimes\mathbf{C}}\otimes\wedge^3\widetilde{L_{\bf{R}}
\otimes\mathbf{C}}
+2\widetilde{L_{\bf{R}}
\otimes\mathbf{C}}\otimes\wedge^2\widetilde{L_{\bf{R}}
\otimes\mathbf{C}}\\
&-\widetilde{L_{\bf{R}}
\otimes\mathbf{C}}\otimes\widetilde{L_{\bf{R}}
\otimes\mathbf{C}}\otimes\widetilde{L_{\bf{R}}
\otimes\mathbf{C}}+\widetilde{L_{\bf{R}}
\otimes\mathbf{C}}+\wedge^2\widetilde{L_{\bf{R}}\otimes\mathbf{C}}\\
&+\widetilde{T_\mathbf{C}M}\otimes(2\wedge^2\widetilde{L_{\bf{R}}
\otimes\mathbf{C}}-(\widetilde{L_{\bf{R}}\otimes\mathbf{C}})
    \otimes(\widetilde{L_{\bf{R}}\otimes\mathbf{C}})
    +\widetilde{L_{\bf{R}}\otimes\mathbf{C}}))\\
    &\otimes(\triangle(E),g^{\triangle(E)},d)
    +(2\wedge^2\widetilde{L_{\bf{R}}\otimes\mathbf{C}}-(\widetilde{L_{\bf{R}}\otimes\mathbf{C}})
    \otimes(\widetilde{L_{\bf{R}}\otimes\mathbf{C}})\\
    &+\widetilde{L_{\bf{R}}\otimes\mathbf{C}})\otimes(\triangle(E)\otimes(2\wedge^2\widetilde
    {E_\mathbf{C}}-\widetilde
{E_\mathbf{C}}\otimes\widetilde{E_\mathbf{C}}+\widetilde{E_\mathbf{C}}),g,d)\\
&+(\triangle(E)\otimes(2\wedge^2\widetilde{E_\mathbf{C}}-\widetilde
{E_\mathbf{C}}\otimes\widetilde{E_\mathbf{C}}+\widetilde{E_\mathbf{C}}),g,d)\\
&+(\triangle(E)\otimes(\wedge^2\widetilde{E_\mathbf{C}}\otimes\wedge^2\widetilde{E_\mathbf{C}}
+2\wedge^4\widetilde{E_\mathbf{C}}-2\widetilde{E_\mathbf{C}}\otimes\wedge^3\widetilde{E_\mathbf{C}}\\
&+2\widetilde{E_\mathbf{C}}\otimes\wedge^2\widetilde{E_\mathbf{C}}-\widetilde{E_\mathbf{C}}\otimes
\widetilde{E_\mathbf{C}}\otimes\widetilde{E_\mathbf{C}}+\widetilde{E_\mathbf{C}}
+\wedge^2\widetilde{E_\mathbf{C}}),g,d)))\equiv 0 ~~{\rm mod} ~~24Z.
  \end{split}
\end{equation}
\end{cor}

Let $\Theta^*(T_{\mathbf{C}}M,L_{\bf{R}}\otimes\mathbf{C})$ be the virtual complex vector bundle over $M$ defined by
$$\Theta^*(T_{\mathbf{C}}M,L_{\bf{R}}\otimes\mathbf{C})=\bigotimes _{n=1}^{\infty}S_{q^n}(\widetilde{T_{\mathbf{C}}M})\otimes
\bigotimes _{m=1}^{\infty}\wedge_{-q^m}(\widetilde{L_{\bf{R}}\otimes\mathbf{C}}),$$
Let ${\rm dim}M=4k+1$ and $y=-\frac{\sqrt{-1}}{2\pi}c.$ Set
\begin{equation}
\begin{split}
\bar{Q}(&\nabla^{TM},\nabla^{L},g,d,\tau)\\
=&\{\widehat{A}(TM,\nabla^{TM}){\rm exp}(\frac{c}{2}){\rm ch}(\Theta^*(T_{\mathbf{C}}M,L_{\bf{R}}\otimes\mathbf{C})){\rm ch}(Q(E),g^{Q(E)},d,\tau)\}^{(4k+1)}.
\end{split}
\end{equation}
Then
\begin{equation}
\bar{Q}(M,L,\tau)=\left\{\left(\prod_{j=1}^{2k+1}\frac{x_j\theta'(0,\tau)}{\theta(x_j,\tau)}\right)
\frac{\sqrt{-1}\theta(y,\tau)}{\theta_1(0,\tau)\theta_2(0,\tau)
\theta_3(0,\tau)}
\right\}^{(4p)},
\end{equation}
and
\begin{equation}
  \bar{Q}(\nabla^{TM},\nabla^{L},g,d,\tau)=(\bar{Q}(M,L,\tau)\cdot{\rm ch}(Q(E),g^{Q(E)},d,\tau))^{(4k+1)}.
\end{equation}
Let $p_1$ denote the first Pontryagin class. By (2.13)-(2.17), we have
$\bar{Q}(\nabla^{TM},\nabla^{L},g,d,\tau+1)=\bar{Q}(\nabla^{TM},\nabla^{L},g,d,\tau)$ and
$\bar{Q}(\nabla^{TM},\nabla^{L},g,d,-\frac{1}{\tau})=\tau^{2k}\bar{Q}(\nabla^{TM},\nabla^{L},g,d,\tau)$
if $p_1(L)-p_1(M)=0$ and $c_3(E_\mathbf{c},g,d)=0$. So
\begin{thm}
Let ${\rm dim}M=4k+1$. If $p_1(L)-p_1(M)=0$ and $c_3(E,g,d)=0$, then $\bar{Q}(\nabla^{TM},\nabla^{L},g,d,\tau)$ is a modular form over $SL_2({\bf Z})$ with the weight $2k$.
\end{thm}
By Theorem 3.23, similar to Theorems 3.13-3.22, we can get the following theorems.
\begin{thm}
Let ${\rm dim}M=7$. If $p_1(L)-p_1(M)=0$ and $c_3(E,g,d)=0$, we have
\begin{equation}
  \begin{split}
    \{\widehat{A}(&TM,\nabla^{TM}){\rm exp}(\frac{c}{2})[{\rm ch}(\widetilde{T_\mathbf{C}M}
-\widetilde{L_{\bf{R}}\otimes\mathbf{C}}){\rm ch}(\triangle(E),g^{\triangle(E)},d)\\
    &+{\rm ch}(\triangle(E)
    \otimes(2\wedge^2\widetilde{E_\mathbf{C}}-\widetilde
{E_\mathbf{C}}\otimes\widetilde{E_\mathbf{C}}+\widetilde{E_\mathbf{C}}),g,d)]\}^{(7)}\\
=&240\{{\widehat{A}(TM,\nabla^{TM})}{\rm exp}(\frac{c}{2}){\rm ch}(\triangle(E),g^{\triangle(E)},d)\}^{(7)},
  \end{split}
\end{equation}
\begin{equation}
  \begin{split}
   \{\widehat{A}(&TM,\nabla^{TM}){\rm exp}(\frac{c}{2})[{\rm ch}(S^2\widetilde{T_\mathbf{C}M}+\widetilde{T_\mathbf{C}M}+\wedge^2\widetilde{L_{\bf{R}}
   \otimes\mathbf{C}}-2\widetilde{L_{\bf{R}}
\otimes\mathbf{C}}\\
&+\widetilde{T_\mathbf{C}M}\otimes\widetilde{L_{\bf{R}}
\otimes\mathbf{C}})
{\rm ch}(\triangle(E),g^{\triangle(E)},d)
+{\rm ch}(\widetilde{T_\mathbf{C}M}\\
&-\widetilde{L_{\bf{R}}
\otimes\mathbf{C}}){\rm ch}(\triangle(E)\otimes(2\wedge^2\widetilde{E_\mathbf{C}}
-\widetilde
{E_\mathbf{C}}\otimes\widetilde{E_\mathbf{C}}+\widetilde{E_\mathbf{C}}),g,d)\\
&+{\rm ch}(\triangle(E)\otimes(\wedge^2\widetilde{E_\mathbf{C}}\otimes\wedge^2\widetilde{E_\mathbf{C}}
+2\wedge^4\widetilde{E_\mathbf{C}}-2\widetilde{E_\mathbf{C}}\otimes\wedge^3\widetilde{E_\mathbf{C}}\\
&+2\widetilde{E_\mathbf{C}}\otimes\wedge^2\widetilde{E_\mathbf{C}}-\widetilde{E_\mathbf{C}}\otimes
\widetilde{E_\mathbf{C}}\otimes\widetilde{E_\mathbf{C}}+\widetilde{E_\mathbf{C}}
+\wedge^2\widetilde{E_\mathbf{C}}),g,d)]\}^{(7)}\\
=&2160\{{\widehat{A}(TM,\nabla^{TM})}{\rm exp}(\frac{c}{2}){\rm ch}(\triangle(E),g^{\triangle(E)},d)\}^{(7)}.
  \end{split}
\end{equation}
\end{thm}
\begin{cor}
Let $M$ be an $7$-dimensional $spin^c$ manifold without boundary. If $c_3(E,g,d)=0$, then
\begin{equation}
\begin{split}
{\rm Ind}(T^c &\otimes ((\widetilde{T_\mathbf{C}M}
-\widetilde{L_{\bf{R}}\otimes\mathbf{C}})\otimes(\triangle(E),g^{\triangle(E)},d)\\
    &+(\triangle(E)
    \otimes(2\wedge^2\widetilde{E_\mathbf{C}}-\widetilde
{E_\mathbf{C}}\otimes\widetilde{E_\mathbf{C}}+\widetilde{E_\mathbf{C}}),g,d)))\equiv 0 ~~{\rm mod} ~~240Z,
\end{split}
\end{equation}
\begin{equation}
  \begin{split}
   {\rm Ind}(T &\otimes ((S^2\widetilde{T_\mathbf{C}M}+\widetilde{T_\mathbf{C}M}+\wedge^2\widetilde{L_{\bf{R}}
   \otimes\mathbf{C}}-2\widetilde{L_{\bf{R}}
\otimes\mathbf{C}}\\
&+\widetilde{T_\mathbf{C}M}\otimes\widetilde{L_{\bf{R}}
\otimes\mathbf{C}})
\otimes(\triangle(E),g^{\triangle(E)},d)
+(\widetilde{T_\mathbf{C}M}\\
&-\widetilde{L_{\bf{R}}
\otimes\mathbf{C}})\otimes(\triangle(E)\otimes(2\wedge^2\widetilde{E_\mathbf{C}}
-\widetilde
{E_\mathbf{C}}\otimes\widetilde{E_\mathbf{C}}+\widetilde{E_\mathbf{C}}),g,d)\\
&+(\triangle(E)\otimes(\wedge^2\widetilde{E_\mathbf{C}}\otimes\wedge^2\widetilde{E_\mathbf{C}}
+2\wedge^4\widetilde{E_\mathbf{C}}-2\widetilde{E_\mathbf{C}}\otimes\wedge^3\widetilde{E_\mathbf{C}}\\
&+2\widetilde{E_\mathbf{C}}\otimes\wedge^2\widetilde{E_\mathbf{C}}-\widetilde{E_\mathbf{C}}\otimes
\widetilde{E_\mathbf{C}}\otimes\widetilde{E_\mathbf{C}}+\widetilde{E_\mathbf{C}}
+\wedge^2\widetilde{E_\mathbf{C}}),g,d)))\equiv 0 ~~{\rm mod} ~~2160Z.
  \end{split}
\end{equation}
\end{cor}
\begin{thm}
Let ${\rm dim}M=11$. If $p_1(L)-p_1(M)=0$ and $c_3(E,g,d)=0$, we have
\begin{equation}
  \begin{split}
    \{\widehat{A}(&TM,\nabla^{TM}){\rm exp}(\frac{c}{2})[{\rm ch}(\widetilde{T_\mathbf{C}M}
-\widetilde{L_{\bf{R}}\otimes\mathbf{C}}){\rm ch}(\triangle(E),g^{\triangle(E)},d)\\
    &+{\rm ch}(\triangle(E)
    \otimes(2\wedge^2\widetilde{E_\mathbf{C}}-\widetilde
{E_\mathbf{C}}\otimes\widetilde{E_\mathbf{C}}+\widetilde{E_\mathbf{C}}),g,d)]\}^{(11)}\\
=&-504\{{\widehat{A}(TM,\nabla^{TM})}{\rm exp}(\frac{c}{2}){\rm ch}(\triangle(E),g^{\triangle(E)},d)\}^{(11)},
  \end{split}
\end{equation}
\begin{equation}
  \begin{split}
   \{\widehat{A}(&TM,\nabla^{TM}){\rm exp}(\frac{c}{2})[{\rm ch}(S^2\widetilde{T_\mathbf{C}M}+\widetilde{T_\mathbf{C}M}+\wedge^2\widetilde{L_{\bf{R}}
   \otimes\mathbf{C}}-2\widetilde{L_{\bf{R}}
\otimes\mathbf{C}}\\
&+\widetilde{T_\mathbf{C}M}\otimes\widetilde{L_{\bf{R}}
\otimes\mathbf{C}})
{\rm ch}(\triangle(E),g^{\triangle(E)},d)
+{\rm ch}(\widetilde{T_\mathbf{C}M}\\
&-\widetilde{L_{\bf{R}}
\otimes\mathbf{C}}){\rm ch}(\triangle(E)\otimes(2\wedge^2\widetilde{E_\mathbf{C}}
-\widetilde
{E_\mathbf{C}}\otimes\widetilde{E_\mathbf{C}}+\widetilde{E_\mathbf{C}}),g,d)\\
&+{\rm ch}(\triangle(E)\otimes(\wedge^2\widetilde{E_\mathbf{C}}\otimes\wedge^2\widetilde{E_\mathbf{C}}
+2\wedge^4\widetilde{E_\mathbf{C}}-2\widetilde{E_\mathbf{C}}\otimes\wedge^3\widetilde{E_\mathbf{C}}\\
&+2\widetilde{E_\mathbf{C}}\otimes\wedge^2\widetilde{E_\mathbf{C}}-\widetilde{E_\mathbf{C}}\otimes
\widetilde{E_\mathbf{C}}\otimes\widetilde{E_\mathbf{C}}+\widetilde{E_\mathbf{C}}
+\wedge^2\widetilde{E_\mathbf{C}}),g,d)]\}^{(11)}\\
=&-16632\{{\widehat{A}(TM,\nabla^{TM})}{\rm exp}(\frac{c}{2}){\rm ch}(\triangle(E),g^{\triangle(E)},d)\}^{(11)}.
  \end{split}
\end{equation}
\end{thm}
\begin{cor}
Let $M$ be an $11$-dimensional $spin^c$ manifold without boundary. If $c_3(E,g,d)=0$, then
\begin{equation}
\begin{split}
{\rm Ind}(T^c &\otimes ((\widetilde{T_\mathbf{C}M}
-\widetilde{L_{\bf{R}}\otimes\mathbf{C}})\otimes(\triangle(E),g^{\triangle(E)},d)\\
    &+(\triangle(E)
    \otimes(2\wedge^2\widetilde{E_\mathbf{C}}-\widetilde
{E_\mathbf{C}}\otimes\widetilde{E_\mathbf{C}}+\widetilde{E_\mathbf{C}}),g,d)))\equiv 0 ~~{\rm mod} ~~504Z,
\end{split}
\end{equation}
\begin{equation}
  \begin{split}
   {\rm Ind}(T &\otimes ((S^2\widetilde{T_\mathbf{C}M}+\widetilde{T_\mathbf{C}M}+\wedge^2\widetilde{L_{\bf{R}}
   \otimes\mathbf{C}}-2\widetilde{L_{\bf{R}}
\otimes\mathbf{C}}\\
&+\widetilde{T_\mathbf{C}M}\otimes\widetilde{L_{\bf{R}}
\otimes\mathbf{C}})
\otimes(\triangle(E),g^{\triangle(E)},d)
+(\widetilde{T_\mathbf{C}M}\\
&-\widetilde{L_{\bf{R}}
\otimes\mathbf{C}})\otimes(\triangle(E)\otimes(2\wedge^2\widetilde{E_\mathbf{C}}
-\widetilde
{E_\mathbf{C}}\otimes\widetilde{E_\mathbf{C}}+\widetilde{E_\mathbf{C}}),g,d)\\
&+(\triangle(E)\otimes(\wedge^2\widetilde{E_\mathbf{C}}\otimes\wedge^2\widetilde{E_\mathbf{C}}
+2\wedge^4\widetilde{E_\mathbf{C}}-2\widetilde{E_\mathbf{C}}\otimes\wedge^3\widetilde{E_\mathbf{C}}\\
&+2\widetilde{E_\mathbf{C}}\otimes\wedge^2\widetilde{E_\mathbf{C}}-\widetilde{E_\mathbf{C}}\otimes
\widetilde{E_\mathbf{C}}\otimes\widetilde{E_\mathbf{C}}+\widetilde{E_\mathbf{C}}
+\wedge^2\widetilde{E_\mathbf{C}}),g,d)))\equiv 0 ~~{\rm mod} ~~16632Z.
  \end{split}
\end{equation}
\end{cor}
\begin{thm}
Let ${\rm dim}M=15$. If $p_1(L)-p_1(M)=0$ and $c_3(E,g,d)=0$, we have
\begin{equation}
  \begin{split}
    \{\widehat{A}(&TM,\nabla^{TM}){\rm exp}(\frac{c}{2})[{\rm ch}(\widetilde{T_\mathbf{C}M}
-\widetilde{L_{\bf{R}}\otimes\mathbf{C}}){\rm ch}(\triangle(E),g^{\triangle(E)},d)\\
    &+{\rm ch}(\triangle(E)
    \otimes(2\wedge^2\widetilde{E_\mathbf{C}}-\widetilde
{E_\mathbf{C}}\otimes\widetilde{E_\mathbf{C}}+\widetilde{E_\mathbf{C}}),g,d)]\}^{(15)}\\
=&480\{{\widehat{A}(TM,\nabla^{TM})}{\rm exp}(\frac{c}{2}){\rm ch}(\triangle(E),g^{\triangle(E)},d)\}^{(15)},
  \end{split}
\end{equation}
\begin{equation}
  \begin{split}
   \{\widehat{A}(&TM,\nabla^{TM}){\rm exp}(\frac{c}{2})[{\rm ch}(S^2\widetilde{T_\mathbf{C}M}+\widetilde{T_\mathbf{C}M}+\wedge^2\widetilde{L_{\bf{R}}
   \otimes\mathbf{C}}-2\widetilde{L_{\bf{R}}
\otimes\mathbf{C}}\\
&+\widetilde{T_\mathbf{C}M}\otimes\widetilde{L_{\bf{R}}
\otimes\mathbf{C}})
{\rm ch}(\triangle(E),g^{\triangle(E)},d)
+{\rm ch}(\widetilde{T_\mathbf{C}M}\\
&-\widetilde{L_{\bf{R}}
\otimes\mathbf{C}}){\rm ch}(\triangle(E)\otimes(2\wedge^2\widetilde{E_\mathbf{C}}
-\widetilde
{E_\mathbf{C}}\otimes\widetilde{E_\mathbf{C}}+\widetilde{E_\mathbf{C}}),g,d)\\
&+{\rm ch}(\triangle(E)\otimes(\wedge^2\widetilde{E_\mathbf{C}}\otimes\wedge^2\widetilde{E_\mathbf{C}}
+2\wedge^4\widetilde{E_\mathbf{C}}-2\widetilde{E_\mathbf{C}}\otimes\wedge^3\widetilde{E_\mathbf{C}}\\
&+2\widetilde{E_\mathbf{C}}\otimes\wedge^2\widetilde{E_\mathbf{C}}-\widetilde{E_\mathbf{C}}\otimes
\widetilde{E_\mathbf{C}}\otimes\widetilde{E_\mathbf{C}}+\widetilde{E_\mathbf{C}}
+\wedge^2\widetilde{E_\mathbf{C}}),g,d)]\}^{(15)}\\
=&6192\{{\widehat{A}(TM,\nabla^{TM})}{\rm exp}(\frac{c}{2}){\rm ch}(\triangle(E),g^{\triangle(E)},d)\}^{(15)}.
  \end{split}
\end{equation}
\end{thm}
\begin{cor}
Let $M$ be an $15$-dimensional $spin^c$ manifold without boundary. If $c_3(E,g,d)=0$, then
\begin{equation}
\begin{split}
{\rm Ind}(T^c &\otimes ((\widetilde{T_\mathbf{C}M}
-\widetilde{L_{\bf{R}}\otimes\mathbf{C}})\otimes(\triangle(E),g^{\triangle(E)},d)\\
    &+(\triangle(E)
    \otimes(2\wedge^2\widetilde{E_\mathbf{C}}-\widetilde
{E_\mathbf{C}}\otimes\widetilde{E_\mathbf{C}}+\widetilde{E_\mathbf{C}}),g,d)))\equiv 0 ~~{\rm mod} ~~480Z,
\end{split}
\end{equation}
\begin{equation}
  \begin{split}
   {\rm Ind}(T &\otimes ((S^2\widetilde{T_\mathbf{C}M}+\widetilde{T_\mathbf{C}M}+\wedge^2\widetilde{L_{\bf{R}}
   \otimes\mathbf{C}}-2\widetilde{L_{\bf{R}}
\otimes\mathbf{C}}\\
&+\widetilde{T_\mathbf{C}M}\otimes\widetilde{L_{\bf{R}}
\otimes\mathbf{C}})
\otimes(\triangle(E),g^{\triangle(E)},d)
+(\widetilde{T_\mathbf{C}M}\\
&-\widetilde{L_{\bf{R}}
\otimes\mathbf{C}})\otimes(\triangle(E)\otimes(2\wedge^2\widetilde{E_\mathbf{C}}
-\widetilde
{E_\mathbf{C}}\otimes\widetilde{E_\mathbf{C}}+\widetilde{E_\mathbf{C}}),g,d)\\
&+(\triangle(E)\otimes(\wedge^2\widetilde{E_\mathbf{C}}\otimes\wedge^2\widetilde{E_\mathbf{C}}
+2\wedge^4\widetilde{E_\mathbf{C}}-2\widetilde{E_\mathbf{C}}\otimes\wedge^3\widetilde{E_\mathbf{C}}\\
&+2\widetilde{E_\mathbf{C}}\otimes\wedge^2\widetilde{E_\mathbf{C}}-\widetilde{E_\mathbf{C}}\otimes
\widetilde{E_\mathbf{C}}\otimes\widetilde{E_\mathbf{C}}+\widetilde{E_\mathbf{C}}
+\wedge^2\widetilde{E_\mathbf{C}}),g,d)))\equiv 0 ~~{\rm mod} ~~6192Z.
  \end{split}
\end{equation}
\end{cor}
\begin{thm}
Let ${\rm dim}M=19$. If $p_1(L)-p_1(M)=0$ and $c_3(E,g,d)=0$, we have
\begin{equation}
  \begin{split}
    \{\widehat{A}(&TM,\nabla^{TM}){\rm exp}(\frac{c}{2})[{\rm ch}(\widetilde{T_\mathbf{C}M}
-\widetilde{L_{\bf{R}}\otimes\mathbf{C}}){\rm ch}(\triangle(E),g^{\triangle(E)},d)\\
    &+{\rm ch}(\triangle(E)
    \otimes(2\wedge^2\widetilde{E_\mathbf{C}}-\widetilde
{E_\mathbf{C}}\otimes\widetilde{E_\mathbf{C}}+\widetilde{E_\mathbf{C}}),g,d)]\}^{(19)}\\
=&-264\{{\widehat{A}(TM,\nabla^{TM})}{\rm exp}(\frac{c}{2}){\rm ch}(\triangle(E),g^{\triangle(E)},d)\}^{(19)},
  \end{split}
\end{equation}
\begin{equation}
  \begin{split}
   \{\widehat{A}(&TM,\nabla^{TM}){\rm exp}(\frac{c}{2})[{\rm ch}(S^2\widetilde{T_\mathbf{C}M}+\widetilde{T_\mathbf{C}M}+\wedge^2\widetilde{L_{\bf{R}}
   \otimes\mathbf{C}}-2\widetilde{L_{\bf{R}}
\otimes\mathbf{C}}\\
&+\widetilde{T_\mathbf{C}M}\otimes\widetilde{L_{\bf{R}}
\otimes\mathbf{C}})
{\rm ch}(\triangle(E),g^{\triangle(E)},d)
+{\rm ch}(\widetilde{T_\mathbf{C}M}\\
&-\widetilde{L_{\bf{R}}
\otimes\mathbf{C}}){\rm ch}(\triangle(E)\otimes(2\wedge^2\widetilde{E_\mathbf{C}}
-\widetilde
{E_\mathbf{C}}\otimes\widetilde{E_\mathbf{C}}+\widetilde{E_\mathbf{C}}),g,d)\\
&+{\rm ch}(\triangle(E)\otimes(\wedge^2\widetilde{E_\mathbf{C}}\otimes\wedge^2\widetilde{E_\mathbf{C}}
+2\wedge^4\widetilde{E_\mathbf{C}}-2\widetilde{E_\mathbf{C}}\otimes\wedge^3\widetilde{E_\mathbf{C}}\\
&+2\widetilde{E_\mathbf{C}}\otimes\wedge^2\widetilde{E_\mathbf{C}}-\widetilde{E_\mathbf{C}}\otimes
\widetilde{E_\mathbf{C}}\otimes\widetilde{E_\mathbf{C}}+\widetilde{E_\mathbf{C}}
+\wedge^2\widetilde{E_\mathbf{C}}),g,d)]\}^{(19)}\\
=&-117288\{{\widehat{A}(TM,\nabla^{TM})}{\rm exp}(\frac{c}{2}){\rm ch}(\triangle(E),g^{\triangle(E)},d)\}^{(19)}.
  \end{split}
\end{equation}
\end{thm}
\begin{cor}
Let $M$ be an $19$-dimensional $spin^c$ manifold without boundary. If $c_3(E,g,d)=0$, then
\begin{equation}
\begin{split}
{\rm Ind}(T^c &\otimes ((\widetilde{T_\mathbf{C}M}
-\widetilde{L_{\bf{R}}\otimes\mathbf{C}})\otimes(\triangle(E),g^{\triangle(E)},d)\\
    &+(\triangle(E)
    \otimes(2\wedge^2\widetilde{E_\mathbf{C}}-\widetilde
{E_\mathbf{C}}\otimes\widetilde{E_\mathbf{C}}+\widetilde{E_\mathbf{C}}),g,d)))\equiv 0 ~~{\rm mod} ~~264Z,
\end{split}
\end{equation}
\begin{equation}
  \begin{split}
   {\rm Ind}(T &\otimes ((S^2\widetilde{T_\mathbf{C}M}+\widetilde{T_\mathbf{C}M}+\wedge^2\widetilde{L_{\bf{R}}
   \otimes\mathbf{C}}-2\widetilde{L_{\bf{R}}
\otimes\mathbf{C}}\\
&+\widetilde{T_\mathbf{C}M}\otimes\widetilde{L_{\bf{R}}
\otimes\mathbf{C}})
\otimes(\triangle(E),g^{\triangle(E)},d)
+(\widetilde{T_\mathbf{C}M}\\
&-\widetilde{L_{\bf{R}}
\otimes\mathbf{C}})\otimes(\triangle(E)\otimes(2\wedge^2\widetilde{E_\mathbf{C}}
-\widetilde
{E_\mathbf{C}}\otimes\widetilde{E_\mathbf{C}}+\widetilde{E_\mathbf{C}}),g,d)\\
&+(\triangle(E)\otimes(\wedge^2\widetilde{E_\mathbf{C}}\otimes\wedge^2\widetilde{E_\mathbf{C}}
+2\wedge^4\widetilde{E_\mathbf{C}}-2\widetilde{E_\mathbf{C}}\otimes\wedge^3\widetilde{E_\mathbf{C}}\\
&+2\widetilde{E_\mathbf{C}}\otimes\wedge^2\widetilde{E_\mathbf{C}}-\widetilde{E_\mathbf{C}}\otimes
\widetilde{E_\mathbf{C}}\otimes\widetilde{E_\mathbf{C}}+\widetilde{E_\mathbf{C}}
+\wedge^2\widetilde{E_\mathbf{C}}),g,d)))\equiv 0 ~~{\rm mod} ~~117288Z.
  \end{split}
\end{equation}
\end{cor}
\begin{thm}
  Let ${\rm dim}M=23$. If $p_1(L)-p_1(M)=0$ and $c_3(E,g,d)=0$, we have
    \begin{equation}
    \begin{split}
    \{\widehat{A}(&TM,\nabla^{TM}){\rm exp}(\frac{c}{2})[{\rm ch}(S^2\widetilde{T_\mathbf{C}M}+\widetilde{T_\mathbf{C}M}+\wedge^2\widetilde{L_{\bf{R}}
   \otimes\mathbf{C}}-2\widetilde{L_{\bf{R}}
\otimes\mathbf{C}}\\
&+\widetilde{T_\mathbf{C}M}\otimes\widetilde{L_{\bf{R}}
\otimes\mathbf{C}})
{\rm ch}(\triangle(E),g^{\triangle(E)},d)
+{\rm ch}(\widetilde{T_\mathbf{C}M}\\
&-\widetilde{L_{\bf{R}}
\otimes\mathbf{C}}){\rm ch}(\triangle(E)\otimes(2\wedge^2\widetilde{E_\mathbf{C}}
-\widetilde
{E_\mathbf{C}}\otimes\widetilde{E_\mathbf{C}}+\widetilde{E_\mathbf{C}}),g,d)\\
&+{\rm ch}(\triangle(E)\otimes(\wedge^2\widetilde{E_\mathbf{C}}\otimes\wedge^2\widetilde{E_\mathbf{C}}
+2\wedge^4\widetilde{E_\mathbf{C}}-2\widetilde{E_\mathbf{C}}\otimes\wedge^3\widetilde{E_\mathbf{C}}\\
&+2\widetilde{E_\mathbf{C}}\otimes\wedge^2\widetilde{E_\mathbf{C}}-\widetilde{E_\mathbf{C}}\otimes
\widetilde{E_\mathbf{C}}\otimes\widetilde{E_\mathbf{C}}+\widetilde{E_\mathbf{C}}
+\wedge^2\widetilde{E_\mathbf{C}}),g,d)]\}^{(23)}\\
=&\{196560[{\widehat{A}(TM,\nabla^{TM})}{\rm exp}(\frac{c}{2}){\rm ch}(\triangle(E),g^{\triangle(E)},d)]\\
&-24\{\widehat{A}(TM,\nabla^{TM}){\rm exp}(\frac{c}{2})[{\rm ch}(\widetilde{T_\mathbf{C}M}
-\widetilde{L_{\bf{R}}\otimes\mathbf{C}}){\rm ch}(\triangle(E),g^{\triangle(E)},d)\\
    &+{\rm ch}(\triangle(E)
    \otimes(2\wedge^2\widetilde{E_\mathbf{C}}-\widetilde
{E_\mathbf{C}}\otimes\widetilde{E_\mathbf{C}}+\widetilde{E_\mathbf{C}}),g,d)]\}\}^{(23)}
    \end{split}
  \end{equation}
\end{thm}
\begin{cor}
Let $M$ be an $23$-dimensional $spin^c$ manifold without boundary. If $c_3(E,g,d)=0$, then
\begin{equation}
  \begin{split}
   {\rm Ind}(T &\otimes ((S^2\widetilde{T_\mathbf{C}M}+\widetilde{T_\mathbf{C}M}+\wedge^2\widetilde{L_{\bf{R}}
   \otimes\mathbf{C}}-2\widetilde{L_{\bf{R}}
\otimes\mathbf{C}}\\
&+\widetilde{T_\mathbf{C}M}\otimes\widetilde{L_{\bf{R}}
\otimes\mathbf{C}})
\otimes(\triangle(E),g^{\triangle(E)},d)
+(\widetilde{T_\mathbf{C}M}\\
&-\widetilde{L_{\bf{R}}
\otimes\mathbf{C}})\otimes(\triangle(E)\otimes(2\wedge^2\widetilde{E_\mathbf{C}}
-\widetilde
{E_\mathbf{C}}\otimes\widetilde{E_\mathbf{C}}+\widetilde{E_\mathbf{C}}),g,d)\\
&+(\triangle(E)\otimes(\wedge^2\widetilde{E_\mathbf{C}}\otimes\wedge^2\widetilde{E_\mathbf{C}}
+2\wedge^4\widetilde{E_\mathbf{C}}-2\widetilde{E_\mathbf{C}}\otimes\wedge^3\widetilde{E_\mathbf{C}}\\
&+2\widetilde{E_\mathbf{C}}\otimes\wedge^2\widetilde{E_\mathbf{C}}-\widetilde{E_\mathbf{C}}\otimes
\widetilde{E_\mathbf{C}}\otimes\widetilde{E_\mathbf{C}}+\widetilde{E_\mathbf{C}}
+\wedge^2\widetilde{E_\mathbf{C}}),g,d)))\equiv 0 ~~{\rm mod} ~~24Z.
  \end{split}
\end{equation}
\end{cor}

\section{Some modular forms and Witten genus over $\Gamma^0(2),\Gamma_0(2),\Gamma_{\theta}$}

We will consider some modular forms over $\Gamma^0(2)$,~$\Gamma_0(2)$,~$\Gamma_{\theta}$. Let $M$ be a $(4k-1)$-dimensional spin manifold. By (3.1)-(3.3) and (3.10)-(3.12), we can construct the following forms:
\begin{equation}
\begin{split}
Q_1(\nabla^{TM},g,d,\tau)=&\{\widehat{A}(TM,\nabla^{TM}){\rm ch}([\triangle(M)\otimes \Theta_1(T_{C}M)+2^{2k}\Theta_2(T_{C}M)\\
&+2^{2k}\Theta_3(T_{C}M)]){\rm ch}(Q_1(E),g^{Q(E)},d,\tau)\}^{(4k-1)},
\end{split}
\end{equation}
\begin{equation}
\begin{split}
Q_2(\nabla^{TM},g,d,\tau)=&\{\widehat{A}(TM,\nabla^{TM}){\rm ch}([\triangle(M)\otimes \Theta_1(T_{C}M)+2^{2k}\Theta_2(T_{C}M)\\
&+2^{2k}\Theta_3(T_{C}M)]){\rm ch}(Q_2(E),g^{Q(E)},d,\tau)\}^{(4k-1)},
\end{split}
\end{equation}
\begin{equation}
\begin{split}
Q_3(\nabla^{TM},g,d,\tau)=&\{\widehat{A}(TM,\nabla^{TM}){\rm ch}([\triangle(M)\otimes \Theta_1(T_{C}M)+2^{2k}\Theta_2(T_{C}M)\\
&+2^{2k}\Theta_3(T_{C}M)]){\rm ch}(Q_3(E),g^{Q(E)},d,\tau)\}^{(4k-1)}.
\end{split}
\end{equation}
Following \cite{HY}, we defined ${\rm ch}(Q_j(E),g^{Q_j(E)},d,\tau)$ for $j=1,2,3$ as following
\begin{equation}
{\rm ch}(Q_j(E),\nabla^{Q_j(E)}_0,\tau)-{\rm ch}(Q_j(E),\nabla^{Q_j(E)}_1,\tau)=d{\rm ch}(Q_j(E),g^{Q_j(E)},d,\tau),
\end{equation}
where
\begin{equation}
{\rm ch}(Q_1(E),g^{Q_1(E)},d,\tau)=-\frac{2^{N/2}}{8\pi^2}\int^1_0{\rm Tr}\left[g^{-1}dg\frac{\theta'_1(R_u/(4\pi^2),\tau)}{\theta_1(R_u/(4\pi^2),\tau)}\right]du,
\end{equation}
and for $j=2,3$
\begin{equation}
{\rm ch}(Q_j(E),g^{Q_j(E)},d,\tau)=-\frac{1}{8\pi^2}\int^1_0{\rm Tr}\left[g^{-1}dg\frac{\theta'_j(R_u/(4\pi^2),\tau)}{\theta_j(R_u/(4\pi^2),\tau)}\right]du.
\end{equation}
By Proposition 2.2 in \cite{HY}, we have if $c_3(E,g,d)=0$, then for any integer $k\geq 1$ and $j=1,2,3$,~
${\rm ch}(Q_j(E),g^{Q_j(E)},d,\tau)^{(4k-1)}$ are modular forms of weight $2k$ over $\Gamma_0(2)$, $\Gamma^0(2)$ and $\Gamma_\theta$ respectively. By Proposition 2.4 and Theorem 2.6 in \cite{HY}, we have that if $c_3(E,g,d)=0$, then for any integer $k\geq 1$ and $j=1,2,3$,~
$Q_1(\nabla^{TM},g,d,\tau)^{(4k-1)},~Q_2(\nabla^{TM},g,d,\tau)^{(4k-1)}$ and $Q_3(\nabla^{TM},g,d,\tau)^{(4k-1)}$
are modular forms of weight $2k$ over $\Gamma_0(2)$, $\Gamma^0(2)$ and $\Gamma_\theta$ respectively.

Let $p_1$ denote the first Pontryagin class. If $\omega$ is a differential form over $M$, we donete $\omega^{4k-1}$ its top degree component. Our main results include the following theorem.

\begin{thm}
 If $c_3(E,g,d)=0$ , then
    \begin{equation}
        \{\widehat{A}(TM,\nabla^{TM}){\rm ch}(\triangle(M)+2^{k+1}){\rm ch}(\triangle(E),g^{\triangle(E)},d)\}^{(4k-1)}
        =2^{\frac{N}{2}+k}\sum_{s=1}^{[\frac{k}{2}]}2^{-6s}h_s,
    \end{equation}
where each $h_s, 1\leq s\leq [\frac{k}{2}],$ is a canonical integral linear combination of
\begin{equation}
\begin{split}
\{\widehat{A}&(TM){\rm ch}(\triangle(M))\widetilde{{\rm ch}}(B^1_{\alpha}(T_{\mathbf{C}}M,E_{\mathbf{C}}))\\
&+2^{2k}\widehat{A}(TM)\widetilde{{\rm ch}}(B^2_{\alpha}(T_{\mathbf{C}}M,E_{\mathbf{C}}))+2^{2k}\widehat{A}(TM)\widetilde{{\rm ch}}(B^3_{\alpha}(T_{\mathbf{C}}M,E_{\mathbf{C}}))\}^{(4k-1)},\ \ \ 0\leq \alpha \leq s,\nonumber
\end{split}
\end{equation}
and $h_1,$~$h_2$ are given by $(4.16)$ and $(4.17)$.
\end{thm}
\begin{proof}
Let $\{\pm 2\pi\sqrt{-1}x_j|1\leq j\leq k\}$ be the Chern roots of $T_{\mathbf{C}}M.$ We have
\begin{align}
Q_1(\nabla^{TM},g,d,\tau)=&\left\{\prod_{j=1}^{2k}\frac{2x_j\theta'(0,\tau)}{\theta(x_j,\tau)}
\left(\prod_{j=1}^{2k}\frac{\theta_1(x_j,\tau)}{\theta_1(0,\tau)}+\prod_{j=1}^{2k}\frac
{\theta_2(x_j,\tau)}{\theta_2(0,\tau)}+\prod_{j=1}^{2k}\frac{\theta_3(x_j,\tau)}
{\theta_3(0,\tau)}\right)\right.\\\notag
&\left.\cdot{\rm ch}(Q_1(E),g^{Q_1(E)},d,\tau)\right\}^{(4k-1)},
\end{align}
\begin{align}
Q_2(\nabla^{TM},g,d,\tau)=&\left\{\prod_{j=1}^{2k}\frac{2x_j\theta'(0,\tau)}{\theta(x_j,\tau)}
\left(\prod_{j=1}^{2k}\frac{\theta_1(x_j,\tau)}{\theta_1(0,\tau)}+\prod_{j=1}^{2k}\frac
{\theta_2(x_j,\tau)}{\theta_2(0,\tau)}+\prod_{j=1}^{2k}\frac{\theta_3(x_j,\tau)}
{\theta_3(0,\tau)}\right)\right.\\\notag
&\left.\cdot{\rm ch}(Q_2(E),g^{Q_1(E)},d,\tau)\right\}^{(4k-1)}.
\end{align}
Moreover, we can direct computations show that
  \begin{equation}
\begin{split}
Q_1(&\nabla^{TM},g,d,\tau)\\
=&[{\widehat{A}(TM,\nabla^{TM})}{\rm ch}(\triangle(M)){\rm ch}(\triangle(E),g^{\triangle(E)},d)\\
&+2^{2k+1}\widehat{A}(TM,\nabla^{TM}){\rm ch}(\triangle(E),g^{\triangle(E)},d)]^{(4k-1)}\\
&+q\{\widehat{A}(TM,\nabla^{TM})[{\rm ch}(\triangle(M)\otimes 2\widetilde{T_\mathbf{C}M}){\rm ch}(\triangle(E),g^{\triangle(E)},d)\\
&+{\rm ch}(\triangle(M)){\rm ch}(\triangle(E)\otimes\widetilde{E_\mathbf{C}},g,d)]
+2^{2k+1}\widehat{A}(TM,\nabla^{TM})[{\rm ch}(\widetilde{T_\mathbf{C}M}\\
&+\wedge^2\widetilde{T_\mathbf{C}M}){\rm ch}(\triangle(E),g^{\triangle(E)},d)
+{\rm ch}(\triangle(E)\otimes\widetilde{E_\mathbf{C}},g,d)]\}^{(4k-1)}\\
&+q^2\{{\widehat{A}(TM,\nabla^{TM})}[{\rm ch}(\triangle(M)\otimes(2\widetilde{T_\mathbf{C}M}+\wedge^2\widetilde{T_\mathbf{C}M}
+\widetilde{T_\mathbf{C}M}\otimes\widetilde{T_\mathbf{C}M}\\
&+S^2\widetilde{T_\mathbf{C}M})){\rm ch}(\triangle(E),g^{\triangle(E)},d)
+{\rm ch}(\triangle(M)\otimes2\widetilde
{T_{\mathbf{C}}M}){\rm ch}(\triangle(E)\otimes\widetilde{E_\mathbf{C}},g,d)\\
&+{\rm ch}(\triangle(M)){\rm ch}(\triangle(E)\otimes(\wedge^2\widetilde{E_\mathbf{C}}+\widetilde{E_\mathbf{C}}),g,d)]
+2^{2k+1}\widehat{A}(TM,\nabla^{TM})[{\rm ch}(\wedge^4\widetilde{T_\mathbf{C}M}\\
&+\wedge^2\widetilde{T_\mathbf{C}M}\otimes
\widetilde{T_\mathbf{C}M}
+\widetilde{T_\mathbf{C}M}\otimes\widetilde{T_\mathbf{C}M}+S^2
\widetilde{T_\mathbf{C}M}+\widetilde{T_\mathbf{C}M}){\rm ch}(\triangle(E),g^{\triangle(E)},d)\\
&+{\rm ch}(\widetilde{T_{
\mathbf{C}}M}+\wedge^2\widetilde{T_{\mathbf{C}}M}){\rm ch}(\triangle(E)\otimes\widetilde{E_\mathbf{C}},g,d)
+{\rm ch}(\triangle(E)\otimes(\wedge^2\widetilde{E_\mathbf{C}}+\widetilde{E_\mathbf{C}})
,g,d)]\}^{(4k-1)}+\cdots,
\end{split}
\end{equation}
  \begin{equation}
\begin{split}
Q_2(&\nabla^{TM},g,d,\tau)\\
=&-q^{\frac{1}{2}}[{\widehat{A}(TM,\nabla^{TM})}{\rm ch}(\triangle(M)){\rm ch}(\widetilde{E_\mathbf{C}},g,d)\\
&+2^{2k+1}{\widehat{A}(TM,\nabla^{TM})}{\rm ch}(\widetilde{E_\mathbf{C}},g,d)]^{(4k-1)}\\
&+q[{\widehat{A}(TM,\nabla^{TM})}{\rm ch}(\triangle(M)){\rm ch}(\wedge^2\widetilde{E_\mathbf{C}},g,d)\\
&+2^{2k+1}{\widehat{A}(TM,\nabla^{TM})}{\rm ch}(\wedge^2\widetilde{E_\mathbf{C}},g,d)]^{(4k-1)}\\
&-q^{\frac{3}{2}}\{{\widehat{A}(TM,\nabla^{TM})}[{\rm ch}(\triangle(M)\otimes2\widetilde{T_\mathbf{C}M}){\rm ch}(\widetilde{E_\mathbf{C}},g,d)+{\rm ch}\triangle(M){\rm ch}(\widetilde{E_\mathbf{C}}\\
&+\wedge^3\widetilde{E_\mathbf{C}},g,d)]
+2^{2k+1}{\widehat{A}(TM,
\nabla^{TM})}[{\rm ch}(\widetilde{T_\mathbf{C}M}+\wedge^2\widetilde{T_\mathbf{C}M}){\rm ch}(\widetilde{E_\mathbf{C}},g,d)\\
&+{\rm ch}(\widetilde{E_\mathbf{C}}+\wedge^3\widetilde{E_\mathbf{C}},g,d)]\}^{(4k-1)}\\
&+q^2\{{\widehat{A}(TM,\nabla^{TM})}[{\rm ch}(\triangle(M)\otimes2\widetilde{T_\mathbf{C}M}){\rm ch}(\wedge^2\widetilde{E_\mathbf{C}},g,d)\\
&+{\rm ch}\triangle(M){\rm ch}(\wedge^4\widetilde{E_\mathbf{C}}+\widetilde{E_\mathbf{C}}\otimes\widetilde{E_\mathbf{C}},g,d)]\\
&+2^{2k+1}{\widehat{A}(TM,\nabla^{TM})}[{\rm ch}(\widetilde{T_\mathbf{C}M}+\wedge^2\widetilde{T_\mathbf{C}M}){\rm ch}(\wedge^2\widetilde{E_\mathbf{C}},g,d)\\
&+{\rm ch}(\wedge^4\widetilde{E_\mathbf{C}}+\widetilde{E_\mathbf{C}}\otimes\widetilde{E_\mathbf{C}},g,d)]
\}^{(4k-1)}+\cdots,
\end{split}
\end{equation}
and we can represent $Q_2(\nabla^{TM},g,d,\tau)$ as
\begin{equation}
\begin{split}
Q_2(&\nabla^{TM},g,d,\tau)\\=&\widehat{A}(TM,\nabla^{TM}){\rm ch}(\triangle(M))\widetilde{{\rm ch}}(B^1_0(T_{\mathbf{C}}M,E_{\mathbf{C}}))\\
&+2^{2k}\widehat{A}(TM,\nabla^{TM})\widetilde{{\rm ch}}(B^2_0(T_{\mathbf{C}}M,E_{\mathbf{C}}))\\
&+2^{2k}\widehat{A}(TM,\nabla^{TM})\widetilde{{\rm ch}}(B^3_0(T_{\mathbf{C}}M,E_{\mathbf{C}}))\\
&+[\widehat{A}(TM,\nabla^{TM}){\rm ch}(\triangle(M))\widetilde{{\rm ch}}(B^1_1(T_{\mathbf{C}}M,E_{\mathbf{C}}))\\
&+2^{2k}\widehat{A}(TM,\nabla^{TM})\widetilde{{\rm ch}}(B^2_1(T_{\mathbf{C}}M,E_{\mathbf{C}}))\\
&+2^{2k}\widehat{A}(TM,\nabla^{TM})\widetilde{{\rm ch}}(B^3_1(T_{\mathbf{C}}M,E_{\mathbf{C}}))]q^{\frac{1}{2}}\\
&+[\widehat{A}(TM,\nabla^{TM}){\rm ch}(\triangle(M))\widetilde{{\rm ch}}(B^1_2(T_{\mathbf{C}}M,E_{\mathbf{C}}))\\
&+2^{2k}\widehat{A}(TM,\nabla^{TM})\widetilde{{\rm ch}}(B^2_2(T_{\mathbf{C}}M,E_{\mathbf{C}}))\\
&+2^{2k}\widehat{A}(TM,\nabla^{TM})\widetilde{{\rm ch}}(B^3_2(T_{\mathbf{C}}M,E_{\mathbf{C}}))]q\\
&+[\widehat{A}(TM,\nabla^{TM}){\rm ch}(\triangle(M))\widetilde{{\rm ch}}(B^1_3(T_{\mathbf{C}}M,E_{\mathbf{C}}))\\
&+2^{2k}\widehat{A}(TM,\nabla^{TM})\widetilde{{\rm ch}}(B^2_3(T_{\mathbf{C}}M,E_{\mathbf{C}}))\\
&+2^{2k}\widehat{A}(TM,\nabla^{TM})\widetilde{{\rm ch}}(B^3_3(T_{\mathbf{C}}M,E_{\mathbf{C}}))]q^{\frac{3}{2}}+\cdots.
\end{split}
\end{equation}
Let $P_1(\tau)=Q_1(\nabla^{TM},g,d,\tau)^{4k-1},$ $P_2(\tau)=Q_2(\nabla^{TM},g,d,\tau)^{4k-1},$ $P_3(\tau)=Q_2(\nabla^{TM},g,d,\tau)^{4k-1}.$ Similarly to the computations in \cite{Li1} and by (2.26) in \cite{HY} and the condition $c_3(E,g,d)=0$, we have
\begin{equation}
    P_1(-\frac{1}{\tau})=2^{\frac{N}{2}}\tau^{2k}P_2(\tau).
\end{equation}
Observe that at any point $x\in M,$ up to the volume form determined by the metric on $T_xM,$ both $P_i(\tau),~i=1,2,$ can be viewed as a power series of $q^{\frac{1}{2}}$ with real Fourier coefficients. By Lemma 2.2, we have
\begin{equation}
P_2(\tau)=h_0(8\delta_2)^{k}+h_1(8\delta_2)^{k-2}\varepsilon_2+\cdots+h_{[\frac{k}{2}]}(8\delta_2)^{k-2}
\varepsilon^{[\frac{k}{2}]}_2,
\end{equation}
where each $h_s,~0\leq s\leq [\frac{k}{2}],$ is a real multiple of the volume form at $x.$ By (2.19), (4.13) and (4.14), we get
\begin{equation}
P_1(\tau)=2^{\frac{N}{2}}[h_0(8\delta_1)^{k}+h_1(8\delta_1)^{k-2}\varepsilon_1+\cdots+h_{[\frac{k}{2}]}
(8\delta_1)^{k-2}\varepsilon^{[\frac{k}{2}]}_1].
\end{equation}
By comparing the constant term in (4.15), we get (4.7). By comparing
the coefficients of $q^{\frac{j}{2}}$,~$j\geq 0$ between the two
sides of (4.14), we can use the induction method to prove that $h_0=0$ and each
$h_l~1\leq s\leq [\frac{k}{2}]$, can be expressed through a
canonical integral linear combination of
\begin{equation}
\begin{split}
\{\widehat{A}&(TM){\rm ch}(\triangle(M))\widetilde{{\rm ch}}(B^1_{\alpha}(T_{\mathbf{C}}M,E_{\mathbf{C}}))\\
&+2^{2k}\widehat{A}(TM)\widetilde{{\rm ch}}(B^2_{\alpha}(T_{\mathbf{C}}M,E_{\mathbf{C}}))+2^{2k}\widehat{A}(TM)\widetilde{{\rm ch}}(B^3_{\alpha}(T_{\mathbf{C}}M,E_{\mathbf{C}}))\}^{(4k-1)},\ \ \ 0\leq \alpha \leq s.\nonumber
\end{split}
\end{equation}
By (4.11), (4.12) and comparing the coefficient of $q^{\frac{1}{2}}$ of (4.14), we get
\begin{equation}
  \begin{split}
    h_1=&(-1)^{k-2}[\widehat{A}(TM,\nabla^{TM}){\rm ch}(\triangle(M))\widetilde{{\rm ch}}(B^1_1(T_{\mathbf{C}}M,E_{\mathbf{C}}))\\
&+2^{2k}\widehat{A}(TM,\nabla^{TM})\widetilde{{\rm ch}}(B^2_1(T_{\mathbf{C}}M,E_{\mathbf{C}}))\\
&+2^{2k}\widehat{A}(TM,\nabla^{TM})\widetilde{{\rm ch}}(B^3_1(T_{\mathbf{C}}M,E_{\mathbf{C}}))]^{(4k-1)}\\
=&(-1)^{k-1}[{\widehat{A}(TM,\nabla^{TM})}{\rm ch}(\triangle(M)){\rm ch}(\widetilde{E_\mathbf{C}},g,d)\\
&+2^{2k+1}{\widehat{A}(TM,\nabla^{TM})}{\rm ch}(\widetilde{E_\mathbf{C}},g,d)]^{(4k-1)}.
  \end{split}
\end{equation}
By (4.11), (4.12) and comparing the coefficient of $q$ of (4.14), we get
\begin{equation}
  \begin{split}
    h_2=&(-1)^{k-4}[\widehat{A}(TM,\nabla^{TM}){\rm ch}(\triangle(M))\widetilde{{\rm ch}}(B^1_2(T_{\mathbf{C}}M,E_{\mathbf{C}}))\\
&+2^{2k}\widehat{A}(TM,\nabla^{TM})\widetilde{{\rm ch}}(B^2_2(T_{\mathbf{C}}M,E_{\mathbf{C}}))\\
&+2^{2k}\widehat{A}(TM,\nabla^{TM})\widetilde{{\rm ch}}(B^3_2(T_{\mathbf{C}}M,E_{\mathbf{C}}))]^{(4k-1)}\\
&-[8-24(k-2)(-1)^d]h_1\\
=&(-1)^{k-4}[{\widehat{A}(TM,\nabla^{TM})}{\rm ch}(\triangle(M)){\rm ch}(\wedge^2\widetilde{E_\mathbf{C}},g,d)\\
&+2^{2k+1}{\widehat{A}(TM,\nabla^{TM})}{\rm ch}(\wedge^2\widetilde{E_\mathbf{C}},g,d)]^{(4k-1)}\\
&-[24(k-2)+8(-1)^{k-1}][{\widehat{A}(TM,\nabla^{TM})}{\rm ch}(\triangle(M)){\rm ch}(\widetilde{E_\mathbf{C}},g,d)\\
&+2^{2k+1}{\widehat{A}(TM,\nabla^{TM})}{\rm ch}(\widetilde{E_\mathbf{C}},g,d)]^{(4k-1)}.
  \end{split}
\end{equation}
The proof is completed.
\end{proof}
\begin{cor}
  Let $M$ be a $(4k-1)$-dimensional spin manifold. If $c_3(E,g,d)=0$, then
  \begin{equation}
    \begin{split}
      {\rm Ind}(T\otimes(\triangle(M)+2^{k+1})\otimes(\triangle(E),g^{\triangle(E)},d))
=-2^{\frac{N}{2}+k}\sum_{s=0}^{[\frac{k}{2}]}2^{-6s}h_s,
    \end{split}
  \end{equation}
where each $h_l$,~$1\leq s\leq [\frac{k}{2}]$, is a canonical
integral linear combination of
\begin{equation}
\begin{split}
{\rm Ind}(T\otimes(\triangle(M)\otimes(B^1_{\alpha}(T_{\mathbf{C}}M,E_{\mathbf{C}}))
+2^{2k}\otimes(B^2_{\alpha}(T_{\mathbf{C}}M,E_{\mathbf{C}}))
+2^{2k}\otimes(B^3_{\alpha}(T_{\mathbf{C}}M,E_{\mathbf{C}})))).
\end{split}
\end{equation}
\end{cor}
\begin{cor}
  Let $M$ be a $(4k-1)$-dimensional spin manifold and $c_3(E,g,d)=0$. If $k$ is even, then
\begin{equation}
  {\rm Ind}(T\otimes(\triangle(M)+2^{k+1})\otimes(\triangle(E),g^{\triangle(E)},d))\equiv 0~~({\rm mod}~ 2^{\frac{N}{2}-2k}),
\end{equation}
If $k$ is odd, then
\begin{equation}
  {\rm Ind}(T\otimes(\triangle(M)+2^{k+1})\otimes(\triangle(E),g^{\triangle(E)},d))\equiv 0~~({\rm mod}~ 2^{\frac{N}{2}+3-2k}).
\end{equation}
\end{cor}
\begin{cor}
  Let $M$ be a $11$-dimensional spin manifold. If $c_3(E,g,d)=0$, then
  \begin{equation}
        \begin{split}
        \{\widehat{A}(&TM,\nabla^{TM}){\rm ch}(\triangle(M)+2^7){\rm ch}(\triangle(E),g^{\triangle(E)},d)\}^{(11)}\\
        &=2^{\frac{N}{2}-3}\{\widehat{A}(TM,\nabla^{TM}){\rm ch}(\triangle(M)+2^7){\rm ch}(\widetilde{E_\mathbf{C}},g,d)\}^{(11)}.
        \end{split}
    \end{equation}
\end{cor}
\begin{cor}
  Let $M$ be a $15$-dimensional spin manifold. If $c_3(E,g,d)=0$, then
  \begin{equation}
        \begin{split}
        \{\widehat{A}(&TM,\nabla^{TM}){\rm ch}(\triangle(M)+2^9){\rm ch}(\triangle(E),g^{\triangle(E)},d)\}^{(15)}\\
        =&-13\times2^{\frac{N}{2}-5}\{\widehat{A}(TM,\nabla^{TM}){\rm ch}(\triangle(M)+2^9){\rm ch}(\widetilde{E_\mathbf{C}},g,d)\}^{(15)}\\
        &+2^{\frac{N}{2}-8}\{\widehat{A}(TM,\nabla^{TM}){\rm ch}(\triangle(M)+2^9){\rm ch}(\wedge^2\widetilde{E_\mathbf{C}},g,d)\}^{(15)}.
        \end{split}
    \end{equation}
\end{cor}

By (5.22) in \cite{CH}, we have for $1\leq s \leq [\frac{k}{2}]$
\begin{equation}
  \begin{split}
  (8\delta_1)^{k-2s}\varepsilon^s_1=&2^{k-6s}[1+(24k-64s)q+(288k^2-1536ks\\
  &+2048s^2+512s-264k)q^2+O(q^3).
  \end{split}
\end{equation}
By (4.10), we are comparing the coefficients of $q$ in (4.15). Then
\begin{thm}
 If $c_3(E,g,d)=0$, then
    \begin{equation}
        \begin{split}
        \{\widehat{A}&(TM,\nabla^{TM})[{\rm ch}(\triangle(M)\otimes 2\widetilde{T_\mathbf{C}M}){\rm ch}(\triangle(E),g^{\triangle(E)},d)\\
&+{\rm ch}(\triangle(M)){\rm ch}(\triangle(E)\otimes\widetilde{E_\mathbf{C}},g,d)]+2^{2k+1}\widehat{A}(TM,\nabla^{TM})[{\rm ch}(\widetilde{T_\mathbf{C}M}\\
&+\wedge^2\widetilde{T_\mathbf{C}M}){\rm ch}(\triangle(E),g^{\triangle(E)},d)+{\rm ch}(\triangle(E)\otimes\widetilde{E_\mathbf{C}},g,d)]\\
&-24k[\widehat{A}(TM,\nabla^{TM}){\rm ch}(\triangle(M)+2^{2k+1}){\rm ch}(\triangle(E),g^{\triangle(E)},d)]\}^{(4k-1)}\\
        =&-2^{\frac{N}{2}+k+6}\sum_{s=1}^{[\frac{k}{2}]}s2^{-6s}h_s.
        \end{split}
    \end{equation}
\end{thm}
\begin{cor}
  Let $M$ be a $(4k-1)$-dimensional spin manifold and $c_3(E,g,d)=0$. If $k$ is even, then
\begin{equation}
  \begin{split}
  {\rm Ind}(&T\otimes(\triangle(M)\otimes(2\widetilde{T_\mathbf{C}M}\otimes(\triangle(E),g^{\triangle(E)}
  ,d)\\
  &+(\triangle(E)\otimes\widetilde{E_\mathbf{C}},g,d))+2^{2k+1}\otimes((\widetilde{T_\mathbf{C}M}\\
&+\wedge^2\widetilde{T_\mathbf{C}M})\otimes(\triangle(E),g^{\triangle(E)},d)
+(\triangle(E)\otimes\widetilde{E_\mathbf{C}},g,d))))\\
&-24k{\rm Ind}(T\otimes(\triangle(M)+2^{2k+1})\otimes(\triangle(E),g^{\triangle(E)},d))\\
&\equiv 0~~({\rm mod}~ 2^{\frac{N}{2}-2k+6}),
  \end{split}
\end{equation}
If $k$ is odd, then
\begin{equation}
  \begin{split}
  {\rm Ind}(&T\otimes(\triangle(M)\otimes(2\widetilde{T_\mathbf{C}M}\otimes(\triangle(E),g^{\triangle(E)}
  ,d)\\
  &+(\triangle(E)\otimes\widetilde{E_\mathbf{C}},g,d))+2^{2k+1}\otimes((\widetilde{T_\mathbf{C}M}\\
&+\wedge^2\widetilde{T_\mathbf{C}M})\otimes(\triangle(E),g^{\triangle(E)},d)
+(\triangle(E)\otimes\widetilde{E_\mathbf{C}},g,d))))\\
&-24k{\rm Ind}(T\otimes(\triangle(M)+2^{2k+1})\otimes(\triangle(E),g^{\triangle(E)},d))\\
&\equiv 0~~({\rm mod}~ 2^{\frac{N}{2}-2k+9}).
  \end{split}
\end{equation}
\end{cor}
\begin{cor}
  Let $M$ be a $11$-dimensional spin manifold. If $c_3(E,g,d)=0$, then
  \begin{equation}
        \begin{split}
        \{\widehat{A}&(TM,\nabla^{TM})[{\rm ch}(\triangle(M)\otimes 2\widetilde{T_\mathbf{C}M}){\rm ch}(\triangle(E),g^{\triangle(E)},d)\\
&+{\rm ch}(\triangle(M)){\rm ch}(\triangle(E)\otimes\widetilde{E_\mathbf{C}},g,d)]+2^{7}\widehat{A}(TM,\nabla^{TM})[{\rm ch}(\widetilde{T_\mathbf{C}M}\\
&+\wedge^2\widetilde{T_\mathbf{C}M}){\rm ch}(\triangle(E),g^{\triangle(E)},d)+{\rm ch}(\triangle(E)\otimes\widetilde{E_\mathbf{C}},g,d)]\\
&-72[\widehat{A}(TM,\nabla^{TM}){\rm ch}(\triangle(M)+2^{7}){\rm ch}(\triangle(E),g^{\triangle(E)},d)]\}^{(11)}\\
        =&-2^{\frac{N}{2}+3}\{\widehat{A}(TM,\nabla^{TM}){\rm ch}(\triangle(M)+2^7){\rm ch}(\widetilde{E_\mathbf{C}},g,d)\}^{(11)}.
        \end{split}
    \end{equation}
\end{cor}
\begin{cor}
  Let $M$ be a $15$-dimensional spin manifold. If $c_3(E,g,d)=0$, then
  \begin{equation}
        \begin{split}
        \{\widehat{A}&(TM,\nabla^{TM})[{\rm ch}(\triangle(M)\otimes 2\widetilde{T_\mathbf{C}M}){\rm ch}(\triangle(E),g^{\triangle(E)},d)\\
&+{\rm ch}(\triangle(M)){\rm ch}(\triangle(E)\otimes\widetilde{E_\mathbf{C}},g,d)]+2^{9}\widehat{A}(TM,\nabla^{TM})[{\rm ch}(\widetilde{T_\mathbf{C}M}\\
&+\wedge^2\widetilde{T_\mathbf{C}M}){\rm ch}(\triangle(E),g^{\triangle(E)},d)+{\rm ch}(\triangle(E)\otimes\widetilde{E_\mathbf{C}},g,d)]\\
&-96[\widehat{A}(TM,\nabla^{TM}){\rm ch}(\triangle(M)+2^{9}){\rm ch}(\triangle(E),g^{\triangle(E)},d)]\}^{(15)}\\
        =&9\times2^{\frac{N}{2}+2}\{\widehat{A}(TM,\nabla^{TM}){\rm ch}(\triangle(M)+2^9){\rm ch}(\widetilde{E_\mathbf{C}},g,d)\}^{(15)}\\
        &-2^{\frac{N}{2}-1}\{\widehat{A}(TM,\nabla^{TM}){\rm ch}(\triangle(M)+2^9){\rm ch}(\wedge^2\widetilde{E_\mathbf{C}},g,d)\}^{(15)}.
        \end{split}
    \end{equation}
\end{cor}

By (4.10), we are comparing the coefficients of $q^2$ in (4.15). Then
\begin{thm}
 If $c_3(E,g,d)=0$, then
    \begin{equation}
        \begin{split}
        \{\widehat{A}(&TM,\nabla^{TM})[{\rm ch}(\triangle(M)\otimes(2\widetilde{T_\mathbf{C}M}+\wedge^2\widetilde{T_\mathbf{C}M}\\
&+\widetilde{T_\mathbf{C}M}\otimes\widetilde{T_\mathbf{C}M}+S^2\widetilde{T_\mathbf{C}M})){\rm ch}(\triangle(E),g^{\triangle(E)},d)\\
&+{\rm ch}(\triangle(M)\otimes2\widetilde
{T_{\mathbf{C}}M}){\rm ch}(\triangle(E)\otimes\widetilde{E_\mathbf{C}},g,d)\\
&+{\rm ch}(\triangle(M)){\rm ch}(\triangle(E)\otimes(\wedge^2\widetilde{E_\mathbf{C}}+\widetilde{E_\mathbf{C}}),g,d)]\\
&+2^{2k+1}\widehat{A}(TM,\nabla^{TM})[{\rm ch}(\wedge^4\widetilde{T_\mathbf{C}M}+\wedge^2\widetilde{T_\mathbf{C}M}\otimes
\widetilde{T_\mathbf{C}M}\\
&+\widetilde{T_\mathbf{C}M}\otimes\widetilde{T_\mathbf{C}M}+S^2
\widetilde{T_\mathbf{C}M}+\widetilde{T_\mathbf{C}M}){\rm ch}(\triangle(E),g^{\triangle(E)},d)\\
&+{\rm ch}(\widetilde{T_{
\mathbf{C}}M}+\wedge^2\widetilde{T_{\mathbf{C}}M}){\rm ch}(\triangle(E)\otimes\widetilde{E_\mathbf{C}},g,d)\\
&+{\rm ch}(\triangle(E)\otimes(\wedge^2\widetilde{E_\mathbf{C}}+\widetilde{E_\mathbf{C}})
,g,d)]\\
&+(288k^2+72k)[\widehat{A}(TM,\nabla^{TM}){\rm ch}(\triangle(M)+2^{2k+1}){\rm ch}(\triangle(E),g^{\triangle(E)},d)]\\
&-(24k-8)\{\widehat{A}(TM,\nabla^{TM})[{\rm ch}(\triangle(M)\otimes 2\widetilde{T_\mathbf{C}M}){\rm ch}(\triangle(E),g^{\triangle(E)},d)\\
&+{\rm ch}(\triangle(M)){\rm ch}(\triangle(E)\otimes\widetilde{E_\mathbf{C}},g,d)]+2^{2k+1}\widehat{A}(TM,\nabla^{TM})[{\rm ch}(\widetilde{T_\mathbf{C}M}\\
&+\wedge^2\widetilde{T_\mathbf{C}M}){\rm ch}(\triangle(E),g^{\triangle(E)},d)+{\rm ch}(\triangle(E)\otimes\widetilde{E_\mathbf{C}},g,d)]\}\}^{(4k-1)}\\
=&2^{\frac{N}{2}+k+11}\sum_{s=1}^{[\frac{k}{2}]}s^22^{-6s}h_s.
        \end{split}
    \end{equation}
\end{thm}
\begin{cor}
  Let $M$ be a $(4k-1)$-dimensional spin manifold and $c_3(E,g,d)=0$. If $k$ is even, then
\begin{equation}
  \begin{split}
  {\rm Ind}(&T\otimes(\triangle(M)\otimes((2\widetilde{T_\mathbf{C}M}+\wedge^2\widetilde{T_\mathbf{C}M}
  +\widetilde{T_\mathbf{C}M}\otimes\widetilde{T_\mathbf{C}M}\\
  &+S^2\widetilde{T_\mathbf{C}M})\otimes
(\triangle(E),g^{\triangle(E)},d)+2\widetilde
{T_{\mathbf{C}}M}\otimes(\triangle(E)\otimes\widetilde{E_\mathbf{C}},g,d)\\
&+(\triangle(E)\otimes
(\wedge^2\widetilde{E_\mathbf{C}}+\widetilde{E_\mathbf{C}}),g,d))+2^{2k+1}\otimes((\wedge^4\widetilde
{T_\mathbf{C}M}\\
&+\wedge^2\widetilde{T_\mathbf{C}M}\otimes
\widetilde{T_\mathbf{C}M}+\widetilde{T_\mathbf{C}M}\otimes\widetilde{T_\mathbf{C}M}+S^2
\widetilde{T_\mathbf{C}M}+\widetilde{T_\mathbf{C}M})\otimes(\triangle(E),g^{\triangle(E)},d)\\
&+(\widetilde{T_{\mathbf{C}}M}
+\wedge^2\widetilde{T_{\mathbf{C}}M})\otimes(\triangle(E)
\otimes\widetilde{E_\mathbf{C}},g,d)\\
&+(\triangle(E)\otimes(\wedge^2\widetilde{E_\mathbf{C}}
+\widetilde{E_\mathbf{C}}),g,d))))\\
&+(288k^2+72k){\rm Ind}(T\otimes(\triangle(M)+2^{2k+1})\otimes(\triangle(E),g^{\triangle(E)},d))\\
&-(24k-8){\rm Ind}(T\otimes(\triangle(M)\otimes(2\widetilde{T_\mathbf{C}M}\otimes(\triangle(E)
,g^{\triangle(E)},d)\\
&+(\triangle(E)\otimes\widetilde{E_\mathbf{C}},g,d))+2^{2k+1}\otimes((\widetilde{T_\mathbf{C}M}\\
&+\wedge^2\widetilde{T_\mathbf{C}M})\otimes(\triangle(E),g^{\triangle(E)},d)
+(\triangle(E)\otimes\widetilde{E_\mathbf{C}},g,d))))\\
&\equiv 0~~({\rm mod}~ 2^{\frac{N}{2}-2k+11}),
  \end{split}
\end{equation}
If $k$ is odd, then
\begin{equation}
  \begin{split}
  {\rm Ind}(&T\otimes(\triangle(M)\otimes((2\widetilde{T_\mathbf{C}M}+\wedge^2\widetilde{T_\mathbf{C}M}
  +\widetilde{T_\mathbf{C}M}\otimes\widetilde{T_\mathbf{C}M}\\
  &+S^2\widetilde{T_\mathbf{C}M})\otimes
(\triangle(E),g^{\triangle(E)},d)+2\widetilde
{T_{\mathbf{C}}M}\otimes(\triangle(E)\otimes\widetilde{E_\mathbf{C}},g,d)\\
&+(\triangle(E)\otimes
(\wedge^2\widetilde{E_\mathbf{C}}+\widetilde{E_\mathbf{C}}),g,d))+2^{2k+1}\otimes((\wedge^4\widetilde
{T_\mathbf{C}M}\\
&+\wedge^2\widetilde{T_\mathbf{C}M}\otimes
\widetilde{T_\mathbf{C}M}+\widetilde{T_\mathbf{C}M}\otimes\widetilde{T_\mathbf{C}M}+S^2
\widetilde{T_\mathbf{C}M}+\widetilde{T_\mathbf{C}M})\otimes(\triangle(E),g^{\triangle(E)},d)\\
&+(\widetilde{T_{\mathbf{C}}M}
+\wedge^2\widetilde{T_{\mathbf{C}}M})\otimes(\triangle(E)
\otimes\widetilde{E_\mathbf{C}},g,d)\\
&+(\triangle(E)\otimes(\wedge^2\widetilde{E_\mathbf{C}}
+\widetilde{E_\mathbf{C}}),g,d))))\\
&+(288k^2+72k){\rm Ind}(T\otimes(\triangle(M)+2^{2k+1})\otimes(\triangle(E),g^{\triangle(E)},d))\\
&-(24k-8){\rm Ind}(T\otimes(\triangle(M)\otimes(2\widetilde{T_\mathbf{C}M}\otimes(\triangle(E)
,g^{\triangle(E)},d)\\
&+(\triangle(E)\otimes\widetilde{E_\mathbf{C}},g,d))+2^{2k+1}\otimes((\widetilde{T_\mathbf{C}M}\\
&+\wedge^2\widetilde{T_\mathbf{C}M})\otimes(\triangle(E),g^{\triangle(E)},d)
+(\triangle(E)\otimes\widetilde{E_\mathbf{C}},g,d))))\\
&\equiv 0~~({\rm mod}~ 2^{\frac{N}{2}-2k+14}),
  \end{split}
\end{equation}
\end{cor}
\begin{cor}
Let $M$ be a $11$-dimensional spin manifold. If $c_3(E,g,d)=0$, then
 \begin{equation}
        \begin{split}
        \{\widehat{A}(&TM,\nabla^{TM})[{\rm ch}(\triangle(M)\otimes(2\widetilde{T_\mathbf{C}M}+\wedge^2\widetilde{T_\mathbf{C}M}\\
&+\widetilde{T_\mathbf{C}M}\otimes\widetilde{T_\mathbf{C}M}+S^2\widetilde{T_\mathbf{C}M})){\rm ch}(\triangle(E),g^{\triangle(E)},d)\\
&+{\rm ch}(\triangle(M)\otimes2\widetilde
{T_{\mathbf{C}}M}){\rm ch}(\triangle(E)\otimes\widetilde{E_\mathbf{C}},g,d)\\
&+{\rm ch}(\triangle(M)){\rm ch}(\triangle(E)\otimes(\wedge^2\widetilde{E_\mathbf{C}}+\widetilde{E_\mathbf{C}}),g,d)]\\
&+2^{7}\widehat{A}(TM,\nabla^{TM})[{\rm ch}(\wedge^4\widetilde{T_\mathbf{C}M}+\wedge^2\widetilde{T_\mathbf{C}M}\otimes
\widetilde{T_\mathbf{C}M}\\
&+\widetilde{T_\mathbf{C}M}\otimes\widetilde{T_\mathbf{C}M}+S^2
\widetilde{T_\mathbf{C}M}+\widetilde{T_\mathbf{C}M}){\rm ch}(\triangle(E),g^{\triangle(E)},d)\\
&+{\rm ch}(\widetilde{T_{
\mathbf{C}}M}+\wedge^2\widetilde{T_{\mathbf{C}}M}){\rm ch}(\triangle(E)\otimes\widetilde{E_\mathbf{C}},g,d)\\
&+{\rm ch}(\triangle(E)\otimes(\wedge^2\widetilde{E_\mathbf{C}}+\widetilde{E_\mathbf{C}})
,g,d)]\\
&+2808[\widehat{A}(TM,\nabla^{TM}){\rm ch}(\triangle(M)+2^{7}){\rm ch}(\triangle(E),g^{\triangle(E)},d)]\\
&-64\{\widehat{A}(TM,\nabla^{TM})[{\rm ch}(\triangle(M)\otimes 2\widetilde{T_\mathbf{C}M}){\rm ch}(\triangle(E),g^{\triangle(E)},d)\\
&+{\rm ch}(\triangle(M)){\rm ch}(\triangle(E)\otimes\widetilde{E_\mathbf{C}},g,d)]+2^{7}\widehat{A}(TM,\nabla^{TM})[{\rm ch}(\widetilde{T_\mathbf{C}M}\\
&+\wedge^2\widetilde{T_\mathbf{C}M}){\rm ch}(\triangle(E),g^{\triangle(E)},d)+{\rm ch}(\triangle(E)\otimes\widetilde{E_\mathbf{C}},g,d)]\}\}^{(11)}\\
=&2^{\frac{N}{2}+8}\{\widehat{A}(TM,\nabla^{TM}){\rm ch}(\triangle(M)+2^9){\rm ch}(\widetilde{E_\mathbf{C}},g,d)\}^{(11)}.
        \end{split}
    \end{equation}
\end{cor}
\begin{cor}
  Let $M$ be a $15$-dimensional spin manifold. If $c_3(E,g,d)=0$, then
 \begin{equation}
        \begin{split}
        \{\widehat{A}(&TM,\nabla^{TM})[{\rm ch}(\triangle(M)\otimes(2\widetilde{T_\mathbf{C}M}+\wedge^2\widetilde{T_\mathbf{C}M}\\
&+\widetilde{T_\mathbf{C}M}\otimes\widetilde{T_\mathbf{C}M}+S^2\widetilde{T_\mathbf{C}M})){\rm ch}(\triangle(E),g^{\triangle(E)},d)\\
&+{\rm ch}(\triangle(M)\otimes2\widetilde
{T_{\mathbf{C}}M}){\rm ch}(\triangle(E)\otimes\widetilde{E_\mathbf{C}},g,d)\\
&+{\rm ch}(\triangle(M)){\rm ch}(\triangle(E)\otimes(\wedge^2\widetilde{E_\mathbf{C}}+\widetilde{E_\mathbf{C}}),g,d)]\\
&+2^{9}\widehat{A}(TM,\nabla^{TM})[{\rm ch}(\wedge^4\widetilde{T_\mathbf{C}M}+\wedge^2\widetilde{T_\mathbf{C}M}\otimes
\widetilde{T_\mathbf{C}M}\\
&+\widetilde{T_\mathbf{C}M}\otimes\widetilde{T_\mathbf{C}M}+S^2
\widetilde{T_\mathbf{C}M}+\widetilde{T_\mathbf{C}M}){\rm ch}(\triangle(E),g^{\triangle(E)},d)\\
&+{\rm ch}(\widetilde{T_{
\mathbf{C}}M}+\wedge^2\widetilde{T_{\mathbf{C}}M}){\rm ch}(\triangle(E)\otimes\widetilde{E_\mathbf{C}},g,d)\\
&+{\rm ch}(\triangle(E)\otimes(\wedge^2\widetilde{E_\mathbf{C}}+\widetilde{E_\mathbf{C}})
,g,d)]\\
&+4896[\widehat{A}(TM,\nabla^{TM}){\rm ch}(\triangle(M)+2^{9}){\rm ch}(\triangle(E),g^{\triangle(E)},d)]\\
&-88\{\widehat{A}(TM,\nabla^{TM})[{\rm ch}(\triangle(M)\otimes 2\widetilde{T_\mathbf{C}M}){\rm ch}(\triangle(E),g^{\triangle(E)},d)\\
&+{\rm ch}(\triangle(M)){\rm ch}(\triangle(E)\otimes\widetilde{E_\mathbf{C}},g,d)]+2^{9}\widehat{A}(TM,\nabla^{TM})[{\rm ch}(\widetilde{T_\mathbf{C}M}\\
&+\wedge^2\widetilde{T_\mathbf{C}M}){\rm ch}(\triangle(E),g^{\triangle(E)},d)+{\rm ch}(\triangle(E)\otimes\widetilde{E_\mathbf{C}},g,d)]\}\}^{(15)}\\
=&-7\times2^{\frac{N}{2}+8}\{\widehat{A}(TM,\nabla^{TM}){\rm ch}(\triangle(M)+2^9){\rm ch}(\widetilde{E_\mathbf{C}},g,d)\}^{(15)}\\
&+2^{\frac{N}{2}+5}\{\widehat{A}(TM,\nabla^{TM}){\rm ch}(\triangle(M)+2^9){\rm ch}(\wedge^2\widetilde{E_\mathbf{C}},g,d)\}^{(15)}.
\end{split}
\end{equation}
\end{cor}

Let $M$ be closed oriented ${\rm spin^{c}}$-manifold and $L$ be the complex line bundle associated to the given ${\rm spin^{c}}$ structure on $M.$ Denote by $c=c_1(L)$ the first Chern class of $L.$ Also, we use $L_{\bf{R}}$ for the notation of $L,$ when it is viewed as an oriented real plane bundle.
Let $\Theta(T_{\mathbf{C}}M,L_{\bf{R}}\otimes\bf{C})$ be the virtual complex vector bundle over $M$ defined by
\begin{equation}
    \begin{split}
        \Theta(T_{\mathbf{C}}M,L_{\bf{R}}\otimes\mathbf{C})=&\bigotimes _{n=1}^{\infty}S_{q^n}(\widetilde{T_{\mathbf{C}}M})\otimes
\bigotimes _{m=1}^{\infty}\wedge_{q^m}(\widetilde{L_{\bf{R}}\otimes\mathbf{C}})\\
&\otimes
\bigotimes _{r=1}^{\infty}\wedge_{-q^{r-\frac{1}{2}}}(\widetilde{L_{\bf{R}}\otimes\mathbf{C}})\otimes
\bigotimes _{s=1}^{\infty}\wedge_{q^{s-\frac{1}{2}}}(\widetilde{L_{\bf{R}}\otimes\mathbf{C}}).\nonumber
    \end{split}
\end{equation}
Let ${\rm dim}M=4k-1$ and $y=-\frac{\sqrt{-1}}{2\pi}c.$ Set
\begin{equation}
\begin{split}
\widetilde{Q}_1(&\nabla^{TM},\nabla^{L},g,d,\tau)\\
=&\{\widehat{A}(TM,\nabla^{TM}){\rm exp}(\frac{c}{2}){\rm ch}(\Theta(T_{\mathbf{C}}M,L_{\bf{R}}\otimes\mathbf{C})){\rm ch}(Q_1(E),g^{Q(E)},d,\tau)\}^{(4k-1)}.
\end{split}
\end{equation}
\begin{equation}
\begin{split}
\widetilde{Q}_2(&\nabla^{TM},\nabla^{L},g,d,\tau)\\
=&\{\widehat{A}(TM,\nabla^{TM}){\rm exp}(\frac{c}{2}){\rm ch}(\Theta(T_{\mathbf{C}}M,L_{\bf{R}}\otimes\mathbf{C})){\rm ch}(Q_2(E),g^{Q(E)},d,\tau)\}^{(4k-1)}.
\end{split}
\end{equation}
\begin{equation}
\begin{split}
\widetilde{Q}_3(&\nabla^{TM},\nabla^{L},g,d,\tau)\\
=&\{\widehat{A}(TM,\nabla^{TM}){\rm exp}(\frac{c}{2}){\rm ch}(\Theta(T_{\mathbf{C}}M,L_{\bf{R}}\otimes\mathbf{C})){\rm ch}(Q_3(E),g^{Q(E)},d,\tau)\}^{(4k-1)}.
\end{split}
\end{equation}
Then
\begin{align}
\widetilde{Q}_i(\nabla^{TM},\nabla^{L},g,d,\tau)=&\left\{\left(\prod_{j=1}^{2k-1}\frac{x_j\theta'(0,\tau)}
{\theta(x_j,\tau)}\left(\frac{\theta_1(y,\tau)}{\theta_1(0,\tau)}
\frac{\theta_2(y,\tau)}{\theta_2(0,\tau)}\frac{\theta_3(y,\tau)}{\theta_3(0,\tau)}
\right)\right)\right.\\\notag
&\left.\cdot{\rm ch}(Q_i(E),g^{Q_i(E)},d,\tau)\right\}^{(4k-1)},\ \ \ 1\leq i\leq 3.
\end{align}
Let $\widetilde{P}_1(\tau)=\widetilde{Q}_1(\nabla^{TM},\nabla^{L},g,d,\tau)^{4k-1},$ $\widetilde{P}_2(\tau)=\widetilde{Q}_2(\nabla^{TM},\nabla^{L},g,d,\tau)^{4k-1}.$
By $(2.13)-(2.17)$ and $3p_1(L)-p_1(M)=0,$ then $\widetilde{P}_1(\tau)$ is a modular form of weight $2k$ over $\Gamma_0(2)$, where $\widetilde{P}_2(\tau)$ is a modular form of weight $2k$ over $\Gamma^0(2).$ Moreover, the following identity holds:
\begin{equation}
    \widetilde{P}_1(-\frac{1}{\tau})=2^{\frac{N}{2}}\tau^{2k}\widetilde{P}_2(\tau).
\end{equation}

We can direct computations show that
\begin{equation}
\begin{split}
\widetilde{Q}_1(\nabla^{TM},\nabla^{L},g,d,\tau)=&[\widehat{A}(TM,\nabla^{TM}){\rm exp}(\frac{c}{2}){\rm ch}(\triangle(E),g^{\triangle(E)},d)]^{(4k-1)}\\
&+q\{\widehat{A}(TM,\nabla^{TM}){\rm exp}(\frac{c}{2})[{\rm ch}(\widetilde{T_\mathbf{C}M}+2\wedge^2\widetilde{L_{\bf{R}}\otimes\mathbf{C}}\\
&-(\widetilde{L_{\bf{R}}
\otimes\mathbf{C}})\otimes(\widetilde{L_{\bf{R}}\otimes\mathbf{C}})
    +\widetilde{L_{\bf{R}}\otimes\mathbf{C}}){\rm ch}(\triangle(E),g^{\triangle(E)},d)\\
&+{\rm ch}(\triangle(E)\otimes\widetilde{E_\mathbf{C}},g,d)]\}^{(4k-1)}\\
&+q^2\{\widehat{A}(TM,\nabla^{TM}){\rm exp}(\frac{c}{2})[{\rm ch}(S^2\widetilde{T_\mathbf{C}M}+\widetilde{T_\mathbf{C}M}\\
&+(2\wedge^2\widetilde{L_{\bf{R}}
\otimes\mathbf{C}}-(\widetilde{L_{\bf{R}}
\otimes\mathbf{C}})\otimes(\widetilde{L_{\bf{R}}\otimes\mathbf{C}})
    +\widetilde{L_{\bf{R}}\otimes\mathbf{C}})\otimes\widetilde{T_\mathbf{C}M}\\
    &+\wedge^2\widetilde{L_{\bf{R}}\otimes\mathbf{C}}\otimes\wedge^2\widetilde{L_{\bf{R}}
\otimes\mathbf{C}}+2\wedge^4\widetilde{L_{\bf{R}}\otimes\mathbf{C}}\\
&-2\widetilde{L_{\bf{R}}
\otimes\mathbf{C}}\otimes\wedge^3\widetilde{L_{\bf{R}}
\otimes\mathbf{C}}
+2\widetilde{L_{\bf{R}}
\otimes\mathbf{C}}\otimes\wedge^2\widetilde{L_{\bf{R}}
\otimes\mathbf{C}}\\
&-\widetilde{L_{\bf{R}}
\otimes\mathbf{C}}\otimes\widetilde{L_{\bf{R}}
\otimes\mathbf{C}}\otimes\widetilde{L_{\bf{R}}
\otimes\mathbf{C}}+\widetilde{L_{\bf{R}}
\otimes\mathbf{C}}\\
&+\wedge^2\widetilde{L_{\bf{R}}\otimes\mathbf{C}}){\rm ch}(\triangle(E),g^{\triangle(E)},d)
+{\rm ch}(\widetilde{T_\mathbf{C}M}+2\wedge^2\widetilde{L_{\bf{R}}\otimes\mathbf{C}}\\
&-(\widetilde{L_{\bf{R}}
\otimes\mathbf{C}})\otimes(\widetilde{L_{\bf{R}}\otimes\mathbf{C}})
+\widetilde{L_{\bf{R}}\otimes\mathbf{C}}){\rm ch}(\triangle(E)\otimes\widetilde{E_\mathbf{C}},g,d)\\
&+{\rm ch}(\triangle(E)\otimes(\wedge^2\widetilde{E_\mathbf{C}}
+\widetilde{E_\mathbf{C}}),g,d)]\}^{(4k-1)}+\cdots,
\end{split}
\end{equation}
  \begin{equation}
\begin{split}
\widetilde{Q}_2(\nabla^{TM},\nabla^{L},g,d,\tau)=&-q^{\frac{1}{2}}[{\widehat{A}(TM,\nabla^{TM})}{\rm exp}(\frac{c}{2}){\rm ch}(\widetilde{E_\mathbf{C}},g,d)]^{(4k-1)}\\
&+q[{\widehat{A}(TM,\nabla^{TM})}{\rm exp}(\frac{c}{2}){\rm ch}(\wedge^2\widetilde{E_\mathbf{C}},g,d)]^{(4k-1)}\\
&-q^{\frac{3}{2}}\{{\widehat{A}(TM,\nabla^{TM})}{\rm exp}(\frac{c}{2})[{\rm ch}(\widetilde{T_\mathbf{C}M}+2\wedge^2\widetilde{L_{\bf{R}}\otimes\mathbf{C}}\\
&-(\widetilde{L_{\bf{R}}
\otimes\mathbf{C}})\otimes(\widetilde{L_{\bf{R}}\otimes\mathbf{C}})
    +\widetilde{L_{\bf{R}}\otimes\mathbf{C}}){\rm ch}(\widetilde{E_\mathbf{C}},g,d)\\
&+{\rm ch}(\widetilde{E_\mathbf{C}}+\wedge^3\widetilde{E_\mathbf{C}},g,d)]\}^{(4k-1)}\\
&+q^2\{{\widehat{A}(TM,\nabla^{TM})}{\rm exp}(\frac{c}{2})[{\rm ch}(\widetilde{T_\mathbf{C}M}+2\wedge^2\widetilde{L_{\bf{R}}\otimes\mathbf{C}}\\
&-(\widetilde{L_{\bf{R}}
\otimes\mathbf{C}})\otimes(\widetilde{L_{\bf{R}}\otimes\mathbf{C}})
    +\widetilde{L_{\bf{R}}\otimes\mathbf{C}}){\rm ch}(\wedge^2\widetilde{E_\mathbf{C}},g,d)\\
&+{\rm ch}(\wedge^4\widetilde{E_\mathbf{C}}+\widetilde{E_\mathbf{C}}\otimes\widetilde{E_\mathbf{C}},g,d)]
\}^{(4k-1)}+\cdots.
\end{split}
\end{equation}
Playing the same game as in the proof of Theorem 4.1, we obtain
\begin{thm}
 If $c_3(E,g,d)=0$ and $3p_1(L)-p_1(M)=0$, then
    \begin{equation}
        \{\widehat{A}(TM,\nabla^{TM}){\rm exp}(\frac{c}{2}){\rm ch}(\triangle(E),g^{\triangle(E)},d)\}^{(4k-1)}
        =2^{\frac{N}{2}+k}\sum_{s=1}^{[\frac{k}{2}]}2^{-6s}h_s,
    \end{equation}
where each $h_s, 1\leq s\leq [\frac{k}{2}],$ is a canonical integral linear combination of
$$\{\widehat{A}(TM){\rm exp}(\frac{c}{2})\widetilde{{\rm ch}}(B_{\alpha}(T_{\mathbf{C}}M,L_{\bf{R}}
\otimes\mathbf{C},E_{\mathbf{C}}))\}^{(4k-1)},\ \ \ 0\leq \alpha \leq s.$$
\end{thm}
\begin{cor}
  Let $M$ be a $(4k-1)$-dimensional ${\rm spin}^c$ manifold. If $c_3(E,g,d)=0$ and $3p_1(L)-p_1(M)=0$, then
  \begin{equation}
    \begin{split}
      {\rm Ind}(T^c\otimes(\triangle(E),g^{\triangle(E)},d))
=-2^{\frac{N}{2}+k}\sum_{s=0}^{[\frac{k}{2}]}2^{-6s}h_s,
    \end{split}
  \end{equation}
where each $h_l$,~$1\leq s\leq [\frac{k}{2}]$, is a canonical
integral linear combination of
\begin{equation}
\begin{split}
{\rm Ind}(T^c\otimes(\widetilde{B}_{\alpha}(T_{\mathbf{C}}M,L_{\bf{R}}\otimes\mathbf{C},E_{\mathbf{C}}))),
\end{split}
\end{equation}
and $T^c$ is the ${\rm spin}^c$ Toeplitz operator.
\end{cor}
\begin{cor}
  Let $M$ be a $11$-dimensional ${\rm spin}^c$ manifold. If $c_3(E,g,d)=0$ and $3p_1(L)-p_1(M)=0$, then
  \begin{equation}
        \begin{split}
        \{\widehat{A}(&TM,\nabla^{TM}){\rm exp}(\frac{c}{2}){\rm ch}(\triangle(E),g^{\triangle(E)},d)\}^{(11)}\\
        &=2^{\frac{N}{2}-3}\{\widehat{A}(TM,\nabla^{TM}){\rm exp}(\frac{c}{2}){\rm ch}(\widetilde{E_\mathbf{C}},g,d)\}^{(11)}.
        \end{split}
    \end{equation}
\end{cor}
\begin{cor}
  Let $M$ be a $15$-dimensional ${\rm spin}^c$ manifold. If $c_3(E,g,d)=0$ and $3p_1(L)-p_1(M)=0$, then
  \begin{equation}
        \begin{split}
        \{\widehat{A}(&TM,\nabla^{TM}){\rm exp}(\frac{c}{2}){\rm ch}(\triangle(E),g^{\triangle(E)},d)\}^{(15)}\\
        =&-13\times2^{\frac{N}{2}-5}\{\widehat{A}(TM,\nabla^{TM}){\rm exp}(\frac{c}{2}){\rm ch}(\widetilde{E_\mathbf{C}},g,d)\}^{(15)}\\
        &+2^{\frac{N}{2}-8}\{\widehat{A}(TM,\nabla^{TM}){\rm exp}(\frac{c}{2}){\rm ch}(\wedge^2\widetilde{E_\mathbf{C}},g,d)\}^{(15)}.
        \end{split}
    \end{equation}
\end{cor}
By comparing the coefficients of $q$,~$q^2$ in (4.39). We obtain
\begin{thm}
 If $c_3(E,g,d)=0$ and $3p_1(L)-p_1(M)=0$, then
    \begin{equation}
        \begin{split}
        \{\widehat{A}&(TM,\nabla^{TM}){\rm exp}(\frac{c}{2})[{\rm ch}(\widetilde{T_\mathbf{C}M}+2\wedge^2\widetilde{L_{\bf{R}}\otimes\mathbf{C}}
-(\widetilde{L_{\bf{R}}
\otimes\mathbf{C}})\otimes(\widetilde{L_{\bf{R}}\otimes\mathbf{C}})\\
    &+\widetilde{L_{\bf{R}}\otimes\mathbf{C}}){\rm ch}(\triangle(E),g^{\triangle(E)},d)
+{\rm ch}(\triangle(E)\otimes\widetilde{E_\mathbf{C}},g,d)]\\
&-24k[\widehat{A}(TM,\nabla^{TM}){\rm exp}(\frac{c}{2}){\rm ch}(\triangle(E),g^{\triangle(E)},d)]\}^{(4k-1)}\\
        =&-2^{\frac{N}{2}+k+6}\sum_{s=1}^{[\frac{k}{2}]}s2^{-6s}h_s.
        \end{split}
    \end{equation}
\end{thm}
\begin{cor}
  Let $M$ be a $(4k-1)$-dimensional ${\rm spin}^c$ manifold. If $c_3(E,g,d)=0$ and $3p_1(L)-p_1(M)=0$, then
  \begin{equation}
    \begin{split}
      {\rm Ind}(T^c&\otimes((\widetilde{T_\mathbf{C}M}+2\wedge^2\widetilde{L_{\bf{R}}\otimes\mathbf{C}}
-(\widetilde{L_{\bf{R}}
\otimes\mathbf{C}})\otimes(\widetilde{L_{\bf{R}}\otimes\mathbf{C}})\\
    &+\widetilde{L_{\bf{R}}\otimes\mathbf{C}})\otimes(\triangle(E),g^{\triangle(E)},d)
    +(\triangle(E)\otimes\widetilde{E_\mathbf{C}},g,d)))\\
    &-24k{\rm Ind}(T^c\otimes(\triangle(E),g^{\triangle(E)},d))\\
=&2^{\frac{N}{2}+k+6}\sum_{s=1}^{[\frac{k}{2}]}s2^{-6s}h_s.
    \end{split}
  \end{equation}
\end{cor}
\begin{cor}
  Let $M$ be a $11$-dimensional ${\rm spin}^c$ manifold. If $c_3(E,g,d)=0$ and $3p_1(L)-p_1(M)=0$, then
  \begin{equation}
        \begin{split}
        \{\widehat{A}&(TM,\nabla^{TM}){\rm exp}(\frac{c}{2})[{\rm ch}(\widetilde{T_\mathbf{C}M}+2\wedge^2\widetilde{L_{\bf{R}}\otimes\mathbf{C}}
-(\widetilde{L_{\bf{R}}
\otimes\mathbf{C}})\otimes(\widetilde{L_{\bf{R}}\otimes\mathbf{C}})\\
    &+\widetilde{L_{\bf{R}}\otimes\mathbf{C}}){\rm ch}(\triangle(E),g^{\triangle(E)},d)
+{\rm ch}(\triangle(E)\otimes\widetilde{E_\mathbf{C}},g,d)]\\
&-72[\widehat{A}(TM,\nabla^{TM}){\rm exp}(\frac{c}{2}){\rm ch}(\triangle(E),g^{\triangle(E)},d)]\}^{(11)}\\
        =&-2^{\frac{N}{2}+3}\{\widehat{A}(TM,\nabla^{TM}){\rm exp}(\frac{c}{2}){\rm ch}(\widetilde{E_\mathbf{C}},g,d)\}^{(11)}.
        \end{split}
    \end{equation}
\end{cor}
\begin{cor}
  Let $M$ be a $15$-dimensional ${\rm spin}^c$ manifold. If $c_3(E,g,d)=0$ and $3p_1(L)-p_1(M)=0$, then
  \begin{equation}
        \begin{split}
        \{\widehat{A}&(TM,\nabla^{TM}){\rm exp}(\frac{c}{2})[{\rm ch}(\widetilde{T_\mathbf{C}M}+2\wedge^2\widetilde{L_{\bf{R}}\otimes\mathbf{C}}
-(\widetilde{L_{\bf{R}}
\otimes\mathbf{C}})\otimes(\widetilde{L_{\bf{R}}\otimes\mathbf{C}})\\
    &+\widetilde{L_{\bf{R}}\otimes\mathbf{C}}){\rm ch}(\triangle(E),g^{\triangle(E)},d)
+{\rm ch}(\triangle(E)\otimes\widetilde{E_\mathbf{C}},g,d)]\\
&-96[\widehat{A}(TM,\nabla^{TM}){\rm exp}(\frac{c}{2}){\rm ch}(\triangle(E),g^{\triangle(E)},d)]\}^{(15)}\\
        =&9\times2^{\frac{N}{2}+2}\{\widehat{A}(TM,\nabla^{TM}){\rm exp}(\frac{c}{2}){\rm ch}(\widetilde{E_\mathbf{C}},g,d)\}^{(15)}\\
        &-2^{\frac{N}{2}-1}\{\widehat{A}(TM,\nabla^{TM}){\rm exp}(\frac{c}{2}){\rm ch}(\wedge^2\widetilde{E_\mathbf{C}},g,d)\}^{(15)}.
        \end{split}
    \end{equation}
\end{cor}
\begin{thm}
 If $c_3(E,g,d)=0$ and $3p_1(L)-p_1(M)=0$, then
    \begin{equation}
        \begin{split}
        \{\widehat{A}(&TM,\nabla^{TM}){\rm exp}(\frac{c}{2})[{\rm ch}(S^2\widetilde{T_\mathbf{C}M}+\widetilde{T_\mathbf{C}M}\\
&+(2\wedge^2\widetilde{L_{\bf{R}}
\otimes\mathbf{C}}-(\widetilde{L_{\bf{R}}
\otimes\mathbf{C}})\otimes(\widetilde{L_{\bf{R}}\otimes\mathbf{C}})
    +\widetilde{L_{\bf{R}}\otimes\mathbf{C}})\otimes\widetilde{T_\mathbf{C}M}\\
    &+\wedge^2\widetilde{L_{\bf{R}}\otimes\mathbf{C}}\otimes\wedge^2\widetilde{L_{\bf{R}}
\otimes\mathbf{C}}+2\wedge^4\widetilde{L_{\bf{R}}\otimes\mathbf{C}}\\
&-2\widetilde{L_{\bf{R}}
\otimes\mathbf{C}}\otimes\wedge^3\widetilde{L_{\bf{R}}
\otimes\mathbf{C}}
+2\widetilde{L_{\bf{R}}
\otimes\mathbf{C}}\otimes\wedge^2\widetilde{L_{\bf{R}}
\otimes\mathbf{C}}\\
&-\widetilde{L_{\bf{R}}
\otimes\mathbf{C}}\otimes\widetilde{L_{\bf{R}}
\otimes\mathbf{C}}\otimes\widetilde{L_{\bf{R}}
\otimes\mathbf{C}}+\widetilde{L_{\bf{R}}
\otimes\mathbf{C}}\\
&+\wedge^2\widetilde{L_{\bf{R}}\otimes\mathbf{C}}){\rm ch}(\triangle(E),g^{\triangle(E)},d)
+{\rm ch}(\widetilde{T_\mathbf{C}M}+2\wedge^2\widetilde{L_{\bf{R}}\otimes\mathbf{C}}\\
&-(\widetilde{L_{\bf{R}}
\otimes\mathbf{C}})\otimes(\widetilde{L_{\bf{R}}\otimes\mathbf{C}})
+\widetilde{L_{\bf{R}}\otimes\mathbf{C}}){\rm ch}(\triangle(E)\otimes\widetilde{E_\mathbf{C}},g,d)\\
&+{\rm ch}(\triangle(E)\otimes(\wedge^2\widetilde{E_\mathbf{C}}
+\widetilde{E_\mathbf{C}}),g,d)]\\
&+(288k^2+72k)[\widehat{A}(TM,\nabla^{TM}){\rm exp}(\frac{c}{2}){\rm ch}(\triangle(E),g^{\triangle(E)},d)]\\
&-(24k-8)\{\widehat{A}(TM,\nabla^{TM}){\rm exp}(\frac{c}{2})[{\rm ch}(\widetilde{T_\mathbf{C}M}\\
&+2\wedge^2\widetilde{L_{\bf{R}}\otimes\mathbf{C}}
-(\widetilde{L_{\bf{R}}
\otimes\mathbf{C}})\otimes(\widetilde{L_{\bf{R}}\otimes\mathbf{C}})\\
    &+\widetilde{L_{\bf{R}}\otimes\mathbf{C}}){\rm ch}(\triangle(E),g^{\triangle(E)},d)
+{\rm ch}(\triangle(E)\otimes\widetilde{E_\mathbf{C}},g,d)]\}\}^{(4k-1)}\\
=&2^{\frac{N}{2}+k+11}\sum_{s=1}^{[\frac{k}{2}]}s^22^{-6s}h_s.
        \end{split}
    \end{equation}
\end{thm}
\begin{cor}
  Let $M$ be a $(4k-1)$-dimensional ${\rm spin}^c$ manifold. If $c_3(E,g,d)=0$ and $3p_1(L)-p_1(M)=0$, then
  \begin{equation}
    \begin{split}
      {\rm Ind}(T^c&\otimes((S^2\widetilde{T_\mathbf{C}M}+\widetilde{T_\mathbf{C}M}
+(2\wedge^2\widetilde{L_{\bf{R}}
\otimes\mathbf{C}}-(\widetilde{L_{\bf{R}}
\otimes\mathbf{C}})\otimes(\widetilde{L_{\bf{R}}\otimes\mathbf{C}})\\
    &+\widetilde{L_{\bf{R}}\otimes\mathbf{C}})\otimes\widetilde{T_\mathbf{C}M}
    +\wedge^2\widetilde{L_{\bf{R}}\otimes\mathbf{C}}\otimes\wedge^2\widetilde{L_{\bf{R}}
\otimes\mathbf{C}}+2\wedge^4\widetilde{L_{\bf{R}}\otimes\mathbf{C}}\\
&-2\widetilde{L_{\bf{R}}
\otimes\mathbf{C}}\otimes\wedge^3\widetilde{L_{\bf{R}}
\otimes\mathbf{C}}
+2\widetilde{L_{\bf{R}}
\otimes\mathbf{C}}\otimes\wedge^2\widetilde{L_{\bf{R}}
\otimes\mathbf{C}}\\
&-\widetilde{L_{\bf{R}}
\otimes\mathbf{C}}\otimes\widetilde{L_{\bf{R}}
\otimes\mathbf{C}}\otimes\widetilde{L_{\bf{R}}
\otimes\mathbf{C}}+\widetilde{L_{\bf{R}}
\otimes\mathbf{C}}\\
&+\wedge^2\widetilde{L_{\bf{R}}\otimes\mathbf{C}})\otimes(\triangle(E),g^{\triangle(E)},d)
+(\widetilde{T_\mathbf{C}M}+2\wedge^2\widetilde{L_{\bf{R}}\otimes\mathbf{C}}\\
&-(\widetilde{L_{\bf{R}}
\otimes\mathbf{C}})\otimes(\widetilde{L_{\bf{R}}\otimes\mathbf{C}})
+\widetilde{L_{\bf{R}}\otimes\mathbf{C}})\otimes(\triangle(E)\otimes\widetilde{E_\mathbf{C}},g,d)\\
&+(\triangle(E)\otimes(\wedge^2\widetilde{E_\mathbf{C}}
+\widetilde{E_\mathbf{C}}),g,d)))\\
&+(288k^2+72k){\rm Ind}(T^c\otimes(\triangle(E),g^{\triangle(E)},d))\\
&-(24k-8){\rm Ind}(T^c\otimes((\widetilde{T_\mathbf{C}M}+2\wedge^2\widetilde{L_{\bf{R}}\otimes\mathbf{C}}
-(\widetilde{L_{\bf{R}}
\otimes\mathbf{C}})\otimes(\widetilde{L_{\bf{R}}\otimes\mathbf{C}})\\
    &+\widetilde{L_{\bf{R}}\otimes\mathbf{C}})\otimes(\triangle(E),g^{\triangle(E)},d)
    +(\triangle(E)\otimes\widetilde{E_\mathbf{C}},g,d)))\\
=&-2^{\frac{N}{2}+k+11}\sum_{s=1}^{[\frac{k}{2}]}s^22^{-6s}h_s.
    \end{split}
  \end{equation}
\end{cor}
\begin{cor}
Let $M$ be a $11$-dimensional ${\rm spin}^c$ manifold. If $c_3(E,g,d)=0$ and $3p_1(L)-p_1(M)=0$, then
 \begin{equation}
        \begin{split}
        \{\widehat{A}(&TM,\nabla^{TM}){\rm exp}(\frac{c}{2})[{\rm ch}(S^2\widetilde{T_\mathbf{C}M}+\widetilde{T_\mathbf{C}M}+(2\wedge^2\widetilde{L_{\bf{R}}
\otimes\mathbf{C}}-(\widetilde{L_{\bf{R}}
\otimes\mathbf{C}})\otimes(\widetilde{L_{\bf{R}}\otimes\mathbf{C}})\\
    &+\widetilde{L_{\bf{R}}\otimes\mathbf{C}})\otimes\widetilde{T_\mathbf{C}M}
    +\wedge^2\widetilde{L_{\bf{R}}\otimes\mathbf{C}}\otimes\wedge^2\widetilde{L_{\bf{R}}
\otimes\mathbf{C}}+2\wedge^4\widetilde{L_{\bf{R}}\otimes\mathbf{C}}+\widetilde{L_{\bf{R}}
\otimes\mathbf{C}}\\
&-2\widetilde{L_{\bf{R}}
\otimes\mathbf{C}}\otimes\wedge^3\widetilde{L_{\bf{R}}
\otimes\mathbf{C}}
+2\widetilde{L_{\bf{R}}
\otimes\mathbf{C}}\otimes\wedge^2\widetilde{L_{\bf{R}}
\otimes\mathbf{C}}\\
&-\widetilde{L_{\bf{R}}
\otimes\mathbf{C}}\otimes\widetilde{L_{\bf{R}}
\otimes\mathbf{C}}\otimes\widetilde{L_{\bf{R}}
\otimes\mathbf{C}}
+\wedge^2\widetilde{L_{\bf{R}}\otimes\mathbf{C}}){\rm ch}(\triangle(E),g^{\triangle(E)},d)\\
&+{\rm ch}(\widetilde{T_\mathbf{C}M}+2\wedge^2\widetilde{L_{\bf{R}}\otimes\mathbf{C}}
-(\widetilde{L_{\bf{R}}
\otimes\mathbf{C}})\otimes(\widetilde{L_{\bf{R}}\otimes\mathbf{C}})\\
&+\widetilde{L_{\bf{R}}\otimes\mathbf{C}}){\rm ch}(\triangle(E)\otimes\widetilde{E_\mathbf{C}},g,d)
+{\rm ch}(\triangle(E)\otimes(\wedge^2\widetilde{E_\mathbf{C}}
+\widetilde{E_\mathbf{C}}),g,d)]\\
&+2808[\widehat{A}(TM,\nabla^{TM}){\rm exp}(\frac{c}{2}){\rm ch}(\triangle(E),g^{\triangle(E)},d)]\\
&-64\{\widehat{A}(TM,\nabla^{TM}){\rm exp}(\frac{c}{2})[{\rm ch}(\widetilde{T_\mathbf{C}M}
+2\wedge^2\widetilde{L_{\bf{R}}\otimes\mathbf{C}}
-(\widetilde{L_{\bf{R}}
\otimes\mathbf{C}})\otimes(\widetilde{L_{\bf{R}}\otimes\mathbf{C}})\\
    &+\widetilde{L_{\bf{R}}\otimes\mathbf{C}}){\rm ch}(\triangle(E),g^{\triangle(E)},d)
+{\rm ch}(\triangle(E)\otimes\widetilde{E_\mathbf{C}},g,d)]\}\}^{(11)}\\
=&2^{\frac{N}{2}+8}\{\widehat{A}(TM,\nabla^{TM}){\rm exp}(\frac{c}{2}){\rm ch}(\widetilde{E_\mathbf{C}},g,d)\}^{(11)}.
        \end{split}
    \end{equation}
\end{cor}
\begin{cor}
  Let $M$ be a $15$-dimensional ${\rm spin}^c$ manifold. If $c_3(E,g,d)=0$ and $3p_1(L)-p_1(M)=0$, then
 \begin{equation}
        \begin{split}
        \{\widehat{A}(&TM,\nabla^{TM}){\rm exp}(\frac{c}{2})[{\rm ch}(S^2\widetilde{T_\mathbf{C}M}+\widetilde{T_\mathbf{C}M}\\
&+(2\wedge^2\widetilde{L_{\bf{R}}
\otimes\mathbf{C}}-(\widetilde{L_{\bf{R}}
\otimes\mathbf{C}})\otimes(\widetilde{L_{\bf{R}}\otimes\mathbf{C}})
    +\widetilde{L_{\bf{R}}\otimes\mathbf{C}})\otimes\widetilde{T_\mathbf{C}M}\\
    &+\wedge^2\widetilde{L_{\bf{R}}\otimes\mathbf{C}}\otimes\wedge^2\widetilde{L_{\bf{R}}
\otimes\mathbf{C}}+2\wedge^4\widetilde{L_{\bf{R}}\otimes\mathbf{C}}\\
&-2\widetilde{L_{\bf{R}}
\otimes\mathbf{C}}\otimes\wedge^3\widetilde{L_{\bf{R}}
\otimes\mathbf{C}}
+2\widetilde{L_{\bf{R}}
\otimes\mathbf{C}}\otimes\wedge^2\widetilde{L_{\bf{R}}
\otimes\mathbf{C}}\\
&-\widetilde{L_{\bf{R}}
\otimes\mathbf{C}}\otimes\widetilde{L_{\bf{R}}
\otimes\mathbf{C}}\otimes\widetilde{L_{\bf{R}}
\otimes\mathbf{C}}+\widetilde{L_{\bf{R}}
\otimes\mathbf{C}}\\
&+\wedge^2\widetilde{L_{\bf{R}}\otimes\mathbf{C}}){\rm ch}(\triangle(E),g^{\triangle(E)},d)
+{\rm ch}(\widetilde{T_\mathbf{C}M}+2\wedge^2\widetilde{L_{\bf{R}}\otimes\mathbf{C}}\\
&-(\widetilde{L_{\bf{R}}
\otimes\mathbf{C}})\otimes(\widetilde{L_{\bf{R}}\otimes\mathbf{C}})
+\widetilde{L_{\bf{R}}\otimes\mathbf{C}}){\rm ch}(\triangle(E)\otimes\widetilde{E_\mathbf{C}},g,d)\\
&+{\rm ch}(\triangle(E)\otimes(\wedge^2\widetilde{E_\mathbf{C}}
+\widetilde{E_\mathbf{C}}),g,d)]\\
&+4896[\widehat{A}(TM,\nabla^{TM}){\rm exp}(\frac{c}{2}){\rm ch}(\triangle(E),g^{\triangle(E)},d)]\\
&-88\{\widehat{A}(TM,\nabla^{TM}){\rm exp}(\frac{c}{2})[{\rm ch}(\widetilde{T_\mathbf{C}M}\\
&+2\wedge^2\widetilde{L_{\bf{R}}\otimes\mathbf{C}}
-(\widetilde{L_{\bf{R}}
\otimes\mathbf{C}})\otimes(\widetilde{L_{\bf{R}}\otimes\mathbf{C}})\\
    &+\widetilde{L_{\bf{R}}\otimes\mathbf{C}}){\rm ch}(\triangle(E),g^{\triangle(E)},d)
+{\rm ch}(\triangle(E)\otimes\widetilde{E_\mathbf{C}},g,d)]\}\}^{(15)}\\
=&-7\times2^{\frac{N}{2}+8}\{\widehat{A}(TM,\nabla^{TM}){\rm exp}(\frac{c}{2}){\rm ch}(\widetilde{E_\mathbf{C}},g,d)\}^{(15)}\\
&+2^{\frac{N}{2}+5}\{\widehat{A}(TM,\nabla^{TM}){\rm exp}(\frac{c}{2}){\rm ch}(\wedge^2\widetilde{E_\mathbf{C}},g,d)\}^{(15)}.
\end{split}
\end{equation}
\end{cor}

Let ${\rm dim}M=4k+1$ and $y=-\frac{\sqrt{-1}}{2\pi}c.$ Set
\begin{equation}
\begin{split}
\bar{Q}_1(&\nabla^{TM},\nabla^{L},g,d,\tau)\\
=&\{\widehat{A}(TM,\nabla^{TM}){\rm exp}(\frac{c}{2}){\rm ch}(\Theta^*(T_{\mathbf{C}}M,L_{\bf{R}}\otimes\mathbf{C})){\rm ch}(Q_1(E),g^{Q(E)},d,\tau)\}^{(4k+1)},
\end{split}
\end{equation}
\begin{equation}
\begin{split}
\bar{Q}_2(&\nabla^{TM},\nabla^{L},g,d,\tau)\\
=&\{\widehat{A}(TM,\nabla^{TM}){\rm exp}(\frac{c}{2}){\rm ch}(\Theta^*(T_{\mathbf{C}}M,L_{\bf{R}}\otimes\mathbf{C})){\rm ch}(Q_2(E),g^{Q(E)},d,\tau)\}^{4k+1},
\end{split}
\end{equation}
\begin{equation}
\begin{split}
\bar{Q}_3(&\nabla^{TM},\nabla^{L},g,d,\tau)\\
=&\{\widehat{A}(TM,\nabla^{TM}){\rm exp}(\frac{c}{2}){\rm ch}(\Theta^*(T_{\mathbf{C}}M,L_{\bf{R}}\otimes\mathbf{C})){\rm ch}(Q_3(E),g^{Q(E)},d,\tau)\}^{(4k+1)},
\end{split}
\end{equation}
where
$$\Theta^*(T_{\mathbf{C}}M,L_{\bf{R}}\otimes\mathbf{C})=\bigotimes _{n=1}^{\infty}S_{q^n}(\widetilde{T_{\mathbf{C}}M})\otimes
\bigotimes _{m=1}^{\infty}\wedge_{-q^m}(\widetilde{L_{\bf{R}}\otimes\mathbf{C}})$$
be the virtual complex vector bundle over $M$.\\
Then
\begin{align}
\bar{Q}_i(\nabla^{TM},\nabla^{L},g,d,\tau)=&\left\{\left\{\left(\prod_{j=1}^{2k+1}\frac{x_j\theta'(0,\tau)}{\theta(x_j,\tau)}\right)
\frac{\sqrt{-1}\theta(y,\tau)}{\theta_1(0,\tau)\theta_2(0,\tau)
\theta_3(0,\tau)}
\right\}\right.\\\notag
&\left.\cdot{\rm ch}(Q_i(E),g^{Q(E)},d,\tau)\right\}^{(4k+1)},\ \ \ 1\leq i\leq 3
\end{align}
Let $\bar{P}_1(\tau)=\bar{Q}_i(\nabla^{TM},\nabla^{L},g,d,\tau)^{4k+1},$ $\bar{P}_2(\tau)=\bar{Q}_i(\nabla^{TM},\nabla^{L},g,d,\tau)^{4k+1}.$
By $(2.13)-(2.17)$ and $p_1(L)-p_1(M)=0,$ then $\bar{P}_1(\tau)$ is a modular form of weight $2k$ over $\Gamma_0(2)$, where $\bar{P}_2(\tau)$ is a modular form of weight $2k$ over $\Gamma^0(2).$ Moreover, the following identity holds:
\begin{equation}
    \bar{P}_1(-\frac{1}{\tau})=2^{\frac{N}{2}}\tau^{2k}\bar{P}_2(\tau).
\end{equation}
We can direct computations show that
\begin{equation}
\begin{split}
\bar{Q}_1(&\nabla^{TM},\nabla^{L},g,d,\tau)\\
=&[\widehat{A}(TM,\nabla^{TM}){\rm exp}(\frac{c}{2}){\rm ch}(\triangle(E),g^{\triangle(E)},d)]^{(4k+1)}\\
&+q\{\widehat{A}(TM,\nabla^{TM}){\rm exp}(\frac{c}{2})[{\rm ch}(\widetilde{T_\mathbf{C}M}-\widetilde{L_{\bf{R}}\otimes\mathbf{C}}){\rm ch}(\triangle(E),g^{\triangle(E)},d)\\
&+{\rm ch}(\triangle(E)\otimes\widetilde{E_\mathbf{C}},g,d)]\}^{(4k+1)}
+q^2\{\widehat{A}(TM,\nabla^{TM}){\rm exp}(\frac{c}{2})[{\rm ch}(S^2\widetilde{T_\mathbf{C}M}+\widetilde{T_\mathbf{C}M}\\
&+\wedge^2\widetilde{L_{\bf{R}}
\otimes\mathbf{C}}-\widetilde{L_{\bf{R}}
\otimes\mathbf{C}}+\widetilde{T_\mathbf{C}M}\otimes\widetilde{L_{\bf{R}}
\otimes\mathbf{C}}){\rm ch}(\triangle(E),g^{\triangle(E)},d)\\
&+{\rm ch}(\widetilde{T_\mathbf{C}M}-\widetilde{L_{\bf{R}}
\otimes\mathbf{C}}{\rm ch}(\triangle(E)\otimes\widetilde{E_\mathbf{C}},g,d)\\
&+{\rm ch}(\triangle(E)\otimes(\wedge^2\widetilde{E_\mathbf{C}}
+\widetilde{E_\mathbf{C}}),g,d))]\}^{(4k+1)}+\cdots,
\end{split}
\end{equation}
  \begin{equation}
\begin{split}
\bar{Q}_2(&\nabla^{TM},\nabla^{L},g,d,\tau)\\
=&-q^{\frac{1}{2}}[{\widehat{A}(TM,\nabla^{TM})}{\rm exp}(\frac{c}{2}){\rm ch}(\widetilde{E_\mathbf{C}},g,d)]^{(4k+1)}\\
&+q[{\widehat{A}(TM,\nabla^{TM})}{\rm exp}(\frac{c}{2}){\rm ch}(\wedge^2\widetilde{E_\mathbf{C}},g,d)]^{(4k+1)}\\
&-q^{\frac{3}{2}}\{{\widehat{A}(TM,\nabla^{TM})}{\rm exp}(\frac{c}{2})[{\rm ch}(\widetilde{T_\mathbf{C}M}-\widetilde{L_{\bf{R}}\otimes\mathbf{C}}){\rm ch}(\widetilde{E_\mathbf{C}},g,d)\\
&+{\rm ch}(\widetilde{E_\mathbf{C}}+\wedge^3\widetilde{E_\mathbf{C}},g,d)]\}^{(4k+1)}
+q^2\{{\widehat{A}(TM,\nabla^{TM})}{\rm exp}(\frac{c}{2})[{\rm ch}(\widetilde{T_\mathbf{C}M}\\
&-\widetilde{L_{\bf{R}}\otimes\mathbf{C}}){\rm ch}(\wedge^2\widetilde{E_\mathbf{C}},g,d)
+{\rm ch}(\wedge^4\widetilde{E_\mathbf{C}}+\widetilde{E_\mathbf{C}}\otimes\widetilde{E_\mathbf{C}},g,d)]
\}^{(4k+1)}+\cdots.
\end{split}
\end{equation}
Playing the same game as in the proof of Theorem 4.1, we obtain
\begin{thm}
 If $c_3(E,g,d)=0$ and $p_1(L)-p_1(M)=0$, then
    \begin{equation}
        \{\widehat{A}(TM,\nabla^{TM}){\rm exp}(\frac{c}{2}){\rm ch}(\triangle(E),g^{\triangle(E)},d)\}^{(4k+1)}
        =2^{\frac{N}{2}+k}\sum_{s=1}^{[\frac{k}{2}]}2^{-6s}h_s,
    \end{equation}
where each $h_s, 1\leq s\leq [\frac{k}{2}],$ is a canonical integral linear combination of
$$\{\widehat{A}(TM){\rm exp}(\frac{c}{2})\widetilde{{\rm ch}}(B_{\alpha}(T_{\mathbf{C}}M,L_{\bf{R}}
\otimes\mathbf{C},E_{\mathbf{C}}))\}^{(4k+1)},\ \ \ 0\leq \alpha \leq s.$$
\end{thm}
\begin{cor}
  Let $M$ be a $(4k+1)$-dimensional ${\rm spin}^c$ manifold. If $c_3(E,g,d)=0$ and $p_1(L)-p_1(M)=0$, then
  \begin{equation}
    \begin{split}
      {\rm Ind}(T^c\otimes(\triangle(E),g^{\triangle(E)},d))
=-2^{\frac{N}{2}+k}\sum_{s=0}^{[\frac{k}{2}]}2^{-6s}h_s,
    \end{split}
  \end{equation}
where each $h_l$,~$1\leq s\leq [\frac{k}{2}]$, is a canonical
integral linear combination of
\begin{equation}
\begin{split}
{\rm Ind}(T^c\otimes(\bar{B}_{\alpha}(T_{\mathbf{C}}M,L_{\bf{R}}\otimes\mathbf{C},E_{\mathbf{C}}))).
\end{split}
\end{equation}
\end{cor}
\begin{cor}
  Let $M$ be a $13$-dimensional ${\rm spin}^c$ manifold. If $c_3(E,g,d)=0$ and $p_1(L)-p_1(M)=0$, then
  \begin{equation}
        \begin{split}
        \{\widehat{A}(&TM,\nabla^{TM}){\rm exp}(\frac{c}{2}){\rm ch}(\triangle(E),g^{\triangle(E)},d)\}^{(13)}\\
        &=2^{\frac{N}{2}-3}\{\widehat{A}(TM,\nabla^{TM}){\rm exp}(\frac{c}{2}){\rm ch}(\widetilde{E_\mathbf{C}},g,d)\}^{(13)}.
        \end{split}
    \end{equation}
\end{cor}
\begin{cor}
  Let $M$ be a $17$-dimensional ${\rm spin}^c$ manifold. If $c_3(E,g,d)=0$ and $p_1(E)+p_1(L)-p_1(M)=0$, then
  \begin{equation}
        \begin{split}
        \{\widehat{A}(&TM,\nabla^{TM}){\rm exp}(\frac{c}{2}){\rm ch}(\triangle(E),g^{\triangle(E)},d)\}^{(17)}\\
        =&-13\times2^{\frac{N}{2}-5}\{\widehat{A}(TM,\nabla^{TM}){\rm exp}(\frac{c}{2}){\rm ch}(\widetilde{E_\mathbf{C}},g,d)\}^{(17)}\\
        &+2^{\frac{N}{2}-8}\{\widehat{A}(TM,\nabla^{TM}){\rm exp}(\frac{c}{2}){\rm ch}(\wedge^2\widetilde{E_\mathbf{C}},g,d)\}^{(17)}.
        \end{split}
    \end{equation}
\end{cor}
By comparing the coefficients of $q$,~$q^2$ in (4.59). We obtain the following theorem.
\begin{thm}
 If $c_3(E,g,d)=0$ and $p_1(L)-p_1(M)=0$, then
    \begin{equation}
        \begin{split}
        \{\widehat{A}&(TM,\nabla^{TM}){\rm exp}(\frac{c}{2})[{\rm ch}(\widetilde{T_\mathbf{C}M}-\widetilde{L_{\bf{R}}\otimes\mathbf{C}}){\rm ch}(\triangle(E),g^{\triangle(E)},d)\\
&+{\rm ch}(\triangle(E)\otimes\widetilde{E_\mathbf{C}},g,d)]
-24k[\widehat{A}(TM,\nabla^{TM}){\rm exp}(\frac{c}{2}){\rm ch}(\triangle(E),g^{\triangle(E)},d)]\}^{(4k+1)}\\
        =&-2^{\frac{N}{2}+k+6}\sum_{s=1}^{[\frac{k}{2}]}s2^{-6s}h_s.
        \end{split}
    \end{equation}
\end{thm}
\begin{cor}
  Let $M$ be a $(4k+1)$-dimensional ${\rm spin}^c$ manifold. If $c_3(E,g,d)=0$ and $p_1(L)-p_1(M)=0$, then
  \begin{equation}
    \begin{split}
      {\rm Ind}(T^c&\otimes((\widetilde{T_\mathbf{C}M}-\widetilde{L_{\bf{R}}\otimes\mathbf{C}})
      \otimes(\triangle(E),g^{\triangle(E)},d)
    +(\triangle(E)\otimes\widetilde{E_\mathbf{C}},g,d)))\\
    &-24k{\rm Ind}(T^c\otimes(\triangle(E),g^{\triangle(E)},d))\\
=&2^{\frac{N}{2}+k+6}\sum_{s=1}^{[\frac{k}{2}]}s2^{-6s}h_s.
    \end{split}
  \end{equation}
\end{cor}
\begin{cor}
  Let $M$ be a $13$-dimensional ${\rm spin}^c$ manifold. If $c_3(E,g,d)=0$ and $p_1(L)-p_1(M)=0$, then
  \begin{equation}
        \begin{split}
        \{\widehat{A}&(TM,\nabla^{TM}){\rm exp}(\frac{c}{2})[{\rm ch}(\widetilde{T_\mathbf{C}M}-\widetilde{L_{\bf{R}}\otimes\mathbf{C}}){\rm ch}(\triangle(E),g^{\triangle(E)},d)\\
&+{\rm ch}(\triangle(E)\otimes\widetilde{E_\mathbf{C}},g,d)]
-72[\widehat{A}(TM,\nabla^{TM}){\rm exp}(\frac{c}{2}){\rm ch}(\triangle(E),g^{\triangle(E)},d)]\}^{(13)}\\
        =&-2^{\frac{N}{2}+3}\{\widehat{A}(TM,\nabla^{TM}){\rm exp}(\frac{c}{2}){\rm ch}(\widetilde{E_\mathbf{C}},g,d)\}^{(13)}.
        \end{split}
    \end{equation}
\end{cor}
\begin{cor}
  Let $M$ be a $17$-dimensional ${\rm spin}^c$ manifold. If $c_3(E,g,d)=0$ and $p_1(L)-p_1(M)=0$, then
  \begin{equation}
        \begin{split}
        \{\widehat{A}&(TM,\nabla^{TM}){\rm exp}(\frac{c}{2})[{\rm ch}(\widetilde{T_\mathbf{C}M}-\widetilde{L_{\bf{R}}\otimes\mathbf{C}}){\rm ch}(\triangle(E),g^{\triangle(E)},d)\\
&+{\rm ch}(\triangle(E)\otimes\widetilde{E_\mathbf{C}},g,d)]
-96[\widehat{A}(TM,\nabla^{TM}){\rm exp}(\frac{c}{2}){\rm ch}(\triangle(E),g^{\triangle(E)},d)]\}^{(17)}\\
        =&9\times2^{\frac{N}{2}+2}\{\widehat{A}(TM,\nabla^{TM}){\rm exp}(\frac{c}{2}){\rm ch}(\widetilde{E_\mathbf{C}},g,d)\}^{(17)}\\
        &-2^{\frac{N}{2}-1}\{\widehat{A}(TM,\nabla^{TM}){\rm exp}(\frac{c}{2}){\rm ch}(\wedge^2\widetilde{E_\mathbf{C}},g,d)\}^{(17)}.
        \end{split}
    \end{equation}
\end{cor}
\begin{thm}
 If $c_3(E,g,d)=0$ and $p_1(L)-p_1(M)=0$, then
    \begin{equation}
        \begin{split}
        \{\widehat{A}(&TM,\nabla^{TM}){\rm exp}(\frac{c}{2})[{\rm ch}((S^2\widetilde{T_\mathbf{C}M}+\widetilde{T_\mathbf{C}M}+\wedge^2\widetilde{L_{\bf{R}}
\otimes\mathbf{C}}\\
&-\widetilde{L_{\bf{R}}
\otimes\mathbf{C}}+\widetilde{T_\mathbf{C}M}\otimes\widetilde{L_{\bf{R}}
\otimes\mathbf{C}}){\rm ch}(\triangle(E),g^{\triangle(E)},d)\\
&+{\rm ch}(\widetilde{T_\mathbf{C}M}-\widetilde{L_{\bf{R}}
\otimes\mathbf{C}}){\rm ch}(\triangle(E)\otimes\widetilde{E_\mathbf{C}},g,d)\\
&+{\rm ch}(\triangle(E)\otimes(\wedge^2\widetilde{E_\mathbf{C}}
+\widetilde{E_\mathbf{C}}),g,d))]\\
&+(288k^2+72k)[\widehat{A}(TM,\nabla^{TM}){\rm exp}(\frac{c}{2}){\rm ch}(\triangle(E),g^{\triangle(E)},d)]\\
&-(24k-8)\{\widehat{A}(TM,\nabla^{TM}){\rm exp}(\frac{c}{2})[{\rm ch}(\widetilde{T_\mathbf{C}M}-\widetilde{L_{\bf{R}}\otimes\mathbf{C}}){\rm ch}(\triangle(E),g^{\triangle(E)},d)\\
&+{\rm ch}(\triangle(E)\otimes\widetilde{E_\mathbf{C}},g,d)]\}\}^{(4k-1)}\\
=&2^{\frac{N}{2}+k+11}\sum_{s=1}^{[\frac{k}{2}]}s^22^{-6s}h_s.
        \end{split}
    \end{equation}
\end{thm}
\begin{cor}
  Let $M$ be a $(4k+1)$-dimensional ${\rm spin}^c$ manifold. If $c_3(E,g,d)=0$ and $p_1(L)-p_1(M)=0$, then
  \begin{equation}
    \begin{split}
      {\rm Ind}(T^c&\otimes((S^2\widetilde{T_\mathbf{C}M}+\widetilde{T_\mathbf{C}M}
      +\wedge^2\widetilde{L_{\bf{R}}\otimes\mathbf{C}}\\
&-\widetilde{L_{\bf{R}}
\otimes\mathbf{C}}+\widetilde{T_\mathbf{C}M}\otimes\widetilde{L_{\bf{R}}
\otimes\mathbf{C}})\otimes(\triangle(E),g^{\triangle(E)},d)\\
&+(\widetilde{T_\mathbf{C}M}-\widetilde{L_{\bf{R}}
\otimes\mathbf{C}})\otimes(\triangle(E)\otimes\widetilde{E_\mathbf{C}},g,d)\\
&+(\triangle(E)\otimes(\wedge^2\widetilde{E_\mathbf{C}}
+\widetilde{E_\mathbf{C}}),g,d)))\\
&+(288k^2+72k){\rm Ind}(T^c\otimes(\triangle(E),g^{\triangle(E)},d))\\
&-(24k-8){\rm Ind}(T^c\otimes((\widetilde{T_\mathbf{C}M}-\widetilde{L_{\bf{R}}\otimes\mathbf{C}})\\
      &\otimes(\triangle(E),g^{\triangle(E)},d)
    +(\triangle(E)\otimes\widetilde{E_\mathbf{C}},g,d)))\\
=&-2^{\frac{N}{2}+k+11}\sum_{s=1}^{[\frac{k}{2}]}s^22^{-6s}h_s.
    \end{split}
  \end{equation}
\end{cor}
\begin{cor}
Let $M$ be a $13$-dimensional ${\rm spin}^c$ manifold. If $c_3(E,g,d)=0$ and $p_1(L)-p_1(M)=0$, then
 \begin{equation}
        \begin{split}
        \{\widehat{A}(&TM,\nabla^{TM}){\rm exp}(\frac{c}{2})[{\rm ch}((S^2\widetilde{T_\mathbf{C}M}+\widetilde{T_\mathbf{C}M}+\wedge^2\widetilde{L_{\bf{R}}
\otimes\mathbf{C}}\\
&-\widetilde{L_{\bf{R}}
\otimes\mathbf{C}}+\widetilde{T_\mathbf{C}M}\otimes\widetilde{L_{\bf{R}}
\otimes\mathbf{C}}){\rm ch}(\triangle(E),g^{\triangle(E)},d)\\
&+{\rm ch}(\widetilde{T_\mathbf{C}M}-\widetilde{L_{\bf{R}}
\otimes\mathbf{C}}){\rm ch}(\triangle(E)\otimes\widetilde{E_\mathbf{C}},g,d)\\
&+{\rm ch}(\triangle(E)\otimes(\wedge^2\widetilde{E_\mathbf{C}}
+\widetilde{E_\mathbf{C}}),g,d))]\\
&+2808[\widehat{A}(TM,\nabla^{TM}){\rm exp}(\frac{c}{2}){\rm ch}(\triangle(E),g^{\triangle(E)},d)]\\
&-64\{\widehat{A}(TM,\nabla^{TM}){\rm exp}(\frac{c}{2})[{\rm ch}(\widetilde{T_\mathbf{C}M}-\widetilde{L_{\bf{R}}\otimes\mathbf{C}}){\rm ch}(\triangle(E),g^{\triangle(E)},d)\\
&+{\rm ch}(\triangle(E)\otimes\widetilde{E_\mathbf{C}},g,d)]\}\}^{(13)}\\
=&2^{\frac{N}{2}+8}\{\widehat{A}(TM,\nabla^{TM}){\rm exp}(\frac{c}{2}){\rm ch}(\widetilde{E_\mathbf{C}},g,d)\}^{(13)}.
        \end{split}
    \end{equation}
\end{cor}
\begin{cor}
  Let $M$ be a $17$-dimensional ${\rm spin}^c$ manifold. If $c_3(E,g,d)=0$ and $p_1(L)-p_1(M)=0$, then
 \begin{equation}
        \begin{split}
        \{\widehat{A}(&TM,\nabla^{TM}){\rm exp}(\frac{c}{2})[{\rm ch}((S^2\widetilde{T_\mathbf{C}M}+\widetilde{T_\mathbf{C}M}+\wedge^2\widetilde{L_{\bf{R}}
\otimes\mathbf{C}}\\
&-\widetilde{L_{\bf{R}}
\otimes\mathbf{C}}+\widetilde{T_\mathbf{C}M}\otimes\widetilde{L_{\bf{R}}
\otimes\mathbf{C}}){\rm ch}(\triangle(E),g^{\triangle(E)},d)\\
&+{\rm ch}(\widetilde{T_\mathbf{C}M}-\widetilde{L_{\bf{R}}
\otimes\mathbf{C}}){\rm ch}(\triangle(E)\otimes\widetilde{E_\mathbf{C}},g,d)\\
&+{\rm ch}(\triangle(E)\otimes(\wedge^2\widetilde{E_\mathbf{C}}
+\widetilde{E_\mathbf{C}}),g,d))]\\
&+4896[\widehat{A}(TM,\nabla^{TM}){\rm exp}(\frac{c}{2}){\rm ch}(\triangle(E),g^{\triangle(E)},d)]\\
&-88\{\widehat{A}(TM,\nabla^{TM}){\rm exp}(\frac{c}{2})[{\rm ch}(\widetilde{T_\mathbf{C}M}-\widetilde{L_{\bf{R}}\otimes\mathbf{C}}){\rm ch}(\triangle(E),g^{\triangle(E)},d)\\
&+{\rm ch}(\triangle(E)\otimes\widetilde{E_\mathbf{C}},g,d)]\}\}^{(17)}\\
=&-7\times2^{\frac{N}{2}+8}\{\widehat{A}(TM,\nabla^{TM}){\rm exp}(\frac{c}{2}){\rm ch}(\widetilde{E_\mathbf{C}},g,d)\}^{(17)}\\
&+2^{\frac{N}{2}+5}\{\widehat{A}(TM,\nabla^{TM}){\rm exp}(\frac{c}{2}){\rm ch}(\wedge^2\widetilde{E_\mathbf{C}},g,d)\}^{(17)}.
\end{split}
\end{equation}
\end{cor}
\section{Acknowledgements}

 The second author was supported in part by NSFC No.11771070. The second author would like to thanks professors Fei Han and Jianqing Yu for helpful discussions. The authors also thank the referee for his (or her) careful reading and helpful comments.

\vskip 1 true cm

\section{Data availability}

No data was gathered for this article.

\section{Conflict of interest}

The authors have no relevant financial or non-financial interests to disclose.

\vskip 1 true cm

\bigskip
\bigskip
\indent{J. Guan}\\
 \indent{School of Mathematics and Statistics,
Northeast Normal University, Changchun Jilin, 130024, China }\\
\indent E-mail: {\it guanjy@nenu.edu.cn }\\
\indent{Y. Wang}\\
 \indent{School of Mathematics and Statistics,
Northeast Normal University, Changchun Jilin, 130024, China }\\
\indent E-mail: {\it wangy581@nenu.edu.cn }\\
\indent{H. Liu}\\
 \indent{School of Mathematics ,
Mudanjiang Normal University, Mudanjiang 157011, China }\\
\indent E-mail: {\it haiming0626@126.com }\\

\end{document}